\documentclass[sn-mathphys-num]{sn-jnl}

\usepackage{graphicx}%
\usepackage{multirow}%
\usepackage{amsmath,amssymb,amsfonts}%
\usepackage{nicefrac}%
\usepackage{amsthm}%
\usepackage{mathrsfs}%
\usepackage[title]{appendix}%
\usepackage{xcolor}%
\usepackage{textcomp}%
\usepackage{manyfoot}%
\usepackage{booktabs}%
\usepackage{algorithm}%
\usepackage{algorithmicx}%
\usepackage{algpseudocode}%
\usepackage{listings}%
\usepackage{subfig}%
\usepackage{matlab-prettifier}%
\usepackage{comment}

\theoremstyle{thmstyleone}%
\newtheorem{theorem}{Theorem}
\newtheorem{proposition}[theorem]{Proposition}%
\newtheorem{lemma}{Lemma}

\theoremstyle{thmstyletwo}%
\newtheorem{remark}{Remark}%

\theoremstyle{thmstylethree}%

\usepackage{amsopn}

\newcommand{\R}{\ensuremath{\mathbb{R}}}

\usepackage{tikz}
\usetikzlibrary{fit,positioning,decorations.pathreplacing,decorations.pathmorphing,external}
\usepackage{pgfplots}
\pgfplotsset{compat=1.18}

\definecolor{OliveGreen}{HTML}{3C8031}
\definecolor{ao(english)}{rgb}{0.0, 0.5, 0.0}
\begin{document}

\title[Optimal Polynomial Smoothers for Parallel AMG]{Optimal Polynomial Smoothers for Parallel AMG}


%
%
%
%
%

\author[1]{\fnm{Pasqua} \sur{D'Ambra}}\email{pasqua.dambra@cnr.it}
\equalcont{These authors contributed equally to this work.}

\author*[1,2]{\fnm{Fabio} \sur{Durastante}}\email{fabio.durastante@unipi.it}
\equalcont{These authors contributed equally to this work.}

\author[1,3]{\fnm{Salvatore} \sur{Filippone}}\email{salvatore.filippone@uniroma2.it}
\equalcont{These authors contributed equally to this work.}

\author[2]{\fnm{Stefano} \sur{Massei}}\email{stefano.massei@unipi.it}
\equalcont{These authors contributed equally to this work.}

\author[4]{\fnm{Stephen} \sur{Thomas}}\email{stephen.thomas2@amd.com}
\equalcont{These authors contributed equally to this work.}

\affil[1]{\orgdiv{Institute for Applied Computing ``Mauro Picone''}, \orgname{National Research Council of Italy}, \orgaddress{\street{Via Pietro Castellino, 111}, \city{Naples}, \postcode{80131}, \country{Italy}}}

\affil*[2]{\orgdiv{Department of Mathematics}, \orgname{University of Pisa}, \orgaddress{\street{Largo Bruno Pontecorvo, 5}, \city{Pisa}, \postcode{56127}, \country{Italy}}}

\affil[3]{\orgdiv{Department of Civil and Computer Engineering}, \orgname{University of Rome ``Tor Vergata''}, \orgaddress{\street{Via del Politecnico, 1}, \city{Rome}, \postcode{00133}, \country{Italy}}}

\affil[4]{\orgdiv{HPC and DC-GPU}, \orgname{Advanced Micro Devices}, \orgaddress{\city{Austin}, \state{TX}, \country{USA}}}


\abstract{In this paper, we explore polynomial accelerators that are well-suited for parallel computations, specifically as smoothers in Algebraic MultiGrid (AMG) preconditioners {for symmetric positive definite matrices}. These accelerators address a minimax problem, initially formulated in~\cite[Lottes, Numer. Lin. Alg. with Appl. 30(6), 2518 (2023)]{LottesNLAA2023}, aiming to achieve an optimal (or near-optimal) bound for a polynomial-dependent constant involved in the AMG V-cycle error bound, without requiring information about the matrices' spectra. Lottes focuses on Chebyshev polynomials of the 4\textsuperscript{th}-kind and defines the relevant recurrence formulas applicable to a general convergent basic smoother. In this paper, we demonstrate the efficacy of these accelerations for large-scale applications on modern GPU-accelerated supercomputers. Furthermore, we formulate a variant of the aforementioned minimax problem, which naturally leads to solutions relying on Chebyshev polynomials of the  1\textsuperscript{st}-kind as accelerators for a basic smoother. For all the polynomial accelerations, we describe efficient GPU kernels for their application and demonstrate their comparable effectiveness on standard benchmarks at very large scales.}




\maketitle

\section{Introduction}\label{sec:introduction}

Algebraic MultiGrid (AMG) methods are optimal iterative solvers for a large class of linear systems, particularly effective for those arising from scalar elliptic isotropic Partial Differential Equations (PDEs), achieving near-constant iteration counts regardless of the problem size~\cite{MR2427040,MR3653855}. This property, known as \emph{algorithmic scalability}, combined with optimal linear-time complexity, makes AMG methods highly suitable for the exascale computing challenge~\cite{MR4072454}. 

Gauss-Seidel has traditionally been the preferred smoother for AMG, often implemented in a parallel hybrid variant using block-Jacobi iterations~\cite{MR2861652}. However, this approach may degrade smoothing properties and prove inefficient on modern multi/many-core processors. Consequently, research focuses on finding efficient smoothers for such environments~\cite{MR1985310,MR2861652,MR1885607,doi:10.1137/21M1430030,FLAIG2011846,10.1007/978-3-642-24025-6_18}{, and has also been extended to address reduction-based algebraic multigrid methods, in which the smoothers operate exclusively on the fine-grid points~\cite{MR2259545,MR3498145}, and to cover the use of accelerators like FPGAs~\cite{9835299}.} 
Polynomial smoothers offer an alternative, leveraging Sparse Matrix-Vector (SpMV)
products, which are well-optimized on contemporary hardware. This paper
deals with the usage of  polynomial smoothers within AMG to solve symmetric
positive definite (SPD) linear systems:
\begin{equation}
	A\mathbf{x}=\mathbf{b}, \ \ \text{with} \  A \in \R^{n \times n} \ \text{SPD and} \ \ \mathbf{x}, \mathbf{b} \in \R^n.
	\label{eq:sys1}
\end{equation}
The selection of polynomial smoothers, their degrees, and related parameters,
significantly impact AMG performance. Lottes~\cite{LottesNLAA2023} recently
provided convergence analysis for AMG's two-level and $V$-cycle formulations,
where error bounds depend on polynomial acceleration. This allows for
strategies optimizing the convergence constant for the $V$-cycle, a method
previously used in geometric multigrid for unsteady flow
problems~\cite{MR2787943, MR3695922,MR2995785,MR3254225} { and high-order continuous finite element discretizations~\cite{MR3367828,doi:10.1177/10943420231217221}}. Our focus is on
Chebyshev-type polynomial smoothers~\cite{MR1985310,LottesNLAA2023}, which
optimize the two-level energy-norm error propagation matrix. Additionally, we
aim to develop smoothers and optimization procedures {for SPD linear systems} independent of the
matrix spectrum, unlike those discussed
in~\cite{MR2861652,MR2896807,KVZ2012,MR3083519,MR3395379}.

\paragraph{Contribution} This paper extends Lottes' results by introducing a variant of the minimax problem to optimize the V-cycle error bound and employing Chebyshev polynomials of the 1\textsuperscript{st}-kind Chebyshev polynomials on a suitable interval $[a^*, 1]$. The left endpoint, $a^*$, is determined by solving a nonlinear equation that varies with the polynomial degree but is independent of the matrix's size. These polynomials achieve lower $V$-cycle error bounds for low-degree polynomials ($k \leq 5$) compared to Lottes' 4\textsuperscript{th}-kind Chebyshev polynomials and are comparable to the numerically-approximated optimal bounds~\cite{LottesNLAA2023}. We have developed efficient NVIDIA GPU kernels for all these polynomial accelerators within the Parallel Sparse Computation Toolkit (\texttt{PSCToolkit}) framework~\cite{softimpact}, demonstrating their comparable effectiveness on large-scale benchmarks.

\paragraph{Synopsis} The paper is organized as follows: Section~\ref{sec:method} outlines the construction of polynomial smoothers and our proposal, Section~\ref{sec:amgprec} details the AMG method construction using \texttt{AMG4PSBLAS}~\cite{MR4331965} from the \texttt{PSCToolkit} suite~\cite{softimpact}. Section~\ref{sec:results} discusses GPU implementation details and benchmark tests. Finally, Section~\ref{sec:conclusion} concludes and suggests future research directions.

\section{Polynomial Smoothers for AMG}\label{sec:method}
We start by briefly recalling the expression for a stationary iteration for the SPD $A$ in~\eqref{eq:sys1} 
\[
\mathbf{x}^{(k)}=\mathbf{x}^{(k-1)}+B\left(\mathbf{b}- A\mathbf{x}^{(k-1)}\right), \ \ k=1, 2, \ldots \ \ \text{given} \ \textbf{x}^{(0)} \in \mathbb{R}^n.
\]
The above equation represents a generic AMG method  
when the matrix $B$ is defined recursively by composing a coarse grid
correction together with a simple fixed point iteration. Specifically,  let
$A_l$ be a sequence of coarse matrices 
computed by the triple-matrix Galerkin product:
\[
A_{l+1}=P_l^TA_lP_l, \ \ l=0, \ldots, N_l-1,
\]
with $A_0=A$ and $P_l$ a sequence of prolongation matrices of size $n_l \times n_{l+1}$, with $n_{l+1} < n_l$
and $n_0=n$. Then we can take any iteration matrix $M_l$ that is an $A_l$-convergent smoother 
for $A_l$, i.e., $\|I-M_l^{-1}A_l\|_{A_l} < 1$, where $I$ is the identity matrix of size $n_l$ and 
$\| \cdot \|_{A_l}$ indicates the $A_l$-energy-norm, and construct $B$ as the multiplicative composition 
of the following error propagation matrices:
\begin{equation}
	\label{eq:amg}
	I-B_lA_l= \left(I-M_l^{-T}A_l\right)^{k}\left(I-P_lB_{l+1}P_l^TA_l\right)\left(I-M_l^{-1}A_l\right)^{k} \ \ \forall\, l< N_l, k \geq 1,
\end{equation}
where we are assuming that the coarsest $A_l$ is solved by a direct method. 

Rather than employing the smoothing iteration: 
\[
G = \left(I-M_l^{-1}A_l\right)^{k}, \quad k \geq 1,
\]
we consider polynomial iterations of the form:
\[
G = p_k(M_l^{-1}A_l), \qquad p_k(x) \in \Pi_{k}, 
\]
where the set of polynomials $\Pi_k$ is defined as:
\begin{equation}\label{eq:polspace}
	\Pi_k=\{p \,\text{polynomial of degree at most }k:\ p(0)=1,\quad |p(x)|\le 1\ \ \forall x\in[0, 1] \}.
\end{equation}
If we restrict the analysis to the case in which we employ only two levels, we can drop the $l$ index, and
let $P \in \R^{n \times n_c}$ be the prolongator from the fine to the coarse space and 
$A_c=P^TAP$ be the Galerkin coarse matrix. Then the error propagation matrix for the two-level 
symmetric $V$-cycle, when $k$ iterations of the smoother are applied before and after the 
coarse-space correction in~\eqref{eq:amg} can be compactly rewritten as:
\[
E=G^T(I-PA_c^{-1}P^TA)G.
\]
The following bound can be obtained:
\begin{equation}
	\label{boundTLE}
	\|E\|^2_A \leq \frac{C}{C+2k},
\end{equation}
where $C$ is the following approximation property constant:
\begin{equation}
	\label{approxC}
	C=\|A^{-1}-PA_c^{-1}P^T\|_{A,B}^2 = \sup_{\|u\|_B \leq 1} \|(A^{-1}-PA_c^{-1}P^T)u\|^2_A,
\end{equation}
that is an approximation property of the vectors in $\R^{n \times n}$ by
vectors from the subspace~$\R^{n_c \times n_c}$; see, e.g., \cite[Theorem~3.4]{MR0795945}. If we now go back to considering all the
levels, it follows that the resulting convergence estimate for~\eqref{eq:amg}
is given by the largest of~\eqref{boundTLE}
at the various levels~\cite[Theorem~4.5]{MR0685774}. The previous analysis
can be refined for the case of a generic polynomial smoother as follows.
\begin{proposition}[{\cite[Lemma~2]{LottesNLAA2023}}]\label{pro:lotte-proposition}
	Let the smoother iteration (on each level) be represented~by:
	\[
	G_l=p^l(M^{-1}A),
	\]
	where $p^l(x)$ is a polynomial satisfying $p^l(0)=1$ and $|p^l(x)|<1$ for $0 < x \leq 1$, possibly different 
	on each level. Then the $V$-cycle error propagation matrix satisfies the following bound:
	\begin{equation}
		\label{boundopt}
		\|E\|^2_A \leq \max_l \frac{C_l}{C_l+\gamma_l^{-1}},
	\end{equation}
	where $C_l$ is the approximation property constant in~\eqref{approxC} at the level $l$ and
	\[
	\gamma_l= \sup_{0 < \lambda \leq 1} \frac{\lambda p^l(\lambda)^2}{1-p^l(\lambda)^2}.
	\]
\end{proposition}
The analysis and construction of the smoothers that we deal with in this
paper  is therefore based on the search for polynomials attempting  to 
optimize the above convergence bound. 

Having observed this behavior on the levels and focusing on the
choice of the smoother, we now drop the level indication--the $l$
subscript--and generally assume that we are using the same smoother at all
levels of the $V$-cycle~\eqref{eq:amg}, i.e., we set $p^l(x) = p_k(x) \in
\Pi_k[x]$ for all $l$. Our main interest for these polynomial
smoothers {is that of defining methods whose application rely on the basic SpMV product as main operation and that does not require \emph{a-priori} knowledge of the spectral interval of $A$}, therefore, as basic method we focus on the $\ell_1-$Jacobi smoother:
\begin{equation}
	\label{l1smooth}
	G = M^{-1}N= I - M^{-1}A,\quad \begin{array}{c}
		M=\operatorname{diag}(M_{i})_{i=1, \ldots, n},  \\
		M_{i}=a_{ii}+\sum_{j \neq i} |a_{ij}|, 
	\end{array}\quad A = M - N.
\end{equation}
Polynomial acceleration of such methods is a classic topic~\cite{MR0145679}
about which we report again only briefly the idea. For any fixed-point
iterative method of the form
\[
\mathbf{x}^{(k+1)} = \mathbf{x}^{(k)} + M^{-1}(\mathbf{b} - A \mathbf{x}^{(k)})  = M^{-1}(N \mathbf{x}^{(k)} + \mathbf{b}) , \qquad k \geq 0, \qquad A = M - N,
\]
we have seen that the error can be expressed as
\[
\mathbf{e}^{(k+1)} = \mathbf{x}^{(k+1)} - \mathbf{x} = M^{-1}N ( \mathbf{x}^{(k)} - \mathbf{x} ) = M^{-1}N \mathbf{e}^{(k)}.
\]
We would like to construct an accelerated sequence based upon 
the linear combination
\begin{equation}\label{eq:acceleration}
	\mathbf{y}^{(k)} = \sum_{j=0}^{k} \nu_{j,k} \mathbf{x}^{(j)}, \qquad \sum_{j=0}^{k} \nu_{j,k} = 1,
\end{equation}
that is used in conjunction with the error equation; this would give us
\begin{equation*}
	\begin{split}
		\sum_{j=0}^{k} \nu_{j,k} ( \mathbf{x}^{(j)} - \mathbf{x} ) & =  \sum_{j=0}^{k} \nu_{j,k} (M^{-1} N)^j \mathbf{e}^{(0)} = \tilde p_k(M^{-1} N) \mathbf{e}^{(0)} \\ & = \tilde p_k(I - M^{-1} A) \mathbf{e}^{(0)} = p_k(M^{-1} A) \mathbf{e}^{(0)}.
	\end{split}
\end{equation*}
However this approach is not convenient on the implementation side, as it
would require storing and summing $k$ vectors for each application. 
The most efficient procedure takes advantage of the short-term recurrences of
the prescribed family of polynomials; see Section~\ref{sec:how-to-lottes}
and Section~\ref{sec:how-to-cheby}.

To optimize the bound in Proposition~\ref{pro:lotte-proposition} we then have
to solve the minimax problem
\begin{equation}\label{eq:prob1}
	\gamma_k := \min_{p(x)\in \Pi_k}\max_{x\in(0, 1]} \left|\frac{xp(x)^2}{1-p(x)^2}\right|
\end{equation}
with the set $\Pi_k$ as in~\eqref{eq:polspace}. {Assuming only that the multigrid method is applied to an SPD matrix, the optimization procedure can be carried out offline. The parameters describing the optimal polynomials for various degrees are precomputed and tabulated. Unlike classical Chebyshev acceleration, this approach does not require recomputing the parameters when the matrix changes~\cite{MR1985310}. In particular, there is no need to estimate the largest eigenvalue or determine a tolerance factor to identify high-frequency components; see~\cite[\S 4.1]{MR2861652} for a discussion regarding the choice of tolerances.}

This is indeed different from the  norm of the residual 
$\|\mathbf{r}\|_2$ or the energy norm of the  error~$\|\mathbf{e}\|_A$ that are optimized in the 
derivation of  (preconditioned) Krylov  methods for symmetric matrices $A$,
such as MINRES and the conjugate gradient methods.

In the remainder of the section we will focus on determining polynomials that
(approximately)  solve~\eqref{eq:prob1} and building the corresponding 
application routines.


%


\subsection{4\textsuperscript{th}-kind Chebyshev Polynomials}\label{sec:how-to-lottes}

In order to make a comparison, we describe the approach discussed
in~\cite{LottesNLAA2023} based upon the 4\textsuperscript{th}-kind Chebyshev
polynomials: 
\[
W_k(x) = \frac{\cos(k+\nicefrac{1}{2})\theta}{\cos(\nicefrac{\theta}{2})},\; n \geq 0,\; x = \cos(\theta),
\]
whose choice produces the following polynomial smoother, expressed as a
simple three-terms recurrence:
\begin{eqnarray}
	\label{polsmoother} 
 &\begin{array}{l}
	\text{Input: }  \mathbf{x}^{(0)},   \\
    \text{Initialize: }  \mathbf{z}^{(0)}  = \mathbf{0},  \quad
	\mathbf{r}^{(0)}  = \mathbf{b}-A \mathbf{x}^{(0)},   \\
 \end{array} \nonumber \\
     & \begin{cases} 
    \mathbf{z}^{(k)}  =   \frac{2k-3}{2k+1} \mathbf{z}^{(k-1)}+\frac{8k-4}{2k+1}\frac{1}{\rho(M^{-1}A)}M^{-1}\mathbf{r}^{(k-1)}\\
	\mathbf{x}^{(k)} =  \mathbf{x}^{(k-1)}+\mathbf{z}^{(k)}, \\ 
	\mathbf{r}^{(k)}  = \mathbf{r}^{(k-1)} - A \mathbf{z}^{(k)}. 
	\end{cases} k \geq 1
\end{eqnarray}
If $M$ is an {SPD} smoother, then using
Proposition~\ref{pro:lotte-proposition}, the polynomial
smoother~\eqref{polsmoother} produces an error propagation matrix of the
symmetric multilevel $V$-cycle~\eqref{eq:amg} obeying the following bound:
\begin{equation}\label{eq:lottebound-4cheby}
	\|E\|^2_A \leq \max_{l = 1,\ldots,N_l-1} \frac{C_l}{C_l+ \nicefrac{4}{3}k_l(k_l+1)},
\end{equation}
where $C_l$ is the approximation property constant and $k_l$ is the
polynomial degree at the level $l$.

The other approach employed in~\cite{LottesNLAA2023} is to solve numerically
for the optimal polynomial $p(x)$ in~\eqref{eq:prob1} by applying Newton's
method to a system of nonlinear equations expressing the equi-oscillation
property, and then write:
\[
p_k(x) = \sum_{j=0}^{k} \frac{\beta_{j,k} - \beta_{j+1,k}}{2j+1} W_j(1 - 2x), \qquad \beta_{0,k} = 1, \quad \beta_{k+1,k} = 0 \qquad \forall k \geq 0.
\]
By this choice we can modify the update of the $\mathbf{x}^{(k)}$ vector
in~\eqref{polsmoother} with
\begin{eqnarray}
	\label{polsmootheropt}
	\mathbf{x}^{(k)} & = & \mathbf{x}^{(k-1)}+\beta_k \mathbf{z}^{(k)}, 
\end{eqnarray}
taking care to always compute the residual with respect to the original
solution, i.e., 
{the update of $\mathbf{r}^{k}$ is computed again as in~\eqref{polsmoother} using
the $\mathbf{z}^{(k)}$ vector via the relation $\mathbf{r}^{(k)}  = \mathbf{r}^{(k-1)} - A \mathbf{z}^{(k)}$,
see~\cite[{\S 4, P. 10}]{LottesNLAA2023} for further comments.}
In Fig.~\ref{fig:lotte_polynomials} we report on the two different choices of
polynomials. The $\beta_k$ coefficients are obtained by the Gaussian
quadrature rule applied to the orthogonality relation for the
4\textsuperscript{th}-kind Chebyshev polynomials $W_j$.
\begin{figure}[htbp]
	\centering
	\subfloat[Scaled and shifted 4\textsuperscript{th}-kind Chebyshev polynomials]{\input{4thkindcheb}}\hspace{1em}
	\subfloat[Optimized polynomial from~\cite{LottesNLAA2023} corresponding to the update rule in~\eqref{polsmootheropt}]{\input{lotteoptimal}}
	\caption{Polynomials for accelerating the fixed point iterative method to
		use as a smoother from~\cite{LottesNLAA2023}. The left panel contains the
		standard 4\textsuperscript{th}-kind Chebyshev polynomials, the right
		panel has the polynomials implicitly used in the iteration~\eqref{polsmootheropt}, where the coefficients $\beta_k$ are obtained by {numerically solving the optimization problem for the polynomial $p^l$ in Proposition~\ref{pro:lotte-proposition} with an alternating minimization approach described in \cite[Appendix B]{LottesNLAA2023}.}}
	\label{fig:lotte_polynomials}
\end{figure}

We observe that application of the above smoother, for a polynomial degree
$k$, requires $k$ applications of the basic smoother to the residual vector,
two vector operations (\texttt{axpy}) for updating vectors $z$ and $x$, and
one sparse matrix-vector product (\texttt{SpMV}) plus another \texttt{axpy}
for updating the residual. Therefore, the smoother application requires $k$
\texttt{axpy} operations more than $k$ iterations of the basic smoother; we
will deal with the efficient GPU implementation of the procedure in
Section~\ref{sec:amgprec}.

\subsection{Accelerated Sequence with Optimal 1\textsuperscript{st}-kind Chebyshev Polynomials}\label{sec:how-to-cheby}
We describe here the procedure for applying the 1\textsuperscript{st}-kind
Chebyshev polynomial acceleration that we propose as smoother. Let us recall
that  the Chebyshev polynomials of the 1\textsuperscript{st}-kind are the
family of orthogonal polynomials on the interval $[-1, 1]$ with respect to
the inner product $\langle h,g\rangle=\int_{-1}^{1}h(x)g(x)\frac{dx}{\sqrt{1-x^2}}$.
Let $\tau_k(x)$ denote the polynomial of degree $k$; we have $\tau_0(x)=1$,
$\tau_1(x)=2x$, and higher degree polynomials can be computed via the
recurrence relation $\tau_{k+1}(x)=2x\cdot \tau_k(x)-\tau_{k-1}(x)$.

Chebyshev polynomials can be generalized  to an arbitrary compact real
interval $[a,b]\subset \mathbb R$, in the sense of constructing the
polynomials $\tau_k^{[a,b]}(x)$ that are monic and orthogonal  with respect
to the inner product $\langle
h,g\rangle=\int_{a}^{b}h(x)g(x)\frac{dx}{\sqrt{(a-x)(b-x)}}$. This can be
done by composing $\tau_k$ with a linear map from $[a,b]$ to $[-1, 1]$ and by
applying a scaling:
\[ 
\tau_k^{[a,b]}(x)=\tau_k\left(\frac{b+a}{b-a}-\frac{2x}{b-a}\right) \Big{/}
\tau_k\left(\frac{b+a}{b-a}\right). 
\] 
We remark that the scaling is not necessary for the orthogonality but it
ensures the normalization property $\tau_k^{[a, b]}(0)=1$. The corresponding
recurrence relation is then given by
\begin{equation}\label{eq:recurrence-relation}
	\begin{array}{ll}
		\tau_0^{[a,b]}(x) & = 1,\\
		\tau_1^{[a,b]}(x) & = 1-\frac{2x}{b+a}, \\
		\tau_{k+1}^{[a,b]}(x)& = \frac{\tau_k(\frac{b+a}{b-a})}{\tau_{k+1}(\frac{b+a}{b-a})}\left[2 \left(\frac{b+a}{b-a}-\frac{2x}{b-a}\right)\,\tau_{k}^{[a,b]}(x) - \frac{\tau_{k-1}(\frac{b+a}{b-a})}{\tau_{k}(\frac{b+a}{b-a})}\tau_{k-1}^{[a,b]}(x)\right].   
	\end{array}
\end{equation}
Since we plan to optimize over the $a$ parameter, having $b$ set to $1$, we pose
\begin{equation}\label{eq:optimal_parameters}
	\theta = \frac{1 + a}{2}, \quad \delta = \frac{1 - a}{2},
\end{equation}
and rewrite~\eqref{eq:recurrence-relation} over a centered interval as 
\begin{equation}\label{eq:recurrence-relation-shifted}
	\begin{array}{ll}
		\tau_0^{[a,1]} (x)= &\; 1, \\
		\tau_1^{[a,1]}(x) = &\; 1 - \frac{x}{\theta}, \\
		\tau_{k+1}^{[a,1]}(x) = &\;  \frac{\sigma_k}{\sigma_{k+1}} \left[ 2 \frac{\theta - x}{\delta} \tau_{k}^{[a,1]}(x) - \frac{\sigma_{k-1}}{\sigma_{k}} \tau_{k-1}^{[a,1]}(x)  \right], \qquad k \geq 1. 
	\end{array}
\end{equation}
with the normalization constants given by
\begingroup
\allowdisplaybreaks
\begin{eqnarray*}
	\sigma_{0} & = &\; 1,\\
	\sigma_{1} & = &\; \frac{\theta}{\delta},\\
	\sigma_{k+1} & = &\; 2 \frac{\theta}{\delta} \sigma_k - \sigma_{k-1}, \qquad k \geq 1.
\end{eqnarray*}
\endgroup
The Chebyshev acceleration for the fixed point iteration is then summarized in 
the following recurrence:
\begingroup
\allowdisplaybreaks
\begin{eqnarray}
	\label{polsmootheropt1kind}
	\mathbf{r}^{(0)} & = & \frac{1}{\rho(M^{-1}A)}M^{-1}(\mathbf{b}-A\mathbf{x}^{(0)}) \nonumber \\
	\mathbf{z}^{(0)}  & = & \frac{2}{1+a}\mathbf{r}^{(0)}, \ \rho_0=\frac{1-a}{1+a} \nonumber\\
	\rho_k &=& \left( \frac{2(1+a)}{1-a}-\rho_{k-1} \right)^{-1} \nonumber \\
	\mathbf{x}^{(k)} & = & \mathbf{x}^{(k-1)}+ \mathbf{z}^{(k-1)} \\ 
	\mathbf{r}^{(k)} & = & \mathbf{r}^{(k-1)}- \frac{1}{\rho(M^{-1}A)}M^{-1}A\mathbf{z}^{(k-1)}  \nonumber \\ 
	\mathbf{z}^{(k)} &=&  \rho_k \rho_{k-1} \mathbf{z}^{(k-1)}+\frac{4\rho_k}{(1-a)}\mathbf{r}^{(k)} \nonumber 
\end{eqnarray}
\endgroup
The above acceleration guarantees that the error at step $k$ behaves as 
\[\mathbf{e}^{(k)} = \mathbf{x}^{(k)} - \mathbf{x} = \tau_{k}^{[a,1]}(\rho^{-1}(M^{-1}A)\:M^{-1}A) \mathbf{e}^{(0)},\] 
for any value of the $a$ parameter. This is the analogue of the
formulas~\eqref{polsmootheropt} for the polynomial acceleration based on the
Chebyshev polynomials of the 4\textsuperscript{th}-kind for the smoother
in~\eqref{l1smooth}. 

We observe that the application of the above smoother, for a polynomial of
degree $k$, has the same cost of the smoother based on the 
4\textsuperscript{th}-kind Chebyshev polynomial, as it requires $k$
applications of the basic smoother, $k$ \texttt{SpMV} for updating the
residual vector, and two vector operations (\texttt{axpy}) for updating
$d$ and $x$; again, we postpone the discussion of an  efficient GPU
implementation of  the procedure to Section~\ref{sec:amgprec}. In the next
Section~\ref{sec:fixing-a}, we will address the choice of $a$ according to
the convergence analysis outlined in Proposition~\ref{pro:lotte-proposition}.

\subsection{Revisiting the Polynomial Optimization Problem}\label{sec:fixing-a}
Herein, we propose an alternative family of polynomials which provides a quasi-optimal solution for the 
minimax problem in~\eqref{boundopt}. We start by rewriting problem~\eqref{eq:prob1}~as
$$
\gamma_k = \min_{p(x)\in \Pi_k}\max_{x\in(0, 1]} x\left|1-\frac{1}{1-p(x)^2}\right|,
$$
i.e., we can consider this as a search for the most favorable minimax approximation 
relative to an unconventional weight, targeting the zero function across the 
interval \( (0, 1] \) over the set of monic polynomials of degree bounded by $k$.
The latter constraint is indeed what makes the issue non trivial as each candidate 
solution has to pass through $1$ at $x=0$. 

The core idea of this section is to analyze the values of the objective function when restricting to the set of scaled 1\textsuperscript{st}-kind Chebyshev polynomials. The motivation for this is that those polynomials would provide the optimal solution whenever we replace the weight function with $1$ and the interval $(0,1]$ with $[a,1]$, for any $a\in(0,1)$.

\subsubsection{1\textsuperscript{st}-kind Chebyshev Polynomials: Basic Properties}

We report here some well known results about Chebyshev polynomials that will be needed in the subsequent analysis. We
refer to~\cite[Chapter~1]{MR1937591} for the basic properties and functional identities, while we point to~\cite[Chapter~3]{MR1937591} 
for the minimax approximation properties.

\begin{itemize}
	\item[$(i)$] For $x\in\mathbb R$, $\tau_k(x)$ verifies the following identity
	$$
	\tau_k(x) = \frac{1}{2}\left[(x+\sqrt{x^2-1})^k + (x-\sqrt{x^2-1})^k\right].$$
	\item[$(ii)$]The polynomial $\tau_k^{[a, b]}(x)$ is the monic polynomial of degree at most $k$ that deviates least from $0$ on $[a,b]$. More precisely, $\tau_k^{[a, b]}(x)$ solves the optimization problem
	$$
	\tau_k^{[a,b]}(x)=\arg\min_{\{\mathrm{deg}(p)=k,\ p(0)=1\}} \max_{x\in[a, b]} |p(x)|,
	$$
	and the associated optimal value verifies
	$$
	\max_{x\in[a, b]} |\tau_k^{[a,b]}(x)| = 2 \left[\left(\frac{\sqrt{\kappa}-1}{\sqrt{\kappa}+1}\right)^k+ \left(\frac{\sqrt{\kappa}+1}{\sqrt{\kappa}-1}\right)^k\right]^{-1} \leq 2 \left(\frac{\sqrt{\kappa}-1}{\sqrt{\kappa}+1}\right)^{k},
	$$
	where $\kappa=\frac ba$.
	\item[$(iii)$] The zeros of $\tau_k^{[a,b]}$, and of $\frac{d}{dx}(\tau_k^{[a,b]})$ are inside $[a,b]$. Combining this with $\tau_k^{[a, b]}(0)=1$, and $|\tau_k^{[a,b]}|\leq 1$ in $[a,b]$, implies that $\tau_k^{[a, b]}$ is decreasing in $[0, a]$. In particular, when $b=1$ we have $\tau_k^{[a,1]}\in\Pi_k$.
\end{itemize}
\subsubsection{Restricting the Problem to the 1\textsuperscript{st}-kind Chebyshev Polynomials of Degree \texorpdfstring{$k$}{k}} 
To get a computable upper bound for the optimal value
$\gamma_k$ of \eqref{eq:prob1}, we restrict the set  
of feasible polynomials to $\mathcal C_k:=\{\tau_k^{[a, 1]},\ a\in [0, 1) \}\subset \Pi_k$ and we consider
\begin{equation}\label{eq:lambdak}
	\Lambda_k:=\min_{p(x)\in \mathcal C_k}\max_{x\in(0, 1]} x\left|1-\frac{1}{1-p(x)^2}\right|= \min_{a\in [0,1)}\max_{x\in(0, 1]} x\left|1-\frac{1}{1-\tau_k^{[a,1]}(x)^2}\right|.
\end{equation}
Since the Chebyshev polynomial $\tau_k^{[a,1]}(x)$ has very different behaviours inside and outside the interval $[a,1]$ it is convenient to divide the analysis of the minimax problem in two parts: one related to the half open interval $(0,a]$ and the other to the closed interval $[a,1]$. Formally, we use that
\begin{align}\label{eq:max2}\sup_{x\in (0, 1]}\hspace{-0.5em} x \left|1-\frac{1}{1-\tau_k^{[a,1]}(x)^2}\right|
	\hspace{-0.1em}=\hspace{-0.1em} 
	\max\hspace{-0.2em}\left\{\sup_{x\in (0, a]} \frac{x\tau_k^{[a, 1]}(x)^2}{1-\tau_k^{[a, 1]}(x)^2}, \max_{x\in [a, 1]} \frac{x\tau_k^{[a, 1]}(x)^2}{1-\tau_k^{[a, 1]}(x)^2}\right\}\hspace{-0.3em}.
\end{align}

The first step is showing that, whenever we fix a polynomial in $\mathcal C_k$, the objective function in the max problem is positive and monotonically decreasing on $(0,a]$; since the proof of this is rather technical, it is moved to Appendix~\ref{appendix}.
\begin{lemma}\label{lem:monotonicity}
	The function $f(x):=\frac{x\tau_k^{[a, 1]}(x)^2}{1-\tau_k^{[a, 1]}(x)^2}$ is positive and monotonically decreasing on $(0,a]$ for any $k\in\mathbb N$, and for any $a\in(0, 1)$.
\end{lemma}

Let $c_j^{(a,k)}$ denote the coefficient of degree $j$ of $\tau_k^{[a,1]}$; in view of Lemma~\ref{lem:monotonicity}, we have	
\begin{align*}	
	\sup_{x\in (0, a]} \frac{x\tau_k^{[a, 1]}(x)^2}{1-\tau_k^{[a, 1]}(x)^2}&= \lim_{x\to0^+} \frac{x\tau_k^{[a, 1]}(x)^2}{1-\tau_k^{[a, 1]}(x)^2}= \frac{1}{2|c_1^{(a,k)}|}, 
\end{align*}
meaning that the maximum error in $(0, a]$ corresponding to $\tau_k^{[a,b]}$ is  $\frac{1}{2|c_1^{(a,k)}|}$ and is equal to  the absolute value of the reciprocal of the coefficient of degree one given by $\tau_k^{[a,1]}(x)^2$.
Let us provide an explicit formula for $c_1^{(a,k)}$; we have that
$
c_1^{(a,k)}= [\tau_k^{[a,1]}]'(0)
$,
and in view of the relation $\tau_k'(x)=k q_{k-1}(x)$, with $q_j(x)$ indicating the $j$th Chebyshev polynomial of the 2\textsuperscript{nd}-kind, it follows
\begin{equation}\label{eq:c1}
	c_1^{(a,k)}=-\frac{2k}{1-a} q_{k-1}\left(\frac{1+a}{1-a}\right)\Big{/}\tau_k\left(\frac{1+a}{1-a}\right).
\end{equation}
For a real argument $x$, Chebyshev polynomials of the 2\textsuperscript{nd}-kind can be expressed with the formula
$$
q_k(x) = \frac{(x+\sqrt{x^2-1})^{k+1} - (x-\sqrt{x^2-1})^{k+1}}{2\sqrt{x^2-1}}.
$$
By substituting the formula for Chebyshev coefficients (both first and second kind) into \eqref{eq:c1}, and letting $y=\frac{1+a}{1-a}$ we get
\begingroup
\allowdisplaybreaks
\begin{align*}
	c_1^{(a,k)}&=-\frac{2k}{1-a}\cdot \frac{(y+\sqrt{y^2-1})^{k} - (y-\sqrt{y^2-1})^{k}}{(y+\sqrt{y^2-1})^{k} + (y-\sqrt{y^2-1})^{k}}\cdot\frac{1}{\sqrt{y^2-1}}\\
	&=-\frac{2k}{1-a}\cdot \left[1-\frac{ 2(y-\sqrt{y^2-1})^{k}}{(y+\sqrt{y^2-1})^{k} + (y-\sqrt{y^2-1})^{k}}\right]\cdot\frac{1}{\sqrt{y^2-1}}\\
	&=-\frac{2k}{1-a}\cdot\left[1-\frac{ 2}{(y+\sqrt{y^2-1})^{2k} + 1}\right]\cdot\frac{1}{\sqrt{y^2-1}}\\
	&=-\frac{2k}{1-a}\cdot\left[1-\frac{ 2}{(\frac{(1+\sqrt{a})^2}{1-a})^{2k} + 1}\right]\cdot \frac{1-a}{2\sqrt{a}}\\
	&=-\frac{k}{\sqrt{a}}\left[1-\frac{2(1-a)^{2k}}{(1+\sqrt{a})^{4k}+ (1-a)^{2k}}\right]\\
	&= -\frac{k}{\sqrt{a}}\cdot \frac{(1+\sqrt{a})^{2k}-(1-\sqrt{a})^{2k}}{(1+\sqrt{a})^{2k}+(1-\sqrt{a})^{2k}}\\
	&=-k\cdot \frac{\sum_{j=0}^{k-1} \binom{2k}{2j+1}a^{j}}{\sum_{j=0}^{k}\binom{2k}{2j}a^{j}}\\
	&\Rightarrow \frac{1}{2|c_1^{(a, k)}|}=\frac{\displaystyle \sum_{j=0}^{k}\binom{2k}{2j}a^{j}}{\displaystyle 2k\sum_{j=0}^{k-1} \binom{2k}{2j+1}a^{j}}.
\end{align*}
\endgroup
The next natural question is about the dependency of 
$\nicefrac{1}{2|c_1^{(a,k)}|}$ on the parameter $a$; the following lemma, 
proven in Appendix~\ref{appendix}, ensures that such a value is monotonically 
increasing with respect to $a$. 

\begin{lemma}\label{lem:interlacing}
	The function \[g(x):=\frac{\displaystyle \sum_{j=0}^{k}\binom{2k}{2j}x^{j}}{\displaystyle 2k\sum_{j=0}^{k-1} \binom{2k}{2j+1}x^{j}}\] is increasing on~$(0, 1)$.
\end{lemma}

Finally, we analyze the objective function in the interval $[a,1]$. Focusing on the expression \( \displaystyle \max_{x\in [a, 1]} x\left|1-\frac{1}{1-\tau_k^{[a,1]}
	(x)^2}\right| \), it becomes readily apparent that the maximum is achieved when 
\( x=1 \), at which point the value is given by:
\begingroup
\allowdisplaybreaks
\begin{align*}
	\frac{\tau_k^{[a,1]}(1)^2}{1-\tau_k^{[a,1]}(1)^2}&=\frac{1}{\tau_k\left(\frac{1+a}{1-a}\right)^2-1}\\
	&=\frac{1}{\frac 14 \left[ \left(\frac{\left(1+\sqrt{a}\right)^2}{1-a}\right)^k+\left(\frac{\left(1-\sqrt{a}\right)^2}{1-a}\right)^k\right]^2-1}\\
	&=\frac{4}{\left[ \left(\frac{1+\sqrt{a}}{1-\sqrt{a}}\right)^k+\left(\frac{1-\sqrt{a}}{1+\sqrt{a}})\right)^k\right]^2-4}.
\end{align*}	
\endgroup
Looking back at \eqref{eq:max2}, we see that, for a fixed $a$, we are taking the maximum between two quantities that depend monotonically on $a$. In particular, $\frac{\tau_k^{[a,1]}(a)^2}{1-\tau_k^{[a,1]}(a)^2}$  decreases from $+\infty$ to $0$ on $(0, 1)$ while $\frac{1}{2|c_1^{(a,k)}|}$ is non negative and monotonically increasing in view of Lemma~\ref{lem:interlacing}.

Therefore, there is only one value $a^*_k$ in $(0,1)$ of the parameter $a$  that satisfies the identity 
$\frac{1}{2|c_1^{(a,k)}|}= \frac{\tau_k^{[a,1]}(1)^2}{1-\tau_k^{[a,1]}(1)^2}$, 
i.e.:
$$
\frac{\sqrt{a}}{2k}\cdot \frac{(1+\sqrt{a})^{2k}+(1-\sqrt{a})^{2k}}{(1+\sqrt{a})^{2k}-(1-\sqrt{a})^{2k}}= \frac{4}{\left[ \left(\frac{1+\sqrt{a}}{1-\sqrt{a}}\right)^k+\left(\frac{1-\sqrt{a}}{1+\sqrt{a}})\right)^k\right]^2-4},
$$
and $a_k^*$ is the optimal choice for \eqref{eq:lambdak}; see Fig.~\ref{fig:minimax} in which we show the behaviour of the objective function for different values of $a$.
\begin{figure}[htbp]
	\definecolor{mycolor2}{rgb}{0.00000,1.00000,1.00000}%
	\centering
	\input{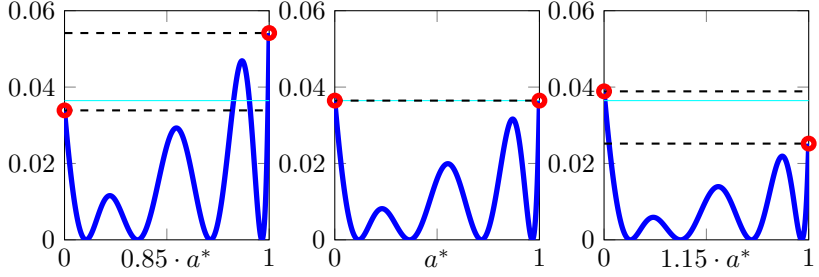}
	\caption{Graphs of the function $\nicefrac{x (\tau_{4}^{[a,1]}(x))^2}{1-(\tau_{4}^{[a,1]}(x))^2}$ for different choices of $a$, on the left panel $a = 0.85\cdot a^*$, on the center $a = a^*$, and on the right $a = 1.15 \cdot a^*$ with $a^*$ the optimal value. The solid cyan line ``\textcolor{mycolor2}{$-$}'' represents the optimal value $\Lambda_4$ of~\eqref{eq:lambdak}. }
	\label{fig:minimax}
\end{figure}

By applying some algebraic manipulations to the above equation, we find that $\sqrt{a_k^*}$ is the only root in $(0,1)$ of the polynomial equation
\begin{align}\label{eq:optimalparameter}
	8k (1 - x^2)^{2 k} + x \left[(1 - x)^{4 k} - (1 + x)^{4 k}\right]=0,
\end{align}
that can be approximated by using the combination of bisection, secant, and inverse quadratic interpolation methods available in MATLAB's \lstinline[style=Matlab-editor]{fzero} routine and based on the version of Brent's Algorithm from~\cite{MR0458783}.

In Figure~\ref{fig:apols} we report the first polynomials $\tau_k^{[a,1]}(x)$ for $a$ the optimal value obtained from the solution of~\eqref{eq:optimalparameter}, that can be compared with the polynomials from Fig.~\ref{fig:lotte_polynomials}.

\begin{figure}
	\centering
	\input{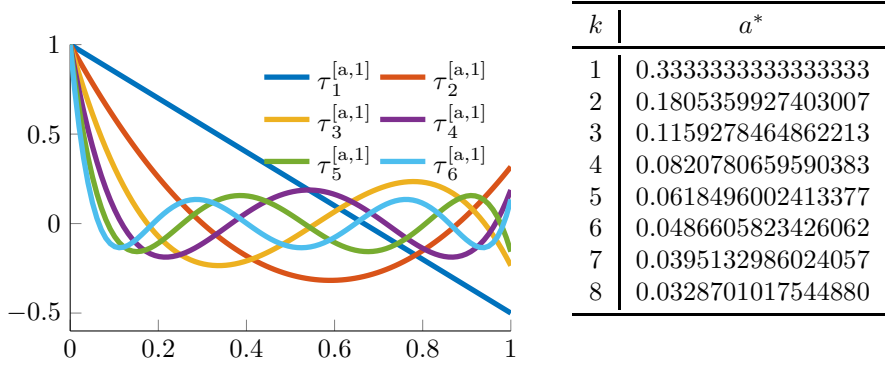}
	\caption{Optimized shifted and scaled Chebsyhev polynomials $\tau_k^{[a^*,1]}(x)$ with $a^*$ computed as the zero in $(0,1)$ of~\eqref{eq:optimalparameter}. This should be compared with the polynomials in Fig.~\ref{fig:lotte_polynomials}.}
	\label{fig:apols}
\end{figure}
\subsubsection{Bounds for the Optimal Parameters and Optimal Values}
In this section we  show that the sequence of optimal parameters $a_k^*$ approaches $0$ at least as fast as $\displaystyle \nicefrac{\log(k)^2}{k^2}$. By means of the latter property we can state  upper and lower bounds for the optimal values $\Lambda_k$.
\begin{theorem}\label{thm:convergence_constants}
	Let $\Lambda_k$ be the constant defined at \eqref{eq:lambdak} and let $a_k^*\in(0,1)$ be such that \[\displaystyle \max_{x\in(0,1]}x \left\lvert 1-\frac{1}{1-\tau_k^{[a_k^*,1]}(x)^2} \right\rvert = \Lambda_k.\] If $k\ge 3$, then
	$$
	\frac{\log(k)^2}{9k^2}\le a_k^*\le \frac{\log(k)^2}{k^2},\qquad \text{and}\qquad \frac{\log(k)}{6k^2}\le \Lambda_k\le 1.03\cdot \frac{\log(k)}{2k^2}.
	$$
\end{theorem}
\begin{proof}
	Let $\varphi(x):=8k (1 - x^2)^{2 k} + x \left[(1 - x)^{4 k} - (1 + x)^{4 k}\right]$ and note that the function $\varphi$ is negative on the right of its root $\sqrt{a_k^*}$ and positive on its left, i.e., for $x\in(0,1)$: $$ \begin{cases}\varphi(x)\le 0\quad\Rightarrow\quad x^2\ge a_k^*\\\varphi(x)\ge 0\quad\Rightarrow\quad x^2\le a_k^*\end{cases}.$$ To prove the first claim it is sufficient to show that for $k\ge 3$ we have $\varphi\left(\frac{\log(k)}{k}\right)\le 0$, and $\varphi\left(\frac{\log(k)}{3k}\right)\ge 0$. By means of some elementary computations  we obtain
	\begingroup
	\allowdisplaybreaks
	\begin{align*}\varphi\left(\frac{\log(k)}{k}\right)\le 0 \quad\Leftrightarrow\quad \left(\frac{k+\log(k)}{k-\log(k)}\right)^{2k}-\left(\frac{k-\log(k)}{k+\log(k)}\right)^{2k}\ge \frac{8k^2}{\log(k)}.
	\end{align*}
	\endgroup
	To get the above inequality we observe that
	\begingroup
	\allowdisplaybreaks
	\begin{align*}
		\left(\frac{k+\log(k)}{k-\log(k)}\right)^{2k}-\left(\frac{k-\log(k)}{k+\log(k)}\right)^{2k}&\ge\left(\frac{k+\log(k)}{k-\log(k)}\right)^{2k}-1\\ &=\left[\left(1+\frac{2\log(k)}{k-\log(k)}\right)^{\frac{k-\log(k)}{2\log(k)}}\right]^{\frac{4k\log(k)}{k-\log(k)}}-1\\
		&\ge 2^{\frac{4k\log(k)}{k-\log(k)}}-1\\
		&=k^{\frac{4k}{\log(2)(k-\log(k))}}-1\\
		&\ge k^4-1,
	\end{align*}
	\endgroup
	and that $k^4-1\ge 8k^2$--thus greater than $\frac{8k^2}{\log(k)}$ as well--for any $k\ge 3$.
	
	For the lower bound, we have that $\varphi\left(\frac{\log(k)}{3k}\right)\ge 0$ if and only if
	$$
	\left(\frac{3k+\log(k)}{3k-\log(k)}\right)^{2k}-\left(\frac{3k-\log(k)}{3k+\log(k)}\right)^{2k}\le \frac{24k^2}{\log(k)}.
	$$
	Then, observe that
	\begin{align*}
		\left(\frac{3k+\log(k)}{3k-\log(k)}\right)^{2k}-\left(\frac{3k-\log(k)}{3k+\log(k)}\right)^{2k}&\le    \left(\frac{3k+\log(k)}{3k-\log(k)}\right)^{2k}\\
		&\le \exp\left(\frac{4k\log(k)}{3k-\log(k)}\right)\\
		&\le k^{\frac{12}{9-\log(3)}},
	\end{align*}
	and that $k^{\frac{12}{9-\log(3)}}\le \frac{24k^2}{\log(k)}$ for any $k\ge 3$.
	
	In view of the argument used to determine $a_k^{*}$ we see that to get an upper bound for $\Lambda_k$ it is sufficient to compute an upper bound for $\theta(\frac{\log(k)}{k})$ with \[\theta(x):=\frac{x}{2k}\frac{(1+x)^{2k}+(1-x)^{2k}}{(1+x)^{2k}-(1-x)^{2k}}.\] Once again, some elementary computations yield
	\begin{align*}
		\theta\left(\frac{\log(k)}{k}\right)&=\frac{\log(k)}{2k^2}\cdot
		\frac{1+\left(1-\frac{2\log(k)}{k+\log(k)}\right)^{2k}}{1-\left(1-\frac{2\log(k)}{k+\log(k)}\right)^{2k}}.
	\end{align*}
	To get the inequality in the claim it is sufficient to note that the right factor in the above expression is a decreasing function of $k$, and for $k= 3$ it takes a value of approximately $1.0201$ that we can safely bound with $1.03$. 
	
	Finally, to get a lower bound for $\Lambda_k$ we compute a lower bound for $\theta\left(\frac{\log(k)}{3k}\right)$:
	\begin{equation*}
	\theta\left(\frac{\log(k)}{3k}\right)=\frac{\log(k)}{6k^2}\cdot
	\frac{1+\left(1-\frac{2\log(k)}{3k+\log(k)}\right)^{2k}}{1-\left(1-\frac{2\log(k)}{3k+\log(k)}\right)^{2k}}\ge \frac{\log(k)}{6k^2}.\qedhere
    \end{equation*}
\end{proof}

In Figure~\ref{fig:constantbounds} we show the bounds obtained in Theorem~\ref{thm:convergence_constants} for the values of $\{ a_k^* \}_k$, and $\{\Lambda_k\}_k$. In Figure~\ref{fig:boundcomparison} we show the computed values of  $\{\Lambda_k\}_k$ for the 1\textsuperscript{st}-kind Chebyshev polynomials, compared with the computed values of the bounds for the 4\textsuperscript{th}-kind Chebyshev polynomials and the lower bound for the numerically-approximated optimal 4\textsuperscript{th}-kind Chebyshev polynomials. We observe that for low values of $k$, i.e., $k \leq 6$, the bound obtained with the optimization of the Chebyshev polynomials of the 1\textsuperscript{st}-kind is lower than that obtained with those of the 4\textsuperscript{th}-kind in~\eqref{eq:lottebound-4cheby}. We also observe that for $k >1$ the bound is larger than the lower bound corresponding to the optimal 4\textsuperscript{th}-kind Chebyshev polynomials.

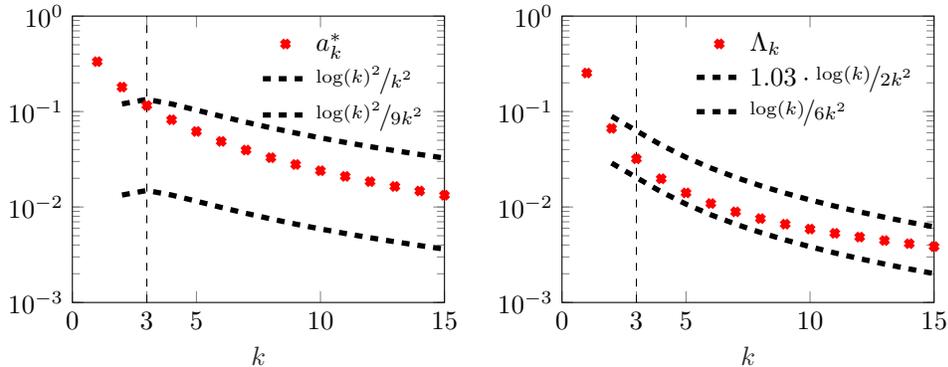
\begin{figure}[htbp]
	\centering
%
%
\begin{tikzpicture}

\begin{axis}[%
width=0.38\columnwidth,
height=1.5in,
at={(0in,0in)},
scale only axis,
xmin=0,
xmax=15,
xtick={ 0,  3,  5, 10, 15},
xlabel style={font=\color{white!15!black}},
xlabel={$k$},
ymin=0.001,
ymax=1,
ymode=log,
axis background/.style={fill=white},
legend style={legend cell align=left, align=left, draw=none, fill=none}
]
\addplot [color=red, line width=2.0pt, only marks, mark=x, mark options={solid, red}]
  table[row sep=crcr]{%
1	0.333333333333333\\
2	0.180535992740301\\
3	0.115927846486221\\
4	0.0820780659590383\\
5	0.0618496002413377\\
6	0.0486605823426062\\
7	0.0395132986024057\\
8	0.032870101754488\\
9	0.02787028627218\\
10	0.023998740960062\\
11	0.0209304400432259\\
12	0.0184513099045066\\
13	0.0164152586042591\\
14	0.0147195638076874\\
15	0.0132901324757843\\
};
\addlegendentry{$a_k^*$}

\addplot [color=black, dashed, line width=2.0pt]
  table[row sep=crcr]{%
1	0\\
2	0.12011325347955\\
3	0.134105440090287\\
4	0.12011325347955\\
5	0.103611615759209\\
6	0.0891778332102334\\
7	0.077276863432581\\
8	0.0675637050822471\\
9	0.0596024178179053\\
10	0.053018981104784\\
11	0.0475198490851965\\
12	0.0428802851261154\\
13	0.0389287882032092\\
14	0.0355337938214083\\
15	0.0325934928519543\\
};
\addlegendentry{$\nicefrac{\log(k)^2}{k^2}$}

\addplot [color=black, dashed, line width=2.0pt]
  table[row sep=crcr]{%
1	0\\
2	0.0133459170532834\\
3	0.0149006044544763\\
4	0.0133459170532834\\
5	0.0115124017510233\\
6	0.00990864813447037\\
7	0.00858631815917567\\
8	0.0075070783424719\\
9	0.00662249086865614\\
10	0.00589099790053155\\
11	0.0052799832316885\\
12	0.00476447612512394\\
13	0.00432542091146769\\
14	0.00394819931348981\\
15	0.0036214992057727\\
};
\addlegendentry{$\nicefrac{\log(k)^2}{9k^2}$}

\addplot [color=black, dashed, forget plot]
  table[row sep=crcr]{%
3	0.001\\
3	1\\
};
\end{axis}

\begin{axis}[%
width=0.38\columnwidth,
height=1.5in,
at={(0.5\columnwidth,0in)},
scale only axis,
xmin=0,
xmax=15,
xtick={ 0,  3,  5, 10, 15},
xlabel style={font=\color{white!15!black}},
xlabel={$k$},
ymode=log,
ymin=0.001,
ymax=1,
yminorticks=true,
axis background/.style={fill=white},
legend style={legend cell align=left, align=left, draw=none, fill=none}
]
\addplot [color=red, line width=2.0pt, only marks, mark=x, mark options={solid, red}]
  table[row sep=crcr]{%
1	0.253322533118946\\
2	0.0666095886474633\\
3	0.0319671862678828\\
4	0.0197699411093773\\
5	0.0140523814284038\\
6	0.0108789004197773\\
7	0.00890413869509085\\
8	0.00756842026098009\\
9	0.00660590212047851\\
10	0.00587748685606092\\
11	0.00530473429956836\\
12	0.00484063307992698\\
13	0.00445552185446093\\
14	0.00412983471181633\\
15	0.0038501517289458\\
};
\addlegendentry{$\Lambda_{k}$}

\addplot [color=black, dashed, line width=2.0pt]
  table[row sep=crcr]{%
1	0\\
2	0.089242699497093\\
3	0.0628650365182307\\
4	0.0446213497485465\\
5	0.0331544209961425\\
6	0.0256321146292347\\
7	0.0204519127910916\\
8	0.0167330061557049\\
9	0.0139700081151624\\
10	0.0118583132289193\\
11	0.0102059178966212\\
12	0.00888699253222792\\
13	0.00781626579344788\\
14	0.00693425777934622\\
15	0.00619842601585617\\
};
\addlegendentry{$1.03\cdot \nicefrac{\log(k)}{2k^2}$}

\addplot [color=black, dashed, line width=2.0pt]
  table[row sep=crcr]{%
1	0\\
2	0.0288811325233311\\
3	0.0203446720123724\\
4	0.0144405662616655\\
5	0.010729586082894\\
6	0.00829518272790766\\
7	0.0066187420035895\\
8	0.00541521234812457\\
9	0.00452103822497165\\
10	0.00383764182165674\\
11	0.00330288605068646\\
12	0.00287604936318056\\
13	0.00252953585548475\\
14	0.00224409636872046\\
15	0.00200596311192756\\
};
\addlegendentry{$\nicefrac{\log(k)}{6k^2}$}

\addplot [color=black, dashed, forget plot]
  table[row sep=crcr]{%
3	0.001\\
3	1\\
};
\end{axis}
\end{tikzpicture}%
	\caption{Bounds and computed quantities for the optimal parameters $a_k^*$ for the 1\textsuperscript{st}-kind Chebyshev polynomials and the smoothing constant $\Lambda_k$, $k=1,\ldots,15$.}
	\label{fig:constantbounds}
\end{figure}

\begin{figure}[htbp]
	\centering
%
%
\definecolor{mycolor1}{rgb}{0.00000,0.44700,0.74100}%
\definecolor{mycolor2}{rgb}{0.85000,0.32500,0.09800}%
\begin{tikzpicture}

\begin{axis}[%
width=0.3\columnwidth,
height=1.3in,
at={(0\columnwidth,0in)},
scale only axis,
xmin=0,
xmax=15,
ymode=log,
ymin=0.001,
ymax=1,
yminorticks=true,
axis background/.style={fill=white},
xmajorgrids,
xminorgrids,
ymajorgrids,
yminorgrids
]
\addplot [color=mycolor1, only marks, mark=square, mark options={solid, mycolor1}, forget plot]
  table[row sep=crcr]{%
1	0.375\\
2	0.125\\
3	0.0625\\
4	0.0375\\
5	0.025\\
6	0.0178571428571429\\
7	0.0133928571428571\\
8	0.0104166666666667\\
9	0.00833333333333333\\
10	0.00681818181818182\\
11	0.00568181818181818\\
12	0.00480769230769231\\
13	0.00412087912087912\\
14	0.00357142857142857\\
15	0.003125\\
};
\addplot [color=mycolor2, only marks, mark=triangle, mark options={solid, mycolor2}, forget plot]
  table[row sep=crcr]{%
1	0.333333333333333\\
2	0.112015284483472\\
3	0.0583799108887474\\
4	0.0364585625794908\\
5	0.0251807505038628\\
6	0.0185523566224834\\
7	0.0142996943551221\\
8	0.0113957022544334\\
9	0.00931777395189635\\
10	0.00777582002921479\\
11	0.00659772898163279\\
12	0.00567585604215863\\
13	0.00493992741097829\\
14	0.00434240512759291\\
15	0.0038501517289458\\
};
\addplot [color=red, only marks, mark=+, mark options={solid, red}, forget plot]
  table[row sep=crcr]{%
1	0.333333333333333\\
2	0.105572809000084\\
3	0.052095083601687\\
4	0.0310912041257632\\
5	0.020672197824105\\
6	0.014743298009627\\
7	0.0110469002707826\\
8	0.00858655133486778\\
9	0.00686617111092233\\
10	0.00561594964618996\\
11	0.00467881635584698\\
12	0.00395825535563739\\
13	0.00339228985424323\\
14	0.00293963756349429\\
15	0.00257193619047714\\
};
\end{axis}

\begin{axis}[%
width=0.3\columnwidth,
height=1.7in,
at={(0.432\columnwidth,0in)},
scale only axis,
xmin=1,
xmax=5,
ymode=log,
ymin=0.020672197824105,
ymax=0.375,
yminorticks=true,
axis background/.style={fill=white},
xmajorgrids,
xminorgrids,
ymajorgrids,
yminorgrids,
legend style={at={(0.211,0.7)}, anchor=south west, legend cell align=left, align=left, draw=none, fill=none, font=\small}
]
\addplot [color=mycolor1, only marks, mark=square, mark options={solid, mycolor1}]
  table[row sep=crcr]{%
1	0.375\\
2	0.125\\
3	0.0625\\
4	0.0375\\
5	0.025\\
};
\addlegendentry{Chebyshev 4 -- Eq.~\eqref{polsmoother}}

\addplot [color=mycolor2, only marks, mark=triangle, mark options={solid, mycolor2}]
  table[row sep=crcr]{%
1	0.333333333333333\\
2	0.112015284483472\\
3	0.0583799108887474\\
4	0.0364585625794908\\
5	0.0251807505038628\\
};
\addlegendentry{Opt. Chebyshev 1 -- Alg.~\ref{alg:chebyshev_iteration}}

\addplot [color=red, only marks, mark=+, mark options={solid, red}]
  table[row sep=crcr]{%
1	0.333333333333333\\
2	0.105572809000084\\
3	0.052095083601687\\
4	0.0310912041257632\\
5	0.020672197824105\\
};
\addlegendentry{Opt. Chebyshev 4 -- Eq.~\eqref{polsmootheropt}}

\end{axis}

\begin{axis}[%
width=\columnwidth,
height=1.84in,
at={(-0.128\columnwidth,-0.202in)},
scale only axis,
xmin=0,
xmax=1,
ymin=0,
ymax=1,
axis line style={draw=none},
ticks=none,
axis x line*=bottom,
axis y line*=left
]
\draw[draw=black] (axis cs:0.143487046632124,0.38) rectangle (axis cs:0.254749568221071,0.78);
\draw[->, color=black,ultra thick] (axis cs:0.253,0.653) -- (axis cs:0.57,0.484);
\end{axis}
\end{tikzpicture}%
	\caption{Comparison of computed bound of the quasi-optimal 1\textsuperscript{st}-kind~\eqref{polsmootheropt1kind} and 4\textsuperscript{th}-kind Chebyshev polynomials~\eqref{polsmoother} and the numerically optimized polynomials expressed in terms of the  4\textsuperscript{th}-kind Chebyshev~\eqref{polsmootheropt}, for $k=1,\ldots,15$. The enlargement on the right shows where polynomials of the 1\textsuperscript{st}-kind return a better bound than those of the 4\textsuperscript{th}.}
	\label{fig:boundcomparison}
\end{figure}
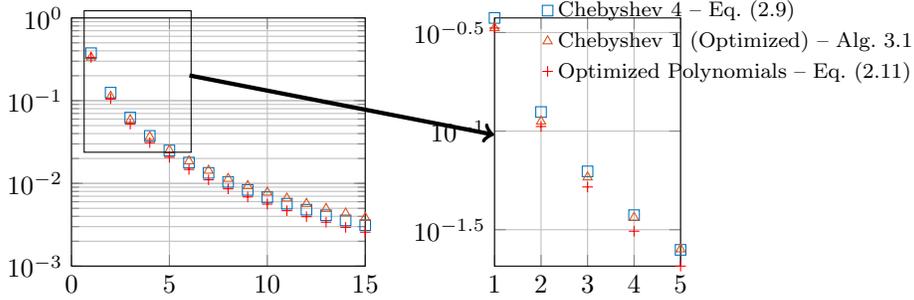

\begin{remark}
    Observe that all the discussed upper bounds in Theorem~\ref{thm:convergence_constants} and the upper bounds discussed in \cite{LottesNLAA2023} concern the solution of the minimax problem~\eqref{eq:prob1}. However, the optimal performance for a given matrix $A$ is achieved by the polynomial that solves the same minimax problem on the discrete set of eigenvalues of $A$. Therefore, it may happen that any of the polynomial accelerators can behave better than expected for certain eigenvalues distribution; see Figure~\ref{fig:better-than-optimal} for an example in which we only use the polynomial smoother as a preconditioner.
\begin{figure}[htbp]
    \centering
    \subfloat[Equispaced]{\includegraphics[width=0.33\columnwidth]{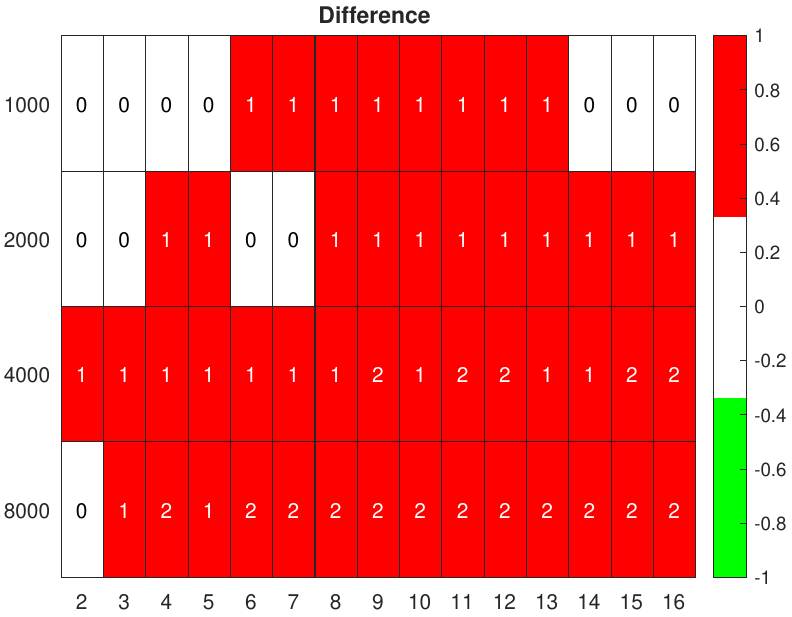}}
    \subfloat[Accumulating at the boundaries]{\includegraphics[width=0.33\columnwidth]{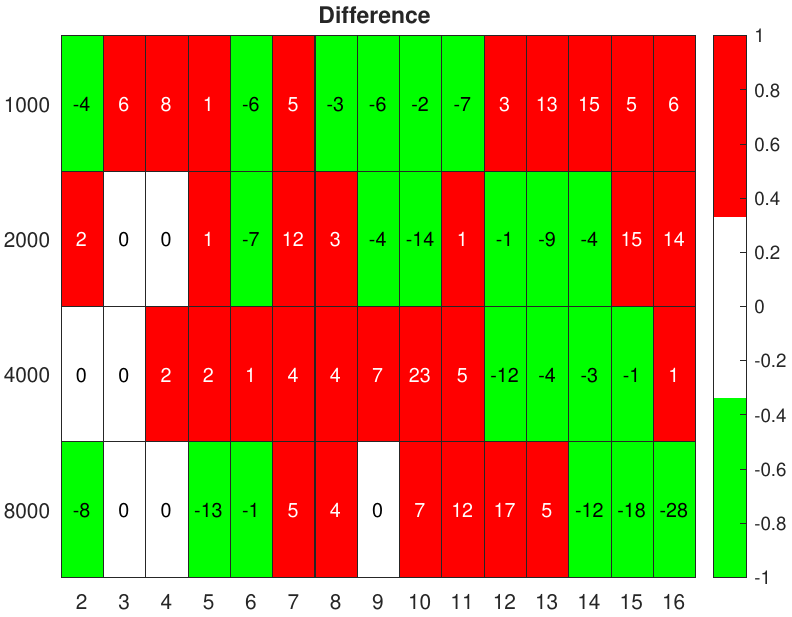}}
    \subfloat[Gap]{\includegraphics[width=0.31\columnwidth]{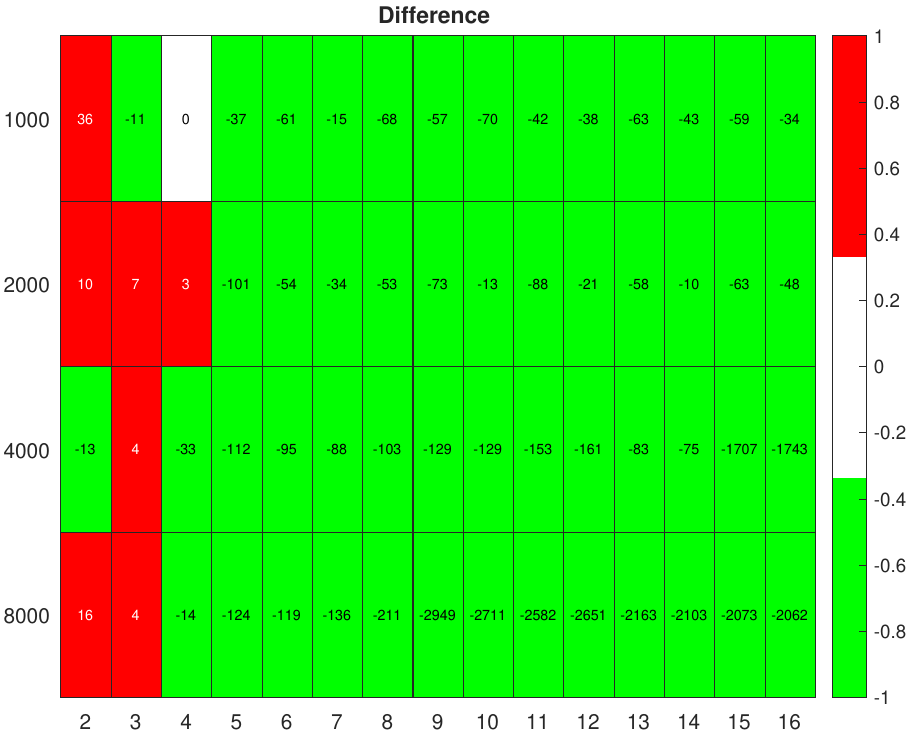}}
    

    \caption{Comparison of PCG iteration counts (\texttt{tol} = $10^{-5}$) using polynomial preconditioning based on optimized Chebyshev polynomials of the 1\textsuperscript{st} and 4\textsuperscript{th} kind~\cite{LottesNLAA2023}, applied on top of the $\ell_1$-Jacobi method as a preconditioner---i.e., employing only the smoother. The test matrix $A$ is defined as $A = Q D Q^\top$, where $Q$ is the eigenvector matrix of the finite-difference discretization of the second derivative, and $D$ is diagonal with three eigenvalue distributions: (a) equispaced in $[1/N, 1]$ (left panel), (b) logarithmically spaced towards $0$ and $1$ as $D = \text{\lstinline[style=Matlab-editor]{diag([1-logspace(-8,-1,N/2),logspace(-8,-1,N/2)])}}$ (center panel), and (c) logarithmically spaced with a gap as $D = \text{\lstinline[style=Matlab-editor]{diag([logspace(-8,-1,N/2),logspace(1,pi,N/2)])}}$ (right panel). The right-hand side is $\mathbf{b} = A\mathbf{1}$. Results are shown for increasing polynomial degrees (abscissa) and matrix sizes (ordinate). Green cells indicate fewer iterations for 1\textsuperscript{st} kind, while red cells indicate fewer iterations for 4\textsuperscript{th} kind.}
    \label{fig:better-than-optimal}
\end{figure}
    
\end{remark}

\section{Design and implementation based on PSCToolkit}\label{sec:amgprec}

The implementation and benchmarking of these smoothers were carried out using the Parallel Sparse Computation Toolkit (\texttt{PSCToolkit}) framework~\cite{softimpact}. {A significant part of the contribution has consisted} in balancing computational efficiency with memory access latency, leveraging optimized GPU memory management techniques. 
The efficient use of the various levels of memory available in GPU accelerators is crucial for
achieving high performance in parallel computing applications; similarly to other devices,  
NVIDIA GPUs contain various memory levels.
The global memory is typically quite large but exhibits  a very high latency, 
of the order of hundreds of cycles, thus necessitating careful programming to hide the cost of its
access; the bandwidth depends on the specific card model, and  while for high-end devices it can be
significantly larger than the bandwidth available on normal CPU memory, it is still much
slower than the combined execution rate of the arithmetic units. 
Shared memory, on the other hand, has a much lower latency and higher bandwidth, and 
can act as a programmable cache enabling faster data exchange, but only among threads within 
the same thread block: proper utilization of shared memory  
can significantly reduce global memory accesses and improve kernel performance. 
GPUs also have registers allocated per-thread, and they are used for storing private 
data of individual threads; these are very fast,  but their combined size is limited. 
Balancing the usage of these different memory types by strategically placing data in 
the most appropriate memory space can drastically enhance the efficiency of GPU kernels. 
If we transform the recurrences necessary to apply the polynomial smoother based on Chebyshev
polynomials of the first type to equation~\eqref{polsmootheropt1kind} into an
algorithmic form we obtain the version described in Algorithm~\ref{alg:chebyshev_iteration}.
\begin{algorithm}[htbp]
	\begin{algorithmic}[1]
		\Function{ChebyshevAcceleration}{$A$,$\mathbf{b}$,$\mathbf{x}^{(0)}$,$M$,$a$,$k$}
		\State $\mathbf{r}^{(0)} \gets  \frac{1}{\rho(M^{-1}A)}M^{-1}(\mathbf{b} - A \mathbf{x}^{(0)})$ \Comment{For $\ell_1$-Jacobi replace $\rho(M^{-1}A)$ with $1$.}
		\State $\sigma_1 \gets \nicefrac{\theta}{\delta}$
		\State $\rho_0 \gets \nicefrac{1}{\sigma_1}$\;
		\State $\mathbf{d}^{(0)} \gets \nicefrac{\mathbf{r}^{(0)}}{\theta}$
		\State $\mathbf{x}^{(1)} \gets \mathbf{x}^{(0)} + \mathbf{d}^{(0)}$ 
		\For{$j=1,\ldots,k-1$}
		\State $\mathbf{s}^{(j)} \gets \frac{1}{\rho(M^{-1}A)} M^{-1} A \mathbf{d}^{(j-1)}$ \Comment{For $\ell_1$-Jacobi replace $\rho(M^{-1}A)$ with $1$.}
		\State $\rho_{j} \gets (2\sigma_1 - \rho_{j-1})^{-1}$
		\State $\mathbf{r}^{(j)} \gets \mathbf{r}^{(j-1)} -  \mathbf{s}^{(j)}$ \label{line:alg1}
		\State $\displaystyle \mathbf{d}^{(j)} \gets \rho_{j}\rho_{j-1} \mathbf{d}^{(j-1)} + \frac{2 \rho_{j}}{\delta} \mathbf{r}^{(j)}$ \label{line:alg2}
		\State $\mathbf{x}^{(j+1)} \gets \mathbf{x}^{(j)} + \mathbf{d}^{(j)}$ \label{line:alg3}
		\EndFor
		\State \Return $\mathbf{x}^{(k)}$
		\EndFunction
	\end{algorithmic}
	\caption{Chebyshev acceleration for the polynomial defined in the recurrence relation~\eqref{eq:recurrence-relation-shifted} and a fixed-point iteration with matrix $M$.}\label{alg:chebyshev_iteration}
\end{algorithm}
From this formulation we can observe that the sequence of \texttt{axpy} calls at
lines~\ref{line:alg1}, \ref{line:alg2} and~\ref{line:alg3} accesses the same vectors 
multiple times; it is therefore advisable to minimize the number of accesses to global memory
by making explicit use of shared memory, e.g. delaying the store operations of lines 10 and 11. 
In a more conventional CPU programming context, these 
optimizations are handled implicitly by the combined action of the memory hierarchy  and cache memory 
policies, and the compiler optimizations; however in a GPU context it is often necessary to pay
explicit attention  to such low-level details in the user code. 
The evolution of programming tools promises to address and alleviate  these issues may; 
however at this time we have chosen to  implement an ad-hoc kernel called
\lstinline[language=Fortran]{psb_upd_xyz(a,b,g,d,x,y,z,desc,info)} to handle the vector updates
minimizing the  data movements to/from global memory. With this implementation, the cost of applying
$k$ steps of the basic $\ell_1$-Jacobi smoother in terms of arithmetic operations and data movement is
indeed the same as that of applying the polynomial smoother in Algorithm~\ref{alg:chebyshev_iteration}
with a polynomial of degree $k$, and thus we expect a
the time per iteration to be comparable. Note that the kernel can be reused for  the
implementation of the other two polynomial smoothers based on the use of Chebyshev 
polynomials of the 4\textsuperscript{th} kind, thus unifying all implementation variants, 
and enabling the user to concentrate on the effect on the number of 
iterations provided by the better approximation property.

\section{Numerical Examples}\label{sec:results}
To test the different polynomial smoothers discussed in the previous section 
we need to complete the construction of the multigrid methods with a coarsening strategy and a coarse solver. 
We focus on methods available through the \texttt{PSCToolkit} framework. This is a software project aimed to make available parallel Krylov solvers and AMG preconditioners efficient on heterogeneous clusters so that we can leverage on high-throughput multi/many-core processors and large numbers of computing nodes for solving linear systems with a very large number of unknowns and prepare our software framework to face the exascale challenge. {AMG preconditioners are set by coarsening based on aggregation of unknowns relying on two different approaches, as summarized in the following list:}
\begin{description}
	\item[VBM] the classical smoothed aggregation strategy by Van\v{e}k, Mandel, and Brezina \cite{MR1393006}, that is adapted to the distributed setting in a \emph{decoupled} way, i.e., for which the aggregates are determined only on the local part of the matrix of that level;
	\item[Compatible Matching] the aggregation strategy based on graph weighted matching and compatible relaxation~\cite{MR3865827,MR4331965}, that is instead a parallel \emph{coupled} strategy which relies on the computation of an approximate graph weighted matching of the whole adjacency matrix. For this choice the maximum size of the requested aggregates can be selected in input as a power of $2$, i.e., it corresponds on the number of the matching application we perform at each level. {For the results presented in this paper, we used the default setting for the number of matching steps, which corresponds to aggregates of size up to $8$.} The prolongator obtained by this construction can 
	be used as tentative prolongator for a classic smoothing procedure as in \cite{MR1393006}, i.e., if we call $\hat{P}_l$ the tentative prolongator obtained via the matching strategy at level $l$, the one we currently use is
	\[
	P_l = (I - \omega D_l^{-1}A_l)\hat{P}_l,
	\]
	for $D_l$ the diagonal of $A_l$, and $\omega = \nicefrac{4}{3 \lambda_{\max}}$ for $\lambda_{\max}$ an estimate of the largest eigenvalue of $D_l^{-1}A_l$.
\end{description}
For more details on the AMG preconditioners included in \texttt{PSCToolkit} we refer the reader to~\cite{MR4331965,softimpact}. 

Both the types of the AMG hierarchies can be employed as standard $V$-Cycle~\eqref{eq:amg} in the application phase as preconditioner for a Krylov solver.
For the experiments discussed in this paper,
at the coarsest level we applied a fixed (sufficiently large) number of iteration of the highly parallel $\ell_1$-Jacobi method.
\begin{table}[htbp]
	\centering
	\caption{The table contains the building blocks from \texttt{PSCToolkit} we use to build the preconditioners we consider in the experiments. The numerical experiments will consider all possible combinations of Pre/Post-Smoother and aggregation strategies.}
	\label{tab:preconditioners}
	\begin{tabular}{llll}
		\toprule
		Pre/Post-smoother & Aggregation & Cycle & Coarse  \\
		\midrule
		$\ell_1$-Jacobi & VBM & $V$-Cycle & $\ell_1$-Jacobi  \\
		Chebyshev 4\textsuperscript{th}-kind & Compatible Matching& & \\
		Optimized Chebyshev 4\textsuperscript{th}-kind & & &  \\
		Optimized Chebyshev 1\textsuperscript{st}-kind &  &  & \\
		\bottomrule
	\end{tabular}
	
\end{table}
Table~\ref{tab:preconditioners} succinctly summarizes the different building blocks we use.

To test the proposed strategy we consider linear systems~\eqref{eq:sys1} arising from usual benchmarks, as described in the following.
\begin{description}
	\item[Poisson 3D:] the matrix $A$ comes from the usual 7-points discretization of the Poisson equation on the unit cube, when Dirichlet boundary conditions are applied and right-hand side is set equal to the unit vector.
	\item[Anisotropy 2D:] the matrix $A$ comes from an anisotropic 2D thermal problem with Q1-finite element methods on the domain $[-1,1]^2$, and the material conductivity given by the tensor
	\[
	\mathcal{K}_{\epsilon} = \begin{bmatrix}
		1 & 0 \\
		0 & \epsilon
	\end{bmatrix}
	\]
	rotated by an angle $\theta$ around the origin, Dirichlet boundary conditions fixed on the $y = -1$ face of the square, and right-hand side given by the discretization of 
	\[
	f(x,y) = \exp\left(-100(x^2 + y^2)\right).
	\]  
\end{description}

For all cases the iterative method of choice is a variant of the Conjugate Gradient (CG), requiring only one global synchronization for scalar products~\cite{NOTAY2015237}, and is set to achieve a relative tolerance on the error of $tol = 10^{-7}$ within a maximum number of iteration $\textrm{it}_{\max} = 1000$. The experiments were performed on the Leonardo Italian supercomputer, ranked 7th in the June 2024 Top500 list (BullSequana XH2000, Xeon Platinum 8358 32C 2.6GHz, NVIDIA A100 SXM4 64 GB, Quad-rail NVIDIA HDR100 Infiniband). The code, execution scripts and the pointers to the \texttt{PSCToolkit} packages versions used are available on the GitHub repository   \href{https://github.com/Cirdans-Home/polynomialsmoothers}{Cirdans-Home/polynomialsmoothers}.

\subsection{Weak Scalability Experiments}
We conduct weak scalability experiments using the two preconditioners described at the beginning of Section~\ref{sec:results}. More specifically, we maintain a fixed number of degrees of freedom per computational unit (GPU) while linearly increasing both the problem size and the number of GPUs involved, ranging from 64 to 1024 GPUs (i.e., from 16 to 256 computing nodes). In these experiments, the key metrics of interest are the solution time, which indicates the method's efficiency for the user and its utility within more complex simulations; the number of iterations, which illustrates the theoretical properties of the considered polynomial acceleration; and the time per iteration, which measures the efficiency of the method's implementation.

\subsubsection{Poisson 3D}\label{sec:poisson3D}
For linear systems of the Poisson 3D test case, we present results for the multigrid preconditioner based on compatible matching aggregation because it always produces the smallest number of iterations. These results are complemented by additional data in the supplementary materials, which include higher degree polynomial smoothers and the preconditioner based on VBM aggregation. {The problem size per GPU is $6$ million, allowing us to solve problems with over $6$ billion dofs using $1024$ GPUs.}

\paragraph{Compatible Matching}

We first look at the solution time with different smoothers as the number of GPUs increases;
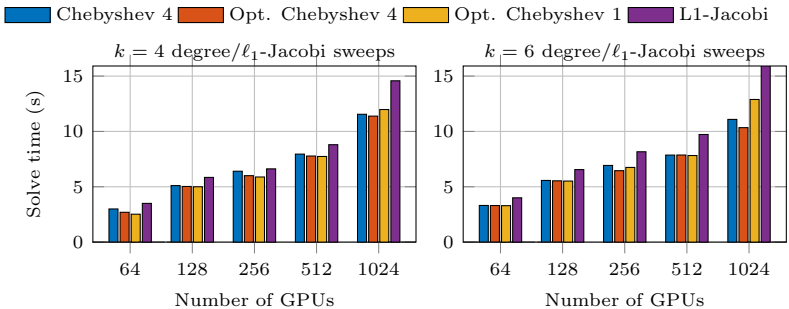
\begin{figure}[htbp]
	\centering
%
%
\definecolor{mycolor1}{rgb}{0.00000,0.44700,0.74100}%
\definecolor{mycolor2}{rgb}{0.85000,0.32500,0.09800}%
\definecolor{mycolor3}{rgb}{0.92900,0.69400,0.12500}%
\definecolor{mycolor4}{rgb}{0.49400,0.18400,0.55600}%
\begin{tikzpicture}

\begin{axis}[%
width=0.33\columnwidth,
height=0.18\columnwidth,
at={(0\columnwidth,0.52\columnwidth)},
scale only axis,
bar shift auto,
ybar=2*\pgflinewidth,
xmin=0.4,
xmax=5.6,
xtick={1,2,3,4,5},
xticklabels={{64},{128},{256},{512},{1024}},ticklabel style = {font=\scriptsize},
xlabel style={font=\scriptsize},
xlabel={Number of GPUs},
ymin=0,
ymax=15.911,
ylabel style={font=\scriptsize},
ylabel={Solve time (s)},
axis background/.style={fill=white},
title style={font=\scriptsize,yshift=-0.8em},
title={$k=4$ degree/$\ell_1$-Jacobi sweeps},
xmajorgrids,
ymajorgrids,
legend columns=4,
legend style={at={(-0.3,1.4)}, anchor=north west, legend cell align=left, align=left, draw=none, fill=none, font=\scriptsize}
]
\addplot[ybar, bar width=0.145, fill=mycolor1, draw=black, area legend] table[row sep=crcr] {%
1	2.99555\\
2	5.10844\\
3	6.40426\\
4	7.95049\\
5	11.5501\\
};
\addplot[forget plot, color=white!15!black] table[row sep=crcr] {%
0.654545454545455	0\\
5.34545454545454	0\\
};
\addlegendentry{Chebyshev 4}

\addplot[ybar, bar width=0.145, fill=mycolor2, draw=black, area legend] table[row sep=crcr] {%
1	2.6942\\
2	5.02838\\
3	6.00266\\
4	7.77814\\
5	11.384\\
};
\addplot[forget plot, color=white!15!black] table[row sep=crcr] {%
0.654545454545455	0\\
5.34545454545454	0\\
};
\addlegendentry{Opt. Chebyshev 4}

\addplot[ybar, bar width=0.145, fill=mycolor3, draw=black, area legend] table[row sep=crcr] {%
1	2.5222\\
2	5.0013\\
3	5.8837\\
4	7.73889\\
5	11.9796\\
};
\addplot[forget plot, color=white!15!black] table[row sep=crcr] {%
0.654545454545455	0\\
5.34545454545454	0\\
};
\addlegendentry{Opt. Chebyshev 1}

\addplot[ybar, bar width=0.145, fill=mycolor4, draw=black, area legend] table[row sep=crcr] {%
1	3.49991\\
2	5.84694\\
3	6.61824\\
4	8.80121\\
5	14.5746\\
};
\addplot[forget plot, color=white!15!black] table[row sep=crcr] {%
0.654545454545455	0\\
5.34545454545454	0\\
};
\addlegendentry{L1-Jacobi}

\end{axis}

\begin{axis}[%
width=0.33\columnwidth,
height=0.18\columnwidth,
at={(0.378\columnwidth,0.52\columnwidth)},
scale only axis,
bar shift auto,
ybar=2*\pgflinewidth,
xmin=0.4,
xmax=5.6,
xtick={1,2,3,4,5},
xticklabels={{64},{128},{256},{512},{1024}},ticklabel style = {font=\scriptsize},
xlabel style={font=\scriptsize},
xlabel={Number of GPUs},
ymin=0,
ymax=15.911,
ylabel style={font=\scriptsize},
axis background/.style={fill=white},
title style={font=\scriptsize,yshift=-0.8em},
title={$k=6$ degree/$\ell_1$-Jacobi sweeps},
xmajorgrids,
ymajorgrids
]
\addplot[ybar, bar width=0.145, fill=mycolor1, draw=black, area legend] table[row sep=crcr] {%
1	3.31203\\
2	5.57183\\
3	6.92748\\
4	7.86276\\
5	11.0869\\
};
\addplot[forget plot, color=white!15!black] table[row sep=crcr] {%
0.654545454545455	0\\
5.34545454545454	0\\
};
\addplot[ybar, bar width=0.145, fill=mycolor2, draw=black, area legend] table[row sep=crcr] {%
1	3.30466\\
2	5.53617\\
3	6.45388\\
4	7.86843\\
5	10.3415\\
};
\addplot[forget plot, color=white!15!black] table[row sep=crcr] {%
0.654545454545455	0\\
5.34545454545454	0\\
};
\addplot[ybar, bar width=0.145, fill=mycolor3, draw=black, area legend] table[row sep=crcr] {%
1	3.29101\\
2	5.51514\\
3	6.7484\\
4	7.82342\\
5	12.888\\
};
\addplot[forget plot, color=white!15!black] table[row sep=crcr] {%
0.654545454545455	0\\
5.34545454545454	0\\
};
\addplot[ybar, bar width=0.145, fill=mycolor4, draw=black, area legend] table[row sep=crcr] {%
1	4.00202\\
2	6.55137\\
3	8.16297\\
4	9.72351\\
5	15.911\\
};
\addplot[forget plot, color=white!15!black] table[row sep=crcr] {%
0.654545454545455	0\\
5.34545454545454	0\\
};
\end{axis}

\end{tikzpicture}%
	
	\caption{Solve time (s). Time to solution (s) for the Poisson 3D test case employing the multigrid scheme based on the compatible matching aggregation in conjunction with the different polynomial accelerators and the baseline $\ell_1$-Jacobi for comparison.}
	\label{fig:3d_1024_match_tsolve_reduced}
\end{figure}
As expected, all the polynomial accelerators reduces the solve time with respect to the baseline $\ell_1$-Jacobi; the reduction is more significant as the problem size increases and as polynomial degree (number of baseline iterations) increases. For a low degree ($k=4$), the quasi-optimal 1\textsuperscript{st}-kind Chebyshev polynomial is better or comparable with the optimized 4\textsuperscript{th}-kind ones. Their application produces the best solution times for all the
number of GPUs but the largest one, where we observe the best solve time for the optimized 4\textsuperscript{th}-kind polynomials when $k=6$. 
This confirms that having a polynomial smoother that performs well for low polynomial degrees can be extremely effective. In other words, the initial transient behavior of the estimate turns out to be more relevant than the asymptotic behavior in the degree. 
Notably, the proposed method significantly improves performance. It results
in a solve time reduction ranging from $\approx 18\%$ to $\approx 35 \%$ for the largest problem size, {which corresponds to a proportional reduction in energy consumption}.
This gain in the solve times comes from the better convergence properties of the polynomial accelerators and from the efficiency of their GPU implementation, as shown in Fig.~\ref{fig:3d_1024_match_tsolve_reduced} and Fig.~\ref{fig:3d_1024_match_timeperiteration_reduced}, respectively.
Indeed, if we investigate the number of iterations, we observe that, Fig.~\ref{fig:3d_1024_match_itercount_reduced} is in a good agreement with the bound behaviour in  Fig.~\ref{fig:boundcomparison}; indeed, the quasi-optimal 1\textsuperscript{st}-kind Chebyshev polynomial and the optimized 4\textsuperscript{th}-kind ones return the same number of iterations for $k=4$, while in the case of $k=6$, the optimized 4\textsuperscript{th}-kind polynomials reduces the number of iterations more than the other approaches when solving the largest size problem. In particular, we observe that increasing polynomial degree
from $k=4$ to $k=6$ reduces the number of iterations from $18$ to $14$ in the case of the optimized 4\textsuperscript{th}-kind 
while we go from $18$ to $15$ for the quasi-optimal 1\textsuperscript{st}-kind. On the contrary the baseline is able to reduce the iterations only from $21$ to $20$. The times per iteration of the different polynomial accelerators
remain fully comparable to those of the $\ell_1$-Jacobi method, showing that the advantage in the approximation properties of the polynomial smoother does not increase the computational cost of the entire solution procedure.
\begin{figure}[htbp]
	\centering
%
%
\definecolor{mycolor1}{rgb}{0.00000,0.44700,0.74100}%
\definecolor{mycolor2}{rgb}{0.85000,0.32500,0.09800}%
\definecolor{mycolor3}{rgb}{0.92900,0.69400,0.12500}%
\definecolor{mycolor4}{rgb}{0.49400,0.18400,0.55600}%
\begin{tikzpicture}

\begin{axis}[%
width=0.33\columnwidth,
height=0.18\columnwidth,
at={(0\columnwidth,0.519\columnwidth)},
scale only axis,
bar shift auto,
ybar=2*\pgflinewidth,
xmin=0.4,
xmax=5.6,
xtick={1,2,3,4,5},
xticklabels={{64},{128},{256},{512},{1024}},ticklabel style = {font=\scriptsize},
ymin=0,
ymax=21,
ylabel style={font=\scriptsize},
ylabel={Iterations},
xlabel style={font=\scriptsize},
xlabel={Number of GPUs},
axis background/.style={fill=white},
title style={font=\scriptsize,yshift=-0.8em},
title={$k=4$ degree/$\ell_1$-Jacobi sweeps},
xmajorgrids,
ymajorgrids,
legend columns=4,
legend style={at={(-0.3,1.4)}, anchor=north west, legend cell align=left, align=left, draw=none, fill=none, font=\scriptsize}
]
\addplot[ybar, bar width=0.145, fill=mycolor1, draw=black, area legend] table[row sep=crcr] {%
1	10\\
2	11\\
3	17\\
4	14\\
5	18\\
};
\addplot[forget plot, color=white!15!black] table[row sep=crcr] {%
0.654545454545455	0\\
5.34545454545454	0\\
};
\addlegendentry{Chebyshev 4}

\addplot[ybar, bar width=0.145, fill=mycolor2, draw=black, area legend] table[row sep=crcr] {%
1	9\\
2	11\\
3	16\\
4	14\\
5	18\\
};
\addplot[forget plot, color=white!15!black] table[row sep=crcr] {%
0.654545454545455	0\\
5.34545454545454	0\\
};
\addlegendentry{Opt. Chebyshev 4}

\addplot[ybar, bar width=0.145, fill=mycolor3, draw=black, area legend] table[row sep=crcr] {%
1	9\\
2	11\\
3	16\\
4	14\\
5	18\\
};
\addplot[forget plot, color=white!15!black] table[row sep=crcr] {%
0.654545454545455	0\\
5.34545454545454	0\\
};
\addlegendentry{Opt. Chebyshev 1}

\addplot[ybar, bar width=0.145, fill=mycolor4, draw=black, area legend] table[row sep=crcr] {%
1	12\\
2	13\\
3	18\\
4	16\\
5	21\\
};
\addplot[forget plot, color=white!15!black] table[row sep=crcr] {%
0.654545454545455	0\\
5.34545454545454	0\\
};
\addlegendentry{L1-Jacobi}

\end{axis}

\begin{axis}[%
width=0.33\columnwidth,
height=0.18\columnwidth,
at={(0.378\columnwidth,0.519\columnwidth)},
scale only axis,
bar shift auto,
ybar=2*\pgflinewidth,
xmin=0.4,
xmax=5.6,
xtick={1,2,3,4,5},
xticklabels={{64},{128},{256},{512},{1024}},ticklabel style = {font=\scriptsize},
xlabel style={font=\scriptsize},
xlabel={Number of GPUs},
ymin=0,
ymax=21,
ylabel style={font=\scriptsize},
axis background/.style={fill=white},
title style={font=\scriptsize,yshift=-0.8em},
title={$k=6$ degree/$\ell_1$-Jacobi sweeps},
xmajorgrids,
ymajorgrids
]
\addplot[ybar, bar width=0.145, fill=mycolor1, draw=black, area legend] table[row sep=crcr] {%
1	9\\
2	10\\
3	15\\
4	12\\
5	15\\
};
\addplot[forget plot, color=white!15!black] table[row sep=crcr] {%
0.654545454545455	0\\
5.34545454545454	0\\
};
\addplot[ybar, bar width=0.145, fill=mycolor2, draw=black, area legend] table[row sep=crcr] {%
1	9\\
2	10\\
3	14\\
4	12\\
5	14\\
};
\addplot[forget plot, color=white!15!black] table[row sep=crcr] {%
0.654545454545455	0\\
5.34545454545454	0\\
};
\addplot[ybar, bar width=0.145, fill=mycolor3, draw=black, area legend] table[row sep=crcr] {%
1	9\\
2	10\\
3	15\\
4	12\\
5	16\\
};
\addplot[forget plot, color=white!15!black] table[row sep=crcr] {%
0.654545454545455	0\\
5.34545454545454	0\\
};
\addplot[ybar, bar width=0.145, fill=mycolor4, draw=black, area legend] table[row sep=crcr] {%
1	11\\
2	12\\
3	18\\
4	15\\
5	20\\
};
\addplot[forget plot, color=white!15!black] table[row sep=crcr] {%
0.654545454545455	0\\
5.34545454545454	0\\
};
\end{axis}

\end{tikzpicture}%
	
	\caption{Iteration count. Total number of CG iterations for the Poisson 3D test case employing the multigrid scheme based on the compatible matching aggregation in conjunction with the different polynomial accelerators and the baseline $\ell_1$-Jacobi for comparison.}
	\label{fig:3d_1024_match_itercount_reduced}
\end{figure}
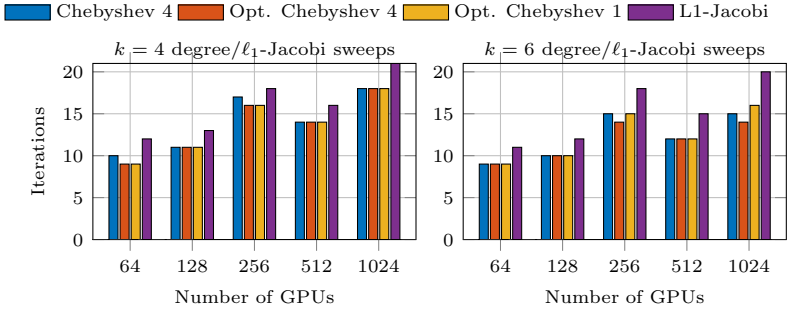
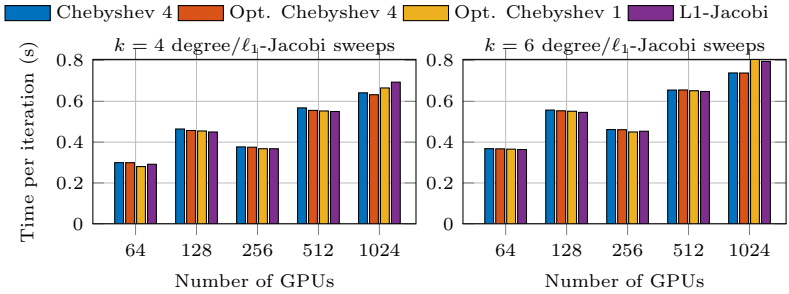
\begin{figure}[htbp]
	\centering
%
%
\definecolor{mycolor1}{rgb}{0.00000,0.44700,0.74100}%
\definecolor{mycolor2}{rgb}{0.85000,0.32500,0.09800}%
\definecolor{mycolor3}{rgb}{0.92900,0.69400,0.12500}%
\definecolor{mycolor4}{rgb}{0.49400,0.18400,0.55600}%
\begin{tikzpicture}

\begin{axis}[%
width=0.33\columnwidth,
height=0.168\columnwidth,
at={(0\columnwidth,0.519\columnwidth)},
scale only axis,
bar shift auto,
ybar=2*\pgflinewidth,
xmin=0.3,
xmax=5.6,
xtick={1,2,3,4,5},
xticklabels={{64},{128},{256},{512},{1024}},ticklabel style = {font=\scriptsize},
ymin=0,
ymax=0.805499,
ylabel style={font=\scriptsize},
ylabel={Time per iteration (s)},
xlabel style={font=\scriptsize},
xlabel={Number of GPUs},
axis background/.style={fill=white},
title style={font=\scriptsize,yshift=-0.8em},
title={$k=4$ degree/$\ell_1$-Jacobi sweeps},
xmajorgrids,
ymajorgrids,
legend columns=4,
legend style={at={(-0.3,1.4)}, anchor=north west, legend cell align=left, align=left, draw=none, fill=none, font=\scriptsize}
]
\addplot[ybar, bar width=0.145, fill=mycolor1, draw=black, area legend] table[row sep=crcr] {%
1	0.299555\\
2	0.464404\\
3	0.376721\\
4	0.567892\\
5	0.64167\\
};
\addplot[forget plot, color=white!15!black] table[row sep=crcr] {%
0.654545454545455	0\\
5.34545454545454	0\\
};
\addlegendentry{Chebyshev 4}

\addplot[ybar, bar width=0.145, fill=mycolor2, draw=black, area legend] table[row sep=crcr] {%
1	0.299355\\
2	0.457125\\
3	0.375166\\
4	0.555581\\
5	0.632442\\
};
\addplot[forget plot, color=white!15!black] table[row sep=crcr] {%
0.654545454545455	0\\
5.34545454545454	0\\
};
\addlegendentry{Opt. Chebyshev 4}

\addplot[ybar, bar width=0.145, fill=mycolor3, draw=black, area legend] table[row sep=crcr] {%
1	0.280245\\
2	0.454664\\
3	0.367731\\
4	0.552778\\
5	0.665533\\
};
\addplot[forget plot, color=white!15!black] table[row sep=crcr] {%
0.654545454545455	0\\
5.34545454545454	0\\
};
\addlegendentry{Opt. Chebyshev 1}

\addplot[ybar, bar width=0.145, fill=mycolor4, draw=black, area legend] table[row sep=crcr] {%
1	0.291659\\
2	0.449765\\
3	0.36768\\
4	0.550076\\
5	0.694026\\
};
\addplot[forget plot, color=white!15!black] table[row sep=crcr] {%
0.654545454545455	0\\
5.34545454545454	0\\
};
\addlegendentry{L1-Jacobi}

\end{axis}

\begin{axis}[%
width=0.33\columnwidth,
height=0.168\columnwidth,
at={(0.378\columnwidth,0.519\columnwidth)},
scale only axis,
bar shift auto,
ybar=2*\pgflinewidth,
xmin=0.3,
xmax=5.6,
xtick={1,2,3,4,5},
xticklabels={{64},{128},{256},{512},{1024}},ticklabel style = {font=\scriptsize},
xlabel style={font=\scriptsize},
xlabel={Number of GPUs},
ymin=0,
ymax=0.805499,
axis background/.style={fill=white},
title style={font=\scriptsize,yshift=-0.8em},
title={$k=6$ degree/$\ell_1$-Jacobi sweeps},
xmajorgrids,
ymajorgrids
]
\addplot[ybar, bar width=0.145, fill=mycolor1, draw=black, area legend] table[row sep=crcr] {%
1	0.368003\\
2	0.557183\\
3	0.461832\\
4	0.65523\\
5	0.739126\\
};
\addplot[forget plot, color=white!15!black] table[row sep=crcr] {%
0.654545454545455	0\\
5.34545454545454	0\\
};
\addplot[ybar, bar width=0.145, fill=mycolor2, draw=black, area legend] table[row sep=crcr] {%
1	0.367185\\
2	0.553617\\
3	0.460992\\
4	0.655702\\
5	0.738682\\
};
\addplot[forget plot, color=white!15!black] table[row sep=crcr] {%
0.654545454545455	0\\
5.34545454545454	0\\
};
\addplot[ybar, bar width=0.145, fill=mycolor3, draw=black, area legend] table[row sep=crcr] {%
1	0.365668\\
2	0.551514\\
3	0.449893\\
4	0.651952\\
5	0.805499\\
};
\addplot[forget plot, color=white!15!black] table[row sep=crcr] {%
0.654545454545455	0\\
5.34545454545454	0\\
};
\addplot[ybar, bar width=0.145, fill=mycolor4, draw=black, area legend] table[row sep=crcr] {%
1	0.36382\\
2	0.545948\\
3	0.453498\\
4	0.648234\\
5	0.79555\\
};
\addplot[forget plot, color=white!15!black] table[row sep=crcr] {%
0.654545454545455	0\\
5.34545454545454	0\\
};
\end{axis}

\end{tikzpicture}%
	
	\caption{Time per iteration. Time per iteration for the Poisson 3D test case employing the multigrid scheme based on the compatible matching aggregation in conjunction with the different polynomial accelerators and the baseline $\ell_1$-Jacobi for comparison.}
	\label{fig:3d_1024_match_timeperiteration_reduced}
\end{figure}

\subsubsection{Anisotropy 2D}\label{sec:anisotropy2D}
We examine the problem based on the 2D diffusion operator with anisotropy, specifically for an anisotropy angle of $\theta = \pi/6$ and a coefficient value of $\epsilon = 100$. In this case, the best results both in terms of iterations and time to solve are obtained when the multigrid preconditioner setup is based on the VBM classical smoothed aggregation. These experiments are supplemented by additional results in the supplementary materials, which include higher polynomial degrees and multigrid with aggregation based on matching.
{The problem size per GPU is about $4$ million, allowing us to solve problems with over $4$ billion dofs using $1024$ GPUs.}

\paragraph{VBM} As in the previous section, in Fig.~\ref{fig:anisotropy_1024_soc1_tsolve_reduced} we report the time needed to reach the solution of the linear system by keeping the multigrid preconditioner hierarchy fixed and changing the smoother and the polynomial smoother degree (or number of smoothing iterations for the $\ell_1$-Jacobi).
\begin{figure}[htbp]
	\centering
%
%
\definecolor{mycolor1}{rgb}{0.00000,0.44700,0.74100}%
\definecolor{mycolor2}{rgb}{0.85000,0.32500,0.09800}%
\definecolor{mycolor3}{rgb}{0.92900,0.69400,0.12500}%
\definecolor{mycolor4}{rgb}{0.49400,0.18400,0.55600}%
\begin{tikzpicture}

\begin{axis}[%
width=0.33\columnwidth,
height=0.18\columnwidth,
at={(0\columnwidth,0.466\columnwidth)},
scale only axis,
bar shift auto,
xmin=0.4,
xmax=5.6,
xtick={1,2,3,4,5},
xticklabels={{64},{128},{256},{512},{1024}},ticklabel style = {font=\scriptsize},
xlabel style={font=\scriptsize},
xlabel={Number of GPUs},
ymin=0,
ymax=295.042,
ybar=2*\pgflinewidth,
ylabel style={font=\scriptsize},
ylabel={Solve time (s)},
axis background/.style={fill=white},
title style={font=\scriptsize,yshift=-0.8em},
title={$k = 4$ degree/$\ell_1$-Jacobi sweeps},
xmajorgrids,
ymajorgrids,
legend columns=4,
legend style={at={(-0.3,1.4)}, anchor=north west, legend cell align=left, align=left, draw=none, fill=none, font=\scriptsize}
]
\addplot[ybar, bar width=0.145, fill=mycolor1, draw=black, area legend] table[row sep=crcr] {%
1	30.1411\\
2	44.9588\\
3	58.2224\\
4	96.3833\\
5	156.201\\
};
\addplot[forget plot, color=white!15!black] table[row sep=crcr] {%
0.654545454545455	0\\
5.34545454545454	0\\
};
\addlegendentry{Chebyshev 4}

\addplot[ybar, bar width=0.145, fill=mycolor2, draw=black, area legend] table[row sep=crcr] {%
1	30.2535\\
2	42.6847\\
3	58.8012\\
4	96.4112\\
5	150.529\\
};
\addplot[forget plot, color=white!15!black] table[row sep=crcr] {%
0.654545454545455	0\\
5.34545454545454	0\\
};
\addlegendentry{Opt. Chebyshev 4}

\addplot[ybar, bar width=0.145, fill=mycolor3, draw=black, area legend] table[row sep=crcr] {%
1	29.5266\\
2	44.2639\\
3	59.34\\
4	96.0507\\
5	158.523\\
};
\addplot[forget plot, color=white!15!black] table[row sep=crcr] {%
0.654545454545455	0\\
5.34545454545454	0\\
};
\addlegendentry{Opt. Chebyshev 1}

\addplot[ybar, bar width=0.145, fill=mycolor4, draw=black, area legend] table[row sep=crcr] {%
1	34.5084\\
2	47.7808\\
3	59.4805\\
4	99.6036\\
5	164.831\\
};
\addplot[forget plot, color=white!15!black] table[row sep=crcr] {%
0.654545454545455	0\\
5.34545454545454	0\\
};
\addlegendentry{$\ell_1$-Jacobi}

\end{axis}

\begin{axis}[%
width=0.33\columnwidth,
height=0.18\columnwidth,
at={(0.378\columnwidth,0.466\columnwidth)},
scale only axis,
bar shift auto,
ybar=2*\pgflinewidth,
xmin=0.4,
xmax=5.6,
xtick={1,2,3,4,5},
xticklabels={{64},{128},{256},{512},{1024}},ticklabel style = {font=\scriptsize},
xlabel style={font=\scriptsize},
xlabel={Number of GPUs},
ymin=0,
ymax=295.042,
axis background/.style={fill=white},
title style={font=\scriptsize,yshift=-0.8em},
title={$k = 6$ degree/$\ell_1$-Jacobi sweeps},
xmajorgrids,
ymajorgrids
]
\addplot[ybar, bar width=0.145, fill=mycolor1, draw=black, area legend] table[row sep=crcr] {%
1	37.6931\\
2	53.9501\\
3	63.6195\\
4	117.123\\
5	175.691\\
};
\addplot[forget plot, color=white!15!black] table[row sep=crcr] {%
0.654545454545455	0\\
5.34545454545454	0\\
};
\addplot[ybar, bar width=0.145, fill=mycolor2, draw=black, area legend] table[row sep=crcr] {%
1	35.7975\\
2	50.807\\
3	65.0002\\
4	109.008\\
5	177.742\\
};
\addplot[forget plot, color=white!15!black] table[row sep=crcr] {%
0.654545454545455	0\\
5.34545454545454	0\\
};
\addplot[ybar, bar width=0.145, fill=mycolor3, draw=black, area legend] table[row sep=crcr] {%
1	36.6786\\
2	51.8935\\
3	69.8769\\
4	112.191\\
5	188.397\\
};
\addplot[forget plot, color=white!15!black] table[row sep=crcr] {%
0.654545454545455	0\\
5.34545454545454	0\\
};
\addplot[ybar, bar width=0.145, fill=mycolor4, draw=black, area legend] table[row sep=crcr] {%
1	39.027\\
2	58.439\\
3	74.4741\\
4	120.732\\
5	221.694\\
};
\addplot[forget plot, color=white!15!black] table[row sep=crcr] {%
0.654545454545455	0\\
5.34545454545454	0\\
};
\end{axis}

\end{tikzpicture}%
	\caption{Solve time (s). Time to solution (s) for the Anisotropy 2D test case $(\varepsilon=100,\theta=\pi/6)$ employing the multigrid scheme based on the VBM aggregation in conjunction with the different polynomial accelerators and the baseline $\ell_1$-Jacobi for comparison.}
	\label{fig:anisotropy_1024_soc1_tsolve_reduced}
\end{figure}
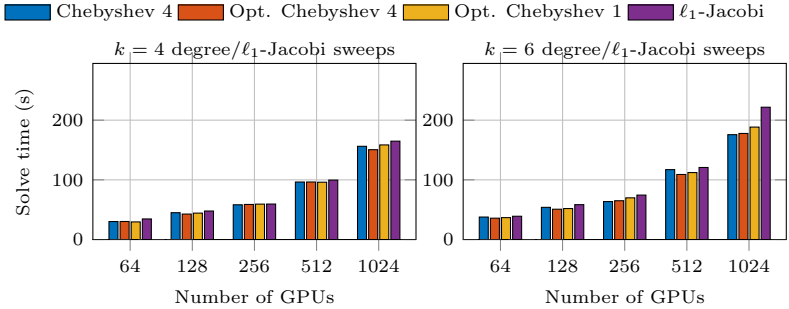
As for the previous test case, using the polynomial accelerations generally reduces the solve times although the percentage of the reduction is smaller ({around $10 \%$ for the best case}) than in the results discussed in the previous section. This demonstrates a smaller impact of the quality of the smoother in this type of preconditioner based on classical smoothed aggregation, as also shown in the results on the Poisson 3D test case reported in the supplementary materials.
The time needed to reach a solution is always lower for the preconditioner using the lightest smoother, i.e., the one with the smallest degree of the polynomial acceleration or, equivalently, the smallest number of matrix-vector products. It is generally obtained when the optimized 4\textsuperscript{th}-kind polynomials are applied when the number of GPUs increases. Indeed, if we turn our attention to Fig.~\ref{fig:anisotropy_1024_soc1_itercount_reduced}, the version based on the optimized 4\textsuperscript{th}-kind polynomials produces the smallest number of iterations for both the polynomial degrees, while, as expected, the quasi-optimal 1\textsuperscript{st}-kind ones are comparable to the optimized 4\textsuperscript{th}-kind polynomials when the lowest degree is used. Applying polynomials accelerators with a larger degree rewards in terms of number of iteration to obtain the given accuracy, but does not pay in terms of total solve times, i.e., the reduction in the number of iterations is not sufficient to balance the cost of the increasing application cost per iteration as shown in Fig.~\ref{fig:anisotropy_1024_soc1_timeperiteration_reduced}. This cost is generally consistent with that of applying a corresponding number of iterations of the $\ell_1$-Jacobi method. That is, we can confirm that from a node-level efficiency perspective and thanks to the careful implementation of shared memory accesses and GPU threads, the use of polynomial smoothers with better approximation capability is achieved practically at a cost that falls within the machine's oscillations.
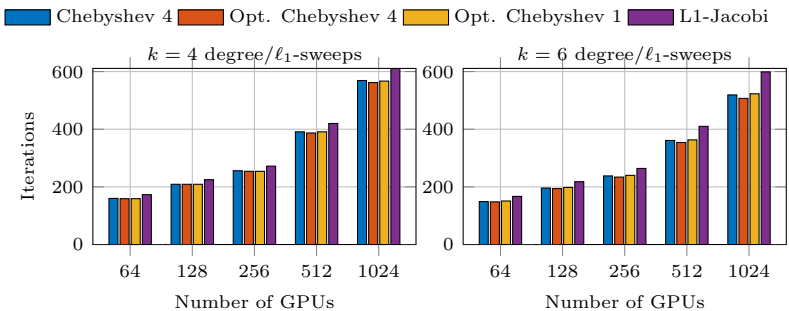
\begin{figure}[htbp]
	\centering
%
%
\definecolor{mycolor1}{rgb}{0.00000,0.44700,0.74100}%
\definecolor{mycolor2}{rgb}{0.85000,0.32500,0.09800}%
\definecolor{mycolor3}{rgb}{0.92900,0.69400,0.12500}%
\definecolor{mycolor4}{rgb}{0.49400,0.18400,0.55600}%
\begin{tikzpicture}

\begin{axis}[%
width=0.33\columnwidth,
height=0.18\columnwidth,
at={(0\columnwidth,0.463\columnwidth)},
scale only axis,
bar shift auto,
xmin=0.4,
xmax=5.6,
xtick={1,2,3,4,5},
xticklabels={{64},{128},{256},{512},{1024}},ticklabel style = {font=\scriptsize},
xlabel style={font=\scriptsize},
ybar=2*\pgflinewidth,
xlabel style={font=\scriptsize},
xlabel={Number of GPUs},
ymin=0,
ymax=611,
ylabel style={font=\scriptsize},
ylabel={Iterations},
axis background/.style={fill=white},
title style={font=\scriptsize,yshift=-0.8em},
title={$k=4$ degree/$\ell_1$-sweeps},
xmajorgrids,
ymajorgrids,
legend columns=4,
legend style={at={(-0.3,1.4)}, anchor=north west, legend cell align=left, align=left, draw=none, fill=none, font=\scriptsize}
]
\addplot[ybar, bar width=0.145, fill=mycolor1, draw=black, area legend] table[row sep=crcr] {%
1	160\\
2	209\\
3	256\\
4	391\\
5	569\\
};
\addplot[forget plot, color=white!15!black] table[row sep=crcr] {%
0.654545454545455	0\\
5.34545454545454	0\\
};
\addlegendentry{Chebyshev 4}

\addplot[ybar, bar width=0.145, fill=mycolor2, draw=black, area legend] table[row sep=crcr] {%
1	159\\
2	209\\
3	254\\
4	387\\
5	562\\
};
\addplot[forget plot, color=white!15!black] table[row sep=crcr] {%
0.654545454545455	0\\
5.34545454545454	0\\
};
\addlegendentry{Opt. Chebyshev 4}

\addplot[ybar, bar width=0.145, fill=mycolor3, draw=black, area legend] table[row sep=crcr] {%
1	159\\
2	209\\
3	254\\
4	391\\
5	567\\
};
\addplot[forget plot, color=white!15!black] table[row sep=crcr] {%
0.654545454545455	0\\
5.34545454545454	0\\
};
\addlegendentry{Opt. Chebyshev 1}

\addplot[ybar, bar width=0.145, fill=mycolor4, draw=black, area legend] table[row sep=crcr] {%
1	173\\
2	225\\
3	272\\
4	420\\
5	611\\
};
\addplot[forget plot, color=white!15!black] table[row sep=crcr] {%
0.654545454545455	0\\
5.34545454545454	0\\
};
\addlegendentry{L1-Jacobi}

\end{axis}

\begin{axis}[%
width=0.33\columnwidth,
height=0.18\columnwidth,
at={(0.378\columnwidth,0.463\columnwidth)},
scale only axis,
bar shift auto,
xmin=0.4,
xmax=5.6,
xtick={1,2,3,4,5},
xticklabels={{64},{128},{256},{512},{1024}},ticklabel style = {font=\scriptsize},
xlabel style={font=\scriptsize},
xlabel={Number of GPUs},
ybar=2*\pgflinewidth,
ymin=0,
ymax=611,
axis background/.style={fill=white},
title style={font=\scriptsize,yshift=-0.8em},
title={$k=6$ degree/$\ell_1$-sweeps},
xmajorgrids,
ymajorgrids
]
\addplot[ybar, bar width=0.145, fill=mycolor1, draw=black, area legend] table[row sep=crcr] {%
1	149\\
2	196\\
3	238\\
4	361\\
5	519\\
};
\addplot[forget plot, color=white!15!black] table[row sep=crcr] {%
0.654545454545455	0\\
5.34545454545454	0\\
};
\addplot[ybar, bar width=0.145, fill=mycolor2, draw=black, area legend] table[row sep=crcr] {%
1	148\\
2	194\\
3	234\\
4	354\\
5	507\\
};
\addplot[forget plot, color=white!15!black] table[row sep=crcr] {%
0.654545454545455	0\\
5.34545454545454	0\\
};
\addplot[ybar, bar width=0.145, fill=mycolor3, draw=black, area legend] table[row sep=crcr] {%
1	151\\
2	198\\
3	240\\
4	363\\
5	523\\
};
\addplot[forget plot, color=white!15!black] table[row sep=crcr] {%
0.654545454545455	0\\
5.34545454545454	0\\
};
\addplot[ybar, bar width=0.145, fill=mycolor4, draw=black, area legend] table[row sep=crcr] {%
1	167\\
2	218\\
3	264\\
4	410\\
5	599\\
};
\addplot[forget plot, color=white!15!black] table[row sep=crcr] {%
0.654545454545455	0\\
5.34545454545454	0\\
};
\end{axis}

\end{tikzpicture}%
	\caption{Iteration count. Total number of CG iterations for the Anisotropy 2D test case $(\varepsilon=100,\theta=\pi/6)$ employing the multigrid scheme based on the VBM aggregation in conjunction with the different polynomial accelerators and the baseline $\ell_1$-Jacobi for comparison.}
	\label{fig:anisotropy_1024_soc1_itercount_reduced}
\end{figure}
\begin{figure}[htbp]
	\centering
%
%
\definecolor{mycolor1}{rgb}{0.00000,0.44700,0.74100}%
\definecolor{mycolor2}{rgb}{0.85000,0.32500,0.09800}%
\definecolor{mycolor3}{rgb}{0.92900,0.69400,0.12500}%
\definecolor{mycolor4}{rgb}{0.49400,0.18400,0.55600}%
\begin{tikzpicture}

\begin{axis}[%
width=0.33\columnwidth,
height=0.18\columnwidth,
at={(0\columnwidth,0.467\columnwidth)},
scale only axis,
bar shift auto,
xmin=0.4,
xmax=5.6,
xtick={1,2,3,4,5},
xticklabels={{64},{128},{256},{512},{1024}},ticklabel style = {font=\scriptsize},
xlabel style={font=\scriptsize},
xlabel={Number of GPUs},
ymin=0,
ymax=0.535495,
ybar=2*\pgflinewidth,
ylabel style={font=\scriptsize},
ylabel={Time per iteration (s)},
axis background/.style={fill=white},
title style={font=\scriptsize,yshift=-0.8em},
title={$k = 4$ degree/$\ell_1$-Jacobi sweeps},
xmajorgrids,
ymajorgrids,
legend columns=4,
legend style={at={(-0.3,1.4)}, anchor=north west, legend cell align=left, align=left, draw=none, fill=none, font=\scriptsize}
]
\addplot[ybar, bar width=0.145, fill=mycolor1, draw=black, area legend] table[row sep=crcr] {%
1	0.188382\\
2	0.215114\\
3	0.227431\\
4	0.246505\\
5	0.274518\\
};
\addplot[forget plot, color=white!15!black] table[row sep=crcr] {%
0.654545454545455	0\\
5.34545454545454	0\\
};
\addlegendentry{Chebyshev 4}

\addplot[ybar, bar width=0.145, fill=mycolor2, draw=black, area legend] table[row sep=crcr] {%
1	0.190274\\
2	0.204233\\
3	0.231501\\
4	0.249125\\
5	0.267846\\
};
\addplot[forget plot, color=white!15!black] table[row sep=crcr] {%
0.654545454545455	0\\
5.34545454545454	0\\
};
\addlegendentry{Opt. Chebyshev 4}

\addplot[ybar, bar width=0.145, fill=mycolor3, draw=black, area legend] table[row sep=crcr] {%
1	0.185702\\
2	0.211789\\
3	0.233622\\
4	0.245654\\
5	0.279582\\
};
\addplot[forget plot, color=white!15!black] table[row sep=crcr] {%
0.654545454545455	0\\
5.34545454545454	0\\
};
\addlegendentry{Opt. Chebyshev 1}

\addplot[ybar, bar width=0.145, fill=mycolor4, draw=black, area legend] table[row sep=crcr] {%
1	0.199471\\
2	0.212359\\
3	0.218678\\
4	0.237151\\
5	0.269773\\
};
\addplot[forget plot, color=white!15!black] table[row sep=crcr] {%
0.654545454545455	0\\
5.34545454545454	0\\
};
\addlegendentry{L1-Jacobi}

\end{axis}

\begin{axis}[%
width=0.33\columnwidth,
height=0.18\columnwidth,
at={(0.378\columnwidth,0.467\columnwidth)},
scale only axis,
bar shift auto,
xmin=0.4,
xmax=5.6,
xtick={1,2,3,4,5},
xticklabels={{64},{128},{256},{512},{1024}},ticklabel style = {font=\scriptsize},
ymin=0,
ymax=0.535495,
ybar=2*\pgflinewidth,
ylabel style={font=\scriptsize},
xlabel style={font=\scriptsize},
xlabel={Number of GPUs},
axis background/.style={fill=white},
title style={font=\scriptsize,yshift=-0.8em},
title={$k = 6$ degree/$\ell_1$-Jacobi sweeps},
xmajorgrids,
ymajorgrids
]
\addplot[ybar, bar width=0.145, fill=mycolor1, draw=black, area legend] table[row sep=crcr] {%
1	0.252974\\
2	0.275256\\
3	0.267309\\
4	0.324439\\
5	0.338519\\
};
\addplot[forget plot, color=white!15!black] table[row sep=crcr] {%
0.654545454545455	0\\
5.34545454545454	0\\
};
\addplot[ybar, bar width=0.145, fill=mycolor2, draw=black, area legend] table[row sep=crcr] {%
1	0.241875\\
2	0.261892\\
3	0.277779\\
4	0.307932\\
5	0.350576\\
};
\addplot[forget plot, color=white!15!black] table[row sep=crcr] {%
0.654545454545455	0\\
5.34545454545454	0\\
};
\addplot[ybar, bar width=0.145, fill=mycolor3, draw=black, area legend] table[row sep=crcr] {%
1	0.242905\\
2	0.262088\\
3	0.291154\\
4	0.309067\\
5	0.360224\\
};
\addplot[forget plot, color=white!15!black] table[row sep=crcr] {%
0.654545454545455	0\\
5.34545454545454	0\\
};
\addplot[ybar, bar width=0.145, fill=mycolor4, draw=black, area legend] table[row sep=crcr] {%
1	0.233695\\
2	0.268069\\
3	0.282099\\
4	0.294468\\
5	0.370107\\
};
\addplot[forget plot, color=white!15!black] table[row sep=crcr] {%
0.654545454545455	0\\
5.34545454545454	0\\
};
\end{axis}

\end{tikzpicture}%
	\caption{Time per iteration. Time per iteration for the Anisotropy 2D test case $(\varepsilon=100,\theta=\pi/6)$ employing the multigrid scheme based on the VBM aggregation in conjunction with the different polynomial accelerators and the baseline $\ell_1$-Jacobi for comparison.}
	\label{fig:anisotropy_1024_soc1_timeperiteration_reduced}
\end{figure}

\section{Conclusion and Future Work}\label{sec:conclusion}
In conclusion, our results demonstrate that polynomial smoothers based on Chebyshev polynomials constitute an effective alternative to traditional smoothers, offering notable improvements in GPU performance and scalability. Although optimizing the convergence bound using 1\textsuperscript{st}-kind Chebyshev polynomials only achieves slightly superior convergence properties for low-degree polynomial accelerators compared to 4\textsuperscript{th}-kind Chebyshev polynomials, {this improvement is meaningful: low-degree polynomials require fewer SpMV operations and thus incur lower computational overhead, enhancing practical efficiency. Consequently, the proposed 1\textsuperscript{st}-kind Chebyshev smoothers combine several key advantages over existing methods—enhanced convergence at low degrees, independence from matrix eigenvalue estimates, and computational efficiency in parallel implementations. These features make them an appealing choice for performance-critical applications and a promising solution for the exascale era.}
From an implementation perspective, we illustrated how efficient memory management on GPUs allows the application cost of a $k$-degree polynomial accelerator, regardless of the type selection, to match that of $k$ applications of the basic smoother, despite the increased number of \texttt{axpy} operations.

To further extend the applicability of our methods, we plan to investigate solving minimax problems in the complex plane, targeting non-symmetric linear systems that arise from discretizing sectorial differential operators. Additionally, we anticipate that using inter-vector operations in our kernels will enable us to leverage hardware-agnostic technologies, achieving similar performance with lower maintenance costs. These future endeavors will enhance the robustness and versatility of our approach in various high-performance computing scenarios.

\bmhead{Acknowledgements}
FD and SM acknowledge the MUR Excellence Department Project awarded to the Department of Mathematics, University of Pisa, CUP I57G22000700001. PD, FD and SF are members of INdAM-GNCS. We thank the referees for their valuable comments and suggestions, which improved the
clarity and presentation of the manuscript.

\section*{Declarations}

\bmhead{Funding} This work was partially supported by: Spoke 1 ``FutureHPC \& BigData'' and Spoke 6 ``Multiscale Modelling \& Engineering Applications'' of the Italian Research Center on High-Performance Computing, Big Data and Quantum Computing (ICSC) funded by MUR Missione 4 Componente 2 Investimento 1.4: Potenziamento strutture di ricerca e creazione di ``campioni nazionali di R\&S (M4C2-19)'' - Next Generation EU (NGEU); the ``Energy Oriented Center of Excellence (EoCoE III): Fostering the European Energy Transition with Exascale" EuroHPC Project N. 101144014, funded by European Commission (EC); the ``INdAM-GNCS Project: Metodi basati su matrici e tensori strutturati per problemi di algebra lineare di grandi dimensioni'', CUP E53C22001930001; the European Union - NextGenerationEU under the National Recovery and Resilience Plan (PNRR) - Mission 4 Education and research - Component 2 from research to business - Investment 1.1 Notice Prin 2022 - DD N. 104 2/2/2022, titled ``Low-rank Structures and Numerical Methods in Matrix and Tensor Computations and their Application'', proposal code 20227PCCKZ -- CUP I53D23002280006. 

\bmhead{Conflict of interest} The authors declare that they have no conflict of interest.
\bmhead{Data availability} Data sharing not is applicable to this article as no new datasets were generated
during the current study.
\bmhead{Code availability} The code, execution scripts and the pointers to the \texttt{PSCToolkit} packages versions used are available on the GitHub repository   \href{https://github.com/Cirdans-Home/polynomialsmoothers}{Cirdans-Home/polynomialsmoothers}.
\bmhead{Author contribution} All authors contributed equally to the paper.

%
%






\begin{appendices}

\section{Proof of Lemma~\ref{lem:monotonicity}, and Lemma~\ref{lem:interlacing}}
\label{appendix}
Here we report the proofs of the two lemmas concerning the minimax approximation with Chebyshev polynomials of the 1\textsuperscript{st} kind.

\begin{proof}[Proof of Lemma~\ref{lem:monotonicity}]
	Observe that we can write $f(x)= f_1(x) \cdot f_2(x)$ where
	$$
	f_1(x):=\frac{x\tau_k^{[a, 1]}(x)}{1-\tau_k^{[a, 1]}(x)},\qquad f_2(x):=\frac{\tau_k^{[a, 1]}(x)}{1+\tau_k^{[a, 1]}(x)},
	$$
	and $f_1,f_2$ are both positive functions on $(0,a]$. Moreover, $f_2(x)^{-1}= 1+\frac{1}{\tau_k^{[a, 1]}(x)}$ is monotonically increasing on $(0,a]$ and this implies that $f_2$ is monotonically decreasing on $(0,a]$. We will show that also $f_1$ is monotonically decreasing and this will give us the claim as the product of non negative decreasing functions is also decreasing.
	
	Let us indicate with $p(x):= \tau_k(\frac{1+a-2x}{1-a})$ and $c_k:=\tau_k(\frac{1+a}{1-a})$, so that $\tau_k^{[a, 1]}(x)= p(x)/c_k$. Studying the negativity of the derivative of $f_1$ yields
	\begin{align*}
		\frac{df_1}{dx}(x)&=\frac{\tau_k^{[a, 1]}(x)\left(1-\tau_k^{[a, 1]}(x)\right)+x(\tau_k^{[a, 1]})'(x)}{\left(1- \tau_k^{[a, 1]}(x)\right)^2}<0 &\text{ for $x\in(0,a]$}\\ &\Leftrightarrow \tau_k^{[a, 1]}(x)\left(1-\tau_k^{[a, 1]}(x)\right)+x(\tau_k^{[a, 1]})'(x)<0&\text{ for $x\in(0,a]$}\\
		& \Leftrightarrow p(x)\left(1- \frac{p(x)^2}{c_k}\right)+xp'(x)<0&\text{ for $x\in(0,a]$}\\
		&\Leftrightarrow 1-\frac{p(x)}{c_k}+x\frac{p'(x)}{p(x)}<0&\text{ for $x\in(0,a]$},
	\end{align*}
	where we have used that $c_k>0$ and that $p(x)>0$ on $(0, a]$. 
	
	We now show that the function $-\frac{p(x)}{c_k}+x\frac{p'(x)}{p(x)}$ is monotonically decreasing and this implies the claim since $1-\frac{p(x)}{c_k}+x\frac{p'(x)}{p(x)}$ takes the value $0$ at $x=0$. Taking the derivative yields
	\begin{align*}
		\frac{d}{dx}\left(-\frac{p(x)}{c_k}+x\frac{p'(x)}{p(x)}\right)&=\frac{x(p''(x)p(x)-p'(x)^2)}{p(x)^2}<0&\text{ for $x\in(0,a]$}\\
		&\Leftrightarrow p''(x)p(x)-p'(x)^2<0&\text{ for $x\in(0,a]$}\\
		&\Leftrightarrow \frac{p''(x)}{p'(x)}-\frac{p'(x)}{p(x)}>0&\text{ for $x\in(0,a]$}\\
		&\Leftrightarrow \frac{d}{dx}\left(\log\left(-\frac{p'(x)}{p(x)}\right)\right)>0&\text{ for $x\in(0,a]$}\\
		&\Leftrightarrow \log\left(-\frac{p'(x)}{p(x)}\right) \text{is increasing} &\text{ for $x\in(0,a]$},
	\end{align*}
	where we have used that $p(x)>0$ and $p'(x)<0$ on $(0, a]$.
	Since the logarithm is increasing, it is sufficient to show that $-\frac{p'(x)}{p(x)}$ is also increasing. Note that $\frac{p'(x)}{p(x)}= \frac{d}{dx}(\log(p(x)))$ and that
	$$\log(p(x))=\log\left(c_k\prod_{j=1}^k(x_j-x)\right)=\log(c_k)+\sum_{j=1}^k \log(x_j-x),$$ where the quantities $x_j\in(a,1)$ are the roots of $p(x)$, for $j=1,\dots,k$. In particular, $\log(p(x))$ is a concave function (being the sum of concave functions) and this yields
	$$
	0>-\frac{d^2}{dx}\left(\log(p(x)\right)= \frac{d}{dx}\left(-\frac{p'(x)}{p(x)}\right),
	$$
	that implies the claim.
\end{proof}

\begin{proof}[Proof of Lemma~\ref{lem:interlacing}]
	Let us define the polynomials $p(x)= \sum_{j=0}^{k}\binom{2k}{2j}x^{j}$ with roots $\alpha_1,\dots,\alpha_k$, and $q(x)=\sum_{j=0}^{k-1} \binom{2k}{2j+1}x^{j}$ with roots $\delta_1,\dots, \delta_{k-1}$. We begin by showing that the roots of the $p(x)$ and $q(x)$ are all negative real and verify the interlacing property:
	\begin{equation}\label{eq:interlacing}
		0>\alpha_1>\delta_1>\alpha_2>\delta_2>\dots>\alpha_{k-1}>\delta_{k-1}>\alpha_k. 
	\end{equation}
	Note that if $\gamma$ is a root of $p(x^2)$ then $\gamma^2$ is equal to $\alpha_i$ for a certain $i=1,\dots,k$. Moreover, we have
	$$
	p(x^2) = \sum_{j=0}^{k}\binom{2k}{2j}x^{2j}=\frac{1}{2}[(1+x)^{2k}+(1-x)^{2k}],
	$$
	so that $p(x^2)=0$ if and only if $|1+x|=|1-x|$. The latter implies that all the roots are of the form $x=\mathbf i b$ with $b\in\mathbb R$. Then,
	\begin{align*}
		p((ib)^2)&= \mathrm{Re}\left((1+\mathbf i b)^{2k}\right)=0\\
		&\Leftrightarrow \mathrm{Arg}\left(1+\mathbf i b\right)=\frac{(2j+1)\pi}{4k},\quad j=0,\dots,2k-1,\\
		&\Leftrightarrow \arctan(b)= \frac{(2j+1)\pi}{4k},\quad j=0,\dots,2k-1,\\
		&\Rightarrow \text{the roots of $p(x)$ are } \alpha_j=-\tan\left(\frac{(2j+1)\pi}{4k}\right)^2, \quad j=0,\dots,k-1.
	\end{align*}
	Similarly, we have
	$$
	xq(x^2) =\sum_{j=0}^{k-1} \binom{2k}{2j+1}x^{2j+1}=\frac{1}{2}[(1+x)^{2k}-(1-x)^{2k}],
	$$
	and if $\gamma$ is a nonzero root of $xq(x^2)$ then $\gamma^2=\delta_j$ for a certain $j=1,\dots,k-1$. For the same reason of $p(x^2)$, $xq(x^2)$ has only purely immaginary roots of the form $x=\mathbf i b$ and
	\begin{align*}
		(\mathbf i b)\cdot q((ib)^2)&= \mathrm{Im}\left((1+\mathbf i b)^{2k}\right)=0\\
		&\Leftrightarrow \mathrm{Arg}\left(1+\mathbf i b\right)=\frac{j\pi}{2k},\quad j=0,\dots,2k-1,\\
		&\Leftrightarrow \arctan(b)= \frac{j\pi}{2k},\quad j=0,\dots,2k-1,\\
		&\Rightarrow \text{the roots of $q(x)$ are } \delta_j=-\tan\left(\frac{j\pi}{2k}\right)^2, \quad j=1,\dots,k-1.
	\end{align*}
	We remark that the values $\alpha_j$ and $\delta_j$ verify \eqref{eq:interlacing}.
	
	To conclude, we look at the first derivative of $2k \cdot g(x)=\frac{p(x)}{q(x)}$ over $(0, 1)$:
	\begin{align*}
		2k g'(x)&=\frac{p'(x)q(x)-p(x)q'(x)}{q(x)^2}>0 \qquad \text{for $x\in(0,1)$}\\
		&\Leftrightarrow\frac{p'(x)}{p(x)}-\frac{q'(x)}{q(x)}>0 \qquad \text{for $x\in(0,1)$}\\
		&\Leftrightarrow\frac{d}{dx}\left(\log\left(\frac{p(x)}{q(x)}\right)\right)>0\qquad \text{for $x\in(0,1)$}\\
		&\Leftrightarrow \sum_{j=1}^{k-1}\frac{d}{dx}\left(\log\left(\frac{x-\alpha_j}{x-\delta_j}\right)\right) + \frac{d}{dx}(\log(x-\alpha_k))>0\qquad \text{for $x\in(0,1)$}\\
		&\Leftrightarrow \sum_{j=1}^{k-1}\frac{\alpha_j-\delta_j}{(x-\alpha_j)(x-\delta_j)} +\frac{1}{x-\alpha_k}>0\qquad \text{for $x\in(0,1)$}.
	\end{align*}
	The latter relation is implied by \eqref{eq:interlacing}, and this gives the claim.
\end{proof}

%

\section{Additional Details on the Numerical experiments in Section~\ref{sec:results}}\label{sm:additional_details}
The experimental setting is exactly the same, both from the computing point of view and from the point of view of the software requirements. The sections are mirrors of those contained in the paper.

\subsection{Poisson 3D}

The following two sections contain additional information regarding the Poisson 3D problem. They extend the results contained in Section~\ref{sec:poisson3D} adding to the considerations the effect of the smoothers for a higher degree of the polynomial and the use of the compatible matching aggregation.

\subsubsection{Compatible Matching}\label{sec:sm_match_poisson3d}
In Fig.~\ref{fig:3d_1024_match_tsolve_reduced} we reported the time needed to reach the solution of the linear system by keeping the multigrid preconditioner hierarchy fixed and changing the smoother and between the different types of polynomial acceleration and with respect to the degree--number of vector matrix products used per iteration; we extend the results in Fig.~\ref{fig:3d_1024_match_tsolve} by considering polynomial of higher degree while plotting all the figures on the same scale.
\begin{figure}[htbp]
	\centering
%
%
\definecolor{mycolor1}{rgb}{0.00000,0.44700,0.74100}%
\definecolor{mycolor2}{rgb}{0.85000,0.32500,0.09800}%
\definecolor{mycolor3}{rgb}{0.92900,0.69400,0.12500}%
\definecolor{mycolor4}{rgb}{0.49400,0.18400,0.55600}%
\begin{tikzpicture}

\begin{axis}[%
width=0.33\columnwidth,
height=0.168\columnwidth,
at={(0\columnwidth,0.52\columnwidth)},
scale only axis,
bar shift auto,
ybar=2*\pgflinewidth,
xmin=0.4,
xmax=5.6,
xtick={1,2,3,4,5},
xticklabels={{64},{128},{256},{512},{1024}},ticklabel style = {font=\scriptsize},
ymin=0,
ymax=20.3605,
ylabel style={font=\scriptsize},
ylabel={Solve time (s)},
axis background/.style={fill=white},
title style={font=\scriptsize,yshift=-0.8em},
title={$k=4$ degree/$\ell_1$-Jacobi sweeps},
xmajorgrids,
ymajorgrids,
legend columns=4,
legend style={at={(-0.3,1.4)}, anchor=north west, legend cell align=left, align=left, draw=none, fill=none, font=\scriptsize}
]
\addplot[ybar, bar width=0.145, fill=mycolor1, draw=black, area legend] table[row sep=crcr] {%
1	2.99555\\
2	5.10844\\
3	6.40426\\
4	7.95049\\
5	11.5501\\
};
\addplot[forget plot, color=white!15!black] table[row sep=crcr] {%
0.654545454545455	0\\
5.34545454545454	0\\
};
\addlegendentry{Chebyshev 4}

\addplot[ybar, bar width=0.145, fill=mycolor2, draw=black, area legend] table[row sep=crcr] {%
1	2.6942\\
2	5.02838\\
3	6.00266\\
4	7.77814\\
5	11.384\\
};
\addplot[forget plot, color=white!15!black] table[row sep=crcr] {%
0.654545454545455	0\\
5.34545454545454	0\\
};
\addlegendentry{Opt. Chebyshev 4}

\addplot[ybar, bar width=0.145, fill=mycolor3, draw=black, area legend] table[row sep=crcr] {%
1	2.5222\\
2	5.0013\\
3	5.8837\\
4	7.73889\\
5	11.9796\\
};
\addplot[forget plot, color=white!15!black] table[row sep=crcr] {%
0.654545454545455	0\\
5.34545454545454	0\\
};
\addlegendentry{Opt. Chebyshev 1}

\addplot[ybar, bar width=0.145, fill=mycolor4, draw=black, area legend] table[row sep=crcr] {%
1	3.49991\\
2	5.84694\\
3	6.61824\\
4	8.80121\\
5	14.5746\\
};
\addplot[forget plot, color=white!15!black] table[row sep=crcr] {%
0.654545454545455	0\\
5.34545454545454	0\\
};
\addlegendentry{L1-Jacobi}

\end{axis}

\begin{axis}[%
width=0.33\columnwidth,
height=0.168\columnwidth,
at={(0.378\columnwidth,0.52\columnwidth)},
scale only axis,
bar shift auto,
ybar=2*\pgflinewidth,
xmin=0.4,
xmax=5.6,
xtick={1,2,3,4,5},
xticklabels={{64},{128},{256},{512},{1024}},ticklabel style = {font=\scriptsize},
ymin=0,
ymax=20.3605,
ylabel style={font=\scriptsize},
axis background/.style={fill=white},
title style={font=\scriptsize,yshift=-0.8em},
title={$k=6$ degree/$\ell_1$-Jacobi sweeps},
xmajorgrids,
ymajorgrids
]
\addplot[ybar, bar width=0.145, fill=mycolor1, draw=black, area legend] table[row sep=crcr] {%
1	3.31203\\
2	5.57183\\
3	6.92748\\
4	7.86276\\
5	11.0869\\
};
\addplot[forget plot, color=white!15!black] table[row sep=crcr] {%
0.654545454545455	0\\
5.34545454545454	0\\
};
\addplot[ybar, bar width=0.145, fill=mycolor2, draw=black, area legend] table[row sep=crcr] {%
1	3.30466\\
2	5.53617\\
3	6.45388\\
4	7.86843\\
5	10.3415\\
};
\addplot[forget plot, color=white!15!black] table[row sep=crcr] {%
0.654545454545455	0\\
5.34545454545454	0\\
};
\addplot[ybar, bar width=0.145, fill=mycolor3, draw=black, area legend] table[row sep=crcr] {%
1	3.29101\\
2	5.51514\\
3	6.7484\\
4	7.82342\\
5	12.888\\
};
\addplot[forget plot, color=white!15!black] table[row sep=crcr] {%
0.654545454545455	0\\
5.34545454545454	0\\
};
\addplot[ybar, bar width=0.145, fill=mycolor4, draw=black, area legend] table[row sep=crcr] {%
1	4.00202\\
2	6.55137\\
3	8.16297\\
4	9.72351\\
5	15.911\\
};
\addplot[forget plot, color=white!15!black] table[row sep=crcr] {%
0.654545454545455	0\\
5.34545454545454	0\\
};
\end{axis}

\begin{axis}[%
width=0.33\columnwidth,
height=0.168\columnwidth,
at={(0\columnwidth,0.26\columnwidth)},
scale only axis,
bar shift auto,
ybar=2*\pgflinewidth,
xmin=0.4,
xmax=5.6,
xtick={1,2,3,4,5},
xticklabels={{64},{128},{256},{512},{1024}},ticklabel style = {font=\scriptsize},
ymin=0,
ymax=20.3605,
ylabel style={font=\scriptsize},
ylabel={Solve time (s)},
axis background/.style={fill=white},
title style={font=\scriptsize,yshift=-0.8em},
title={$k=8$ degree/$\ell_1$-Jacobi sweeps},
xmajorgrids,
ymajorgrids
]
\addplot[ybar, bar width=0.145, fill=mycolor1, draw=black, area legend] table[row sep=crcr] {%
1	3.55554\\
2	5.77133\\
3	7.12979\\
4	8.37889\\
5	10.9739\\
};
\addplot[forget plot, color=white!15!black] table[row sep=crcr] {%
0.654545454545455	0\\
5.34545454545454	0\\
};
\addplot[ybar, bar width=0.145, fill=mycolor2, draw=black, area legend] table[row sep=crcr] {%
1	3.55148\\
2	5.83995\\
3	6.58093\\
4	8.46996\\
5	10.9109\\
};
\addplot[forget plot, color=white!15!black] table[row sep=crcr] {%
0.654545454545455	0\\
5.34545454545454	0\\
};
\addplot[ybar, bar width=0.145, fill=mycolor3, draw=black, area legend] table[row sep=crcr] {%
1	3.42057\\
2	5.8209\\
3	6.96293\\
4	8.33664\\
5	12.995\\
};
\addplot[forget plot, color=white!15!black] table[row sep=crcr] {%
0.654545454545455	0\\
5.34545454545454	0\\
};
\addplot[ybar, bar width=0.145, fill=mycolor4, draw=black, area legend] table[row sep=crcr] {%
1	4.83018\\
2	7.66922\\
3	9.13501\\
4	10.9936\\
5	17.1865\\
};
\addplot[forget plot, color=white!15!black] table[row sep=crcr] {%
0.654545454545455	0\\
5.34545454545454	0\\
};
\end{axis}

\begin{axis}[%
width=0.33\columnwidth,
height=0.168\columnwidth,
at={(0.378\columnwidth,0.26\columnwidth)},
scale only axis,
bar shift auto,
ybar=2*\pgflinewidth,
xmin=0.4,
xmax=5.6,
xtick={1,2,3,4,5},
xticklabels={{64},{128},{256},{512},{1024}},ticklabel style = {font=\scriptsize},
xlabel style={font=\scriptsize},
xlabel={Number of GPUs},
ymin=0,
ymax=20.3605,
ylabel style={font=\scriptsize},
axis background/.style={fill=white},
title style={font=\scriptsize,yshift=-0.8em},
title={$k=10$ degree/$\ell_1$-Jacobi sweeps},
xmajorgrids,
ymajorgrids
]
\addplot[ybar, bar width=0.145, fill=mycolor1, draw=black, area legend] table[row sep=crcr] {%
1	4.13251\\
2	6.68876\\
3	6.99996\\
4	8.51958\\
5	11.3681\\
};
\addplot[forget plot, color=white!15!black] table[row sep=crcr] {%
0.654545454545455	0\\
5.34545454545454	0\\
};
\addplot[ybar, bar width=0.145, fill=mycolor2, draw=black, area legend] table[row sep=crcr] {%
1	3.61456\\
2	5.98068\\
3	6.98126\\
4	8.50166\\
5	10.3614\\
};
\addplot[forget plot, color=white!15!black] table[row sep=crcr] {%
0.654545454545455	0\\
5.34545454545454	0\\
};
\addplot[ybar, bar width=0.145, fill=mycolor3, draw=black, area legend] table[row sep=crcr] {%
1	4.07997\\
2	6.66412\\
3	7.42596\\
4	8.44117\\
5	12.4769\\
};
\addplot[forget plot, color=white!15!black] table[row sep=crcr] {%
0.654545454545455	0\\
5.34545454545454	0\\
};
\addplot[ybar, bar width=0.145, fill=mycolor4, draw=black, area legend] table[row sep=crcr] {%
1	5.12613\\
2	8.08032\\
3	10.5244\\
4	11.755\\
5	19.3676\\
};
\addplot[forget plot, color=white!15!black] table[row sep=crcr] {%
0.654545454545455	0\\
5.34545454545454	0\\
};
\end{axis}

\begin{axis}[%
width=0.33\columnwidth,
height=0.168\columnwidth,
at={(0\columnwidth,0\columnwidth)},
scale only axis,
bar shift auto,
ybar=2*\pgflinewidth,
xmin=0.4,
xmax=5.6,
xtick={1,2,3,4,5},
xticklabels={{64},{128},{256},{512},{1024}},ticklabel style = {font=\scriptsize},
xlabel style={font=\scriptsize},
xlabel={Number of GPUs},
ymin=0,
ymax=20.3605,
ylabel style={font=\scriptsize},
ylabel={Solve time (s)},
axis background/.style={fill=white},
title style={font=\scriptsize,yshift=-0.8em},
title={$k=12$ degree/$\ell_1$-Jacobi sweeps},
xmajorgrids,
ymajorgrids
]
\addplot[ybar, bar width=0.145, fill=mycolor1, draw=black, area legend] table[row sep=crcr] {%
1	4.13887\\
2	6.71727\\
3	7.23561\\
4	9.50135\\
5	11.4897\\
};
\addplot[forget plot, color=white!15!black] table[row sep=crcr] {%
0.654545454545455	0\\
5.34545454545454	0\\
};
\addplot[ybar, bar width=0.145, fill=mycolor2, draw=black, area legend] table[row sep=crcr] {%
1	4.1789\\
2	6.73042\\
3	6.48069\\
4	9.47244\\
5	11.7892\\
};
\addplot[forget plot, color=white!15!black] table[row sep=crcr] {%
0.654545454545455	0\\
5.34545454545454	0\\
};
\addplot[ybar, bar width=0.145, fill=mycolor3, draw=black, area legend] table[row sep=crcr] {%
1	3.90646\\
2	6.65484\\
3	7.76683\\
4	9.4001\\
5	12.7408\\
};
\addplot[forget plot, color=white!15!black] table[row sep=crcr] {%
0.654545454545455	0\\
5.34545454545454	0\\
};
\addplot[ybar, bar width=0.145, fill=mycolor4, draw=black, area legend] table[row sep=crcr] {%
1	5.53168\\
2	9.092\\
3	11.2562\\
4	13.2292\\
5	20.3605\\
};
\addplot[forget plot, color=white!15!black] table[row sep=crcr] {%
0.654545454545455	0\\
5.34545454545454	0\\
};
\end{axis}
\end{tikzpicture}%
	
	\caption{Solve time (s). Time to solution (s) for the Poisson 3D test case employing the multigrid scheme based on the compatible matching aggregation in conjunction with the different polynomial accelerators and the baseline $\ell_1$-Jacobi for comparison.}
	\label{fig:3d_1024_match_tsolve}
\end{figure}
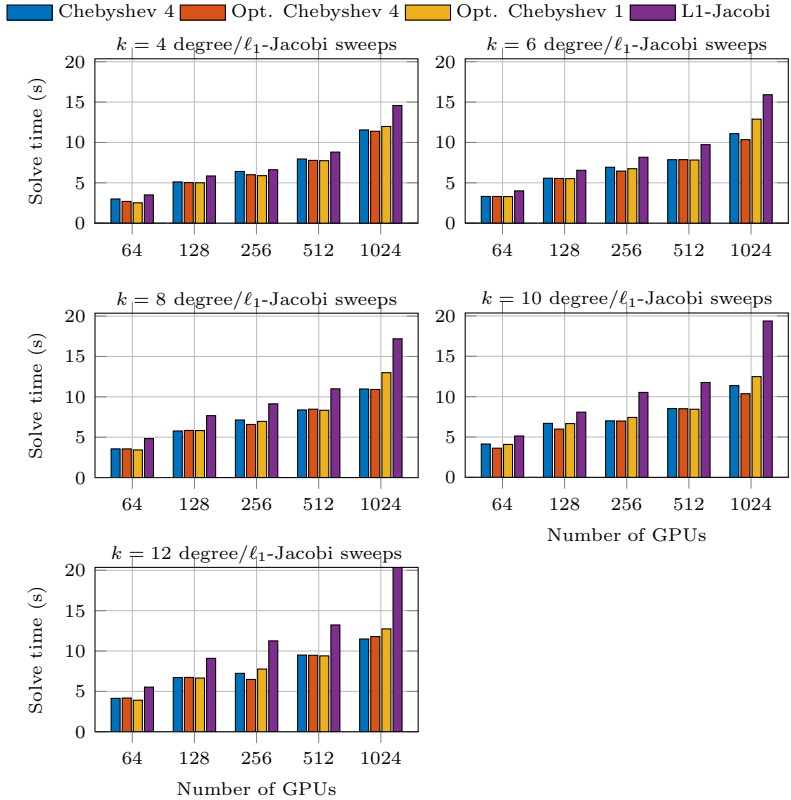
The behavior of the solution times is analogous to that observed for lower degrees in the paper. There is an increase in the total solution time with the increase in the degree of the polynomial used as smoother, however, the increase ratio observed for the polynomial smoothers is smaller than the increase ratio for the baseline $\ell_1$-Jacobi, especially for increasing problem size and number of GPUs. This confirms the better scalability potential of the polynomial accelerations, in the face of the expected reduction in the number of iterations as $k$ increases shown in Fig.~\ref{fig:3d_1024_match_itercount}.
\begin{figure}[htbp]
	\centering
%
%
\definecolor{mycolor1}{rgb}{0.00000,0.44700,0.74100}%
\definecolor{mycolor2}{rgb}{0.85000,0.32500,0.09800}%
\definecolor{mycolor3}{rgb}{0.92900,0.69400,0.12500}%
\definecolor{mycolor4}{rgb}{0.49400,0.18400,0.55600}%
\begin{tikzpicture}

\begin{axis}[%
width=0.33\columnwidth,
height=0.168\columnwidth,
at={(0\columnwidth,0.519\columnwidth)},
scale only axis,
bar shift auto,
ybar=2*\pgflinewidth,
xmin=0.4,
xmax=5.6,
xtick={1,2,3,4,5},
xticklabels={{64},{128},{256},{512},{1024}},ticklabel style = {font=\scriptsize},
ymin=0,
ymax=21,
ylabel style={font=\scriptsize},
ylabel={Iterations},
axis background/.style={fill=white},
title style={font=\scriptsize,yshift=-0.8em},
title={$k=4$ degree/$\ell_1$-Jacobi sweeps},
xmajorgrids,
ymajorgrids,
legend columns=4,
legend style={at={(-0.3,1.4)}, anchor=north west, legend cell align=left, align=left, draw=none, fill=none, font=\scriptsize}
]
\addplot[ybar, bar width=0.145, fill=mycolor1, draw=black, area legend] table[row sep=crcr] {%
1	10\\
2	11\\
3	17\\
4	14\\
5	18\\
};
\addplot[forget plot, color=white!15!black] table[row sep=crcr] {%
0.654545454545455	0\\
5.34545454545454	0\\
};
\addlegendentry{Chebyshev 4}

\addplot[ybar, bar width=0.145, fill=mycolor2, draw=black, area legend] table[row sep=crcr] {%
1	9\\
2	11\\
3	16\\
4	14\\
5	18\\
};
\addplot[forget plot, color=white!15!black] table[row sep=crcr] {%
0.654545454545455	0\\
5.34545454545454	0\\
};
\addlegendentry{Opt. Chebyshev 4}

\addplot[ybar, bar width=0.145, fill=mycolor3, draw=black, area legend] table[row sep=crcr] {%
1	9\\
2	11\\
3	16\\
4	14\\
5	18\\
};
\addplot[forget plot, color=white!15!black] table[row sep=crcr] {%
0.654545454545455	0\\
5.34545454545454	0\\
};
\addlegendentry{Opt. Chebyshev 1}

\addplot[ybar, bar width=0.145, fill=mycolor4, draw=black, area legend] table[row sep=crcr] {%
1	12\\
2	13\\
3	18\\
4	16\\
5	21\\
};
\addplot[forget plot, color=white!15!black] table[row sep=crcr] {%
0.654545454545455	0\\
5.34545454545454	0\\
};
\addlegendentry{L1-Jacobi}

\end{axis}

\begin{axis}[%
width=0.33\columnwidth,
height=0.168\columnwidth,
at={(0.378\columnwidth,0.519\columnwidth)},
scale only axis,
bar shift auto,
ybar=2*\pgflinewidth,
xmin=0.4,
xmax=5.6,
xtick={1,2,3,4,5},
xticklabels={{64},{128},{256},{512},{1024}},ticklabel style = {font=\scriptsize},
ymin=0,
ymax=21,
ylabel style={font=\scriptsize},
axis background/.style={fill=white},
title style={font=\scriptsize,yshift=-0.8em},
title={$k=6$ degree/$\ell_1$-Jacobi sweeps},
xmajorgrids,
ymajorgrids
]
\addplot[ybar, bar width=0.145, fill=mycolor1, draw=black, area legend] table[row sep=crcr] {%
1	9\\
2	10\\
3	15\\
4	12\\
5	15\\
};
\addplot[forget plot, color=white!15!black] table[row sep=crcr] {%
0.654545454545455	0\\
5.34545454545454	0\\
};
\addplot[ybar, bar width=0.145, fill=mycolor2, draw=black, area legend] table[row sep=crcr] {%
1	9\\
2	10\\
3	14\\
4	12\\
5	14\\
};
\addplot[forget plot, color=white!15!black] table[row sep=crcr] {%
0.654545454545455	0\\
5.34545454545454	0\\
};
\addplot[ybar, bar width=0.145, fill=mycolor3, draw=black, area legend] table[row sep=crcr] {%
1	9\\
2	10\\
3	15\\
4	12\\
5	16\\
};
\addplot[forget plot, color=white!15!black] table[row sep=crcr] {%
0.654545454545455	0\\
5.34545454545454	0\\
};
\addplot[ybar, bar width=0.145, fill=mycolor4, draw=black, area legend] table[row sep=crcr] {%
1	11\\
2	12\\
3	18\\
4	15\\
5	20\\
};
\addplot[forget plot, color=white!15!black] table[row sep=crcr] {%
0.654545454545455	0\\
5.34545454545454	0\\
};
\end{axis}

\begin{axis}[%
width=0.33\columnwidth,
height=0.168\columnwidth,
at={(0\columnwidth,0.26\columnwidth)},
scale only axis,
bar shift auto,
ybar=2*\pgflinewidth,
xmin=0.4,
xmax=5.6,
xtick={1,2,3,4,5},
xticklabels={{64},{128},{256},{512},{1024}},ticklabel style = {font=\scriptsize},
ymin=0,
ymax=21,
ylabel style={font=\scriptsize},
ylabel={Iterations},
axis background/.style={fill=white},
title style={font=\scriptsize,yshift=-0.8em},
title={$k=8$ degree/$\ell_1$-Jacobi sweeps},
xmajorgrids,
ymajorgrids
]
\addplot[ybar, bar width=0.145, fill=mycolor1, draw=black, area legend] table[row sep=crcr] {%
1	8\\
2	9\\
3	13\\
4	11\\
5	13\\
};
\addplot[forget plot, color=white!15!black] table[row sep=crcr] {%
0.654545454545455	0\\
5.34545454545454	0\\
};
\addplot[ybar, bar width=0.145, fill=mycolor2, draw=black, area legend] table[row sep=crcr] {%
1	8\\
2	9\\
3	12\\
4	11\\
5	13\\
};
\addplot[forget plot, color=white!15!black] table[row sep=crcr] {%
0.654545454545455	0\\
5.34545454545454	0\\
};
\addplot[ybar, bar width=0.145, fill=mycolor3, draw=black, area legend] table[row sep=crcr] {%
1	8\\
2	9\\
3	13\\
4	11\\
5	14\\
};
\addplot[forget plot, color=white!15!black] table[row sep=crcr] {%
0.654545454545455	0\\
5.34545454545454	0\\
};
\addplot[ybar, bar width=0.145, fill=mycolor4, draw=black, area legend] table[row sep=crcr] {%
1	11\\
2	12\\
3	17\\
4	15\\
5	19\\
};
\addplot[forget plot, color=white!15!black] table[row sep=crcr] {%
0.654545454545455	0\\
5.34545454545454	0\\
};
\end{axis}

\begin{axis}[%
width=0.33\columnwidth,
height=0.168\columnwidth,
at={(0.378\columnwidth,0.26\columnwidth)},
scale only axis,
bar shift auto,
ybar=2*\pgflinewidth,
xmin=0.4,
xmax=5.6,
xtick={1,2,3,4,5},
xticklabels={{64},{128},{256},{512},{1024}},ticklabel style = {font=\scriptsize},
xlabel style={font=\scriptsize},
xlabel={Number of GPUs},
ymin=0,
ymax=21,
ylabel style={font=\scriptsize},
axis background/.style={fill=white},
title style={font=\scriptsize,yshift=-0.8em},
title={$k=10$ degree/$\ell_1$-Jacobi sweeps},
xmajorgrids,
ymajorgrids
]
\addplot[ybar, bar width=0.145, fill=mycolor1, draw=black, area legend] table[row sep=crcr] {%
1	8\\
2	9\\
3	11\\
4	10\\
5	12\\
};
\addplot[forget plot, color=white!15!black] table[row sep=crcr] {%
0.654545454545455	0\\
5.34545454545454	0\\
};
\addplot[ybar, bar width=0.145, fill=mycolor2, draw=black, area legend] table[row sep=crcr] {%
1	7\\
2	8\\
3	11\\
4	10\\
5	11\\
};
\addplot[forget plot, color=white!15!black] table[row sep=crcr] {%
0.654545454545455	0\\
5.34545454545454	0\\
};
\addplot[ybar, bar width=0.145, fill=mycolor3, draw=black, area legend] table[row sep=crcr] {%
1	8\\
2	9\\
3	12\\
4	10\\
5	12\\
};
\addplot[forget plot, color=white!15!black] table[row sep=crcr] {%
0.654545454545455	0\\
5.34545454545454	0\\
};
\addplot[ybar, bar width=0.145, fill=mycolor4, draw=black, area legend] table[row sep=crcr] {%
1	10\\
2	11\\
3	17\\
4	14\\
5	19\\
};
\addplot[forget plot, color=white!15!black] table[row sep=crcr] {%
0.654545454545455	0\\
5.34545454545454	0\\
};
\end{axis}

\begin{axis}[%
width=0.33\columnwidth,
height=0.168\columnwidth,
at={(0\columnwidth,0\columnwidth)},
scale only axis,
bar shift auto,
ybar=2*\pgflinewidth,
xmin=0.4,
xmax=5.6,
xtick={1,2,3,4,5},
xticklabels={{64},{128},{256},{512},{1024}},ticklabel style = {font=\scriptsize},
xlabel style={font=\scriptsize},
xlabel={Number of GPUs},
ymin=0,
ymax=21,
ylabel style={font=\scriptsize},
ylabel={Iterations},
axis background/.style={fill=white},
title style={font=\scriptsize,yshift=-0.8em},
title={$k=12$ degree/$\ell_1$-Jacobi sweeps},
xmajorgrids,
ymajorgrids
]
\addplot[ybar, bar width=0.145, fill=mycolor1, draw=black, area legend] table[row sep=crcr] {%
1	7\\
2	8\\
3	10\\
4	10\\
5	11\\
};
\addplot[forget plot, color=white!15!black] table[row sep=crcr] {%
0.654545454545455	0\\
5.34545454545454	0\\
};
\addplot[ybar, bar width=0.145, fill=mycolor2, draw=black, area legend] table[row sep=crcr] {%
1	7\\
2	8\\
3	9\\
4	10\\
5	11\\
};
\addplot[forget plot, color=white!15!black] table[row sep=crcr] {%
0.654545454545455	0\\
5.34545454545454	0\\
};
\addplot[ybar, bar width=0.145, fill=mycolor3, draw=black, area legend] table[row sep=crcr] {%
1	7\\
2	8\\
3	11\\
4	10\\
5	11\\
};
\addplot[forget plot, color=white!15!black] table[row sep=crcr] {%
0.654545454545455	0\\
5.34545454545454	0\\
};
\addplot[ybar, bar width=0.145, fill=mycolor4, draw=black, area legend] table[row sep=crcr] {%
1	10\\
2	11\\
3	16\\
4	14\\
5	18\\
};
\addplot[forget plot, color=white!15!black] table[row sep=crcr] {%
0.654545454545455	0\\
5.34545454545454	0\\
};
\end{axis}
\end{tikzpicture}%
	
	\caption{Iteration count. Total number of FCG iterations for the Poisson 3D test case employing the multigrid scheme based on the compatible matching aggregation in conjunction with the different polynomial accelerators and the baseline $\ell_1$-Jacobi for comparison.}
	\label{fig:3d_1024_match_itercount}
\end{figure}
However, the improvement in the number of iterations is not sufficient to amortize the cost of the additional sparse matrix-vector products needed to apply the higher degree polynomial. Hence the error reduction effect due to the smoother is again more important with respect to a low degree polynomial than with respect to a higher degree polynomial. To conclude this set of experiments we can look again at the implementation scalability properties of the method by focusing on the iteration time of the different polynomial smoothers in Fig.~\ref{fig:3d_1024_match_timeperiteration}.
\begin{figure}[htbp]
	\centering
%
%
\definecolor{mycolor1}{rgb}{0.00000,0.44700,0.74100}%
\definecolor{mycolor2}{rgb}{0.85000,0.32500,0.09800}%
\definecolor{mycolor3}{rgb}{0.92900,0.69400,0.12500}%
\definecolor{mycolor4}{rgb}{0.49400,0.18400,0.55600}%
\begin{tikzpicture}

\begin{axis}[%
width=0.33\columnwidth,
height=0.168\columnwidth,
at={(0\columnwidth,0.519\columnwidth)},
scale only axis,
bar shift auto,
ybar=2*\pgflinewidth,
xmin=0.3,
xmax=5.6,
xtick={1,2,3,4,5},
xticklabels={{64},{128},{256},{512},{1024}},ticklabel style = {font=\scriptsize},
ymin=0,
ymax=1.15826,
ylabel style={font=\scriptsize},
ylabel={Time per iteration (s)},
axis background/.style={fill=white},
title style={font=\scriptsize,yshift=-0.8em},
title={$k=4$ degree/$\ell_1$-Jacobi sweeps},
xmajorgrids,
ymajorgrids,
legend columns=4,
legend style={at={(-0.3,1.4)}, anchor=north west, legend cell align=left, align=left, draw=none, fill=none, font=\scriptsize}
]
\addplot[ybar, bar width=0.145, fill=mycolor1, draw=black, area legend] table[row sep=crcr] {%
1	0.299555\\
2	0.464404\\
3	0.376721\\
4	0.567892\\
5	0.64167\\
};
\addplot[forget plot, color=white!15!black] table[row sep=crcr] {%
0.654545454545455	0\\
5.34545454545454	0\\
};
\addlegendentry{Chebyshev 4}

\addplot[ybar, bar width=0.145, fill=mycolor2, draw=black, area legend] table[row sep=crcr] {%
1	0.299355\\
2	0.457125\\
3	0.375166\\
4	0.555581\\
5	0.632442\\
};
\addplot[forget plot, color=white!15!black] table[row sep=crcr] {%
0.654545454545455	0\\
5.34545454545454	0\\
};
\addlegendentry{Opt. Chebyshev 4}

\addplot[ybar, bar width=0.145, fill=mycolor3, draw=black, area legend] table[row sep=crcr] {%
1	0.280245\\
2	0.454664\\
3	0.367731\\
4	0.552778\\
5	0.665533\\
};
\addplot[forget plot, color=white!15!black] table[row sep=crcr] {%
0.654545454545455	0\\
5.34545454545454	0\\
};
\addlegendentry{Opt. Chebyshev 1}

\addplot[ybar, bar width=0.145, fill=mycolor4, draw=black, area legend] table[row sep=crcr] {%
1	0.291659\\
2	0.449765\\
3	0.36768\\
4	0.550076\\
5	0.694026\\
};
\addplot[forget plot, color=white!15!black] table[row sep=crcr] {%
0.654545454545455	0\\
5.34545454545454	0\\
};
\addlegendentry{L1-Jacobi}

\end{axis}

\begin{axis}[%
width=0.33\columnwidth,
height=0.168\columnwidth,
at={(0.378\columnwidth,0.519\columnwidth)},
scale only axis,
bar shift auto,
ybar=2*\pgflinewidth,
xmin=0.3,
xmax=5.6,
xtick={1,2,3,4,5},
xticklabels={{64},{128},{256},{512},{1024}},ticklabel style = {font=\scriptsize},
ymin=0,
ymax=1.15826,
axis background/.style={fill=white},
title style={font=\scriptsize,yshift=-0.8em},
title={$k=6$ degree/$\ell_1$-Jacobi sweeps},
xmajorgrids,
ymajorgrids
]
\addplot[ybar, bar width=0.145, fill=mycolor1, draw=black, area legend] table[row sep=crcr] {%
1	0.368003\\
2	0.557183\\
3	0.461832\\
4	0.65523\\
5	0.739126\\
};
\addplot[forget plot, color=white!15!black] table[row sep=crcr] {%
0.654545454545455	0\\
5.34545454545454	0\\
};
\addplot[ybar, bar width=0.145, fill=mycolor2, draw=black, area legend] table[row sep=crcr] {%
1	0.367185\\
2	0.553617\\
3	0.460992\\
4	0.655702\\
5	0.738682\\
};
\addplot[forget plot, color=white!15!black] table[row sep=crcr] {%
0.654545454545455	0\\
5.34545454545454	0\\
};
\addplot[ybar, bar width=0.145, fill=mycolor3, draw=black, area legend] table[row sep=crcr] {%
1	0.365668\\
2	0.551514\\
3	0.449893\\
4	0.651952\\
5	0.805499\\
};
\addplot[forget plot, color=white!15!black] table[row sep=crcr] {%
0.654545454545455	0\\
5.34545454545454	0\\
};
\addplot[ybar, bar width=0.145, fill=mycolor4, draw=black, area legend] table[row sep=crcr] {%
1	0.36382\\
2	0.545948\\
3	0.453498\\
4	0.648234\\
5	0.79555\\
};
\addplot[forget plot, color=white!15!black] table[row sep=crcr] {%
0.654545454545455	0\\
5.34545454545454	0\\
};
\end{axis}

\begin{axis}[%
width=0.33\columnwidth,
height=0.168\columnwidth,
at={(0\columnwidth,0.259\columnwidth)},
scale only axis,
bar shift auto,
ybar=2*\pgflinewidth,
xmin=0.3,
xmax=5.6,
xtick={1,2,3,4,5},
xticklabels={{64},{128},{256},{512},{1024}},ticklabel style = {font=\scriptsize},
ymin=0,
ymax=1.15826,
ylabel style={font=\scriptsize},
ylabel={Time per iteration (s)},
axis background/.style={fill=white},
title style={font=\scriptsize,yshift=-0.8em},
title={$k=8$ degree/$\ell_1$-Jacobi sweeps},
xmajorgrids,
ymajorgrids
]
\addplot[ybar, bar width=0.145, fill=mycolor1, draw=black, area legend] table[row sep=crcr] {%
1	0.444442\\
2	0.641259\\
3	0.548445\\
4	0.761717\\
5	0.844145\\
};
\addplot[forget plot, color=white!15!black] table[row sep=crcr] {%
0.654545454545455	0\\
5.34545454545454	0\\
};
\addplot[ybar, bar width=0.145, fill=mycolor2, draw=black, area legend] table[row sep=crcr] {%
1	0.443935\\
2	0.648884\\
3	0.548411\\
4	0.769996\\
5	0.839302\\
};
\addplot[forget plot, color=white!15!black] table[row sep=crcr] {%
0.654545454545455	0\\
5.34545454545454	0\\
};
\addplot[ybar, bar width=0.145, fill=mycolor3, draw=black, area legend] table[row sep=crcr] {%
1	0.427571\\
2	0.646767\\
3	0.53561\\
4	0.757876\\
5	0.928213\\
};
\addplot[forget plot, color=white!15!black] table[row sep=crcr] {%
0.654545454545455	0\\
5.34545454545454	0\\
};
\addplot[ybar, bar width=0.145, fill=mycolor4, draw=black, area legend] table[row sep=crcr] {%
1	0.439107\\
2	0.639102\\
3	0.537353\\
4	0.732909\\
5	0.904553\\
};
\addplot[forget plot, color=white!15!black] table[row sep=crcr] {%
0.654545454545455	0\\
5.34545454545454	0\\
};
\end{axis}

\begin{axis}[%
width=0.33\columnwidth,
height=0.168\columnwidth,
at={(0.378\columnwidth,0.259\columnwidth)},
scale only axis,
bar shift auto,
ybar=2*\pgflinewidth,
xmin=0.3,
xmax=5.6,
xtick={1,2,3,4,5},
xticklabels={{64},{128},{256},{512},{1024}},ticklabel style = {font=\scriptsize},
xlabel style={font=\scriptsize},
xlabel={Number of GPUs},
ymin=0,
ymax=1.15826,
axis background/.style={fill=white},
title style={font=\scriptsize,yshift=-0.8em},
title={$k=10$ degree/$\ell_1$-Jacobi sweeps},
xmajorgrids,
ymajorgrids
]
\addplot[ybar, bar width=0.145, fill=mycolor1, draw=black, area legend] table[row sep=crcr] {%
1	0.516564\\
2	0.743196\\
3	0.63636\\
4	0.851958\\
5	0.947339\\
};
\addplot[forget plot, color=white!15!black] table[row sep=crcr] {%
0.654545454545455	0\\
5.34545454545454	0\\
};
\addplot[ybar, bar width=0.145, fill=mycolor2, draw=black, area legend] table[row sep=crcr] {%
1	0.516366\\
2	0.747585\\
3	0.63466\\
4	0.850166\\
5	0.941942\\
};
\addplot[forget plot, color=white!15!black] table[row sep=crcr] {%
0.654545454545455	0\\
5.34545454545454	0\\
};
\addplot[ybar, bar width=0.145, fill=mycolor3, draw=black, area legend] table[row sep=crcr] {%
1	0.509996\\
2	0.740458\\
3	0.61883\\
4	0.844117\\
5	1.03974\\
};
\addplot[forget plot, color=white!15!black] table[row sep=crcr] {%
0.654545454545455	0\\
5.34545454545454	0\\
};
\addplot[ybar, bar width=0.145, fill=mycolor4, draw=black, area legend] table[row sep=crcr] {%
1	0.512613\\
2	0.734575\\
3	0.61908\\
4	0.839641\\
5	1.01935\\
};
\addplot[forget plot, color=white!15!black] table[row sep=crcr] {%
0.654545454545455	0\\
5.34545454545454	0\\
};
\end{axis}

\begin{axis}[%
width=0.33\columnwidth,
height=0.168\columnwidth,
at={(0\columnwidth,0\columnwidth)},
scale only axis,
bar shift auto,
ybar=2*\pgflinewidth,
xmin=0.3,
xmax=5.6,
xtick={1,2,3,4,5},
xticklabels={{64},{128},{256},{512},{1024}},ticklabel style = {font=\scriptsize},
xlabel style={font=\scriptsize},
xlabel={Number of GPUs},
ymin=0,
ymax=1.15826,
ylabel style={font=\scriptsize},
ylabel={Time per iteration (s)},
axis background/.style={fill=white},
title style={font=\scriptsize,yshift=-0.8em},
title={$k=12$ degree/$\ell_1$-Jacobi sweeps},
xmajorgrids,
ymajorgrids
]
\addplot[ybar, bar width=0.145, fill=mycolor1, draw=black, area legend] table[row sep=crcr] {%
1	0.591267\\
2	0.839659\\
3	0.723561\\
4	0.950135\\
5	1.04452\\
};
\addplot[forget plot, color=white!15!black] table[row sep=crcr] {%
0.654545454545455	0\\
5.34545454545454	0\\
};
\addplot[ybar, bar width=0.145, fill=mycolor2, draw=black, area legend] table[row sep=crcr] {%
1	0.596985\\
2	0.841303\\
3	0.720076\\
4	0.947244\\
5	1.07175\\
};
\addplot[forget plot, color=white!15!black] table[row sep=crcr] {%
0.654545454545455	0\\
5.34545454545454	0\\
};
\addplot[ybar, bar width=0.145, fill=mycolor3, draw=black, area legend] table[row sep=crcr] {%
1	0.558065\\
2	0.831856\\
3	0.706076\\
4	0.94001\\
5	1.15826\\
};
\addplot[forget plot, color=white!15!black] table[row sep=crcr] {%
0.654545454545455	0\\
5.34545454545454	0\\
};
\addplot[ybar, bar width=0.145, fill=mycolor4, draw=black, area legend] table[row sep=crcr] {%
1	0.553168\\
2	0.826546\\
3	0.703515\\
4	0.94494\\
5	1.13114\\
};
\addplot[forget plot, color=white!15!black] table[row sep=crcr] {%
0.654545454545455	0\\
5.34545454545454	0\\
};
\end{axis}

\end{tikzpicture}%
	
	\caption{Time per iteration. Time per iteration for the Poisson 3D test case employing the multigrid scheme based on the compatible matching aggregation in conjunction with the different polynomial accelerators and the baseline $\ell_1$-Jacobi for comparison.}
	\label{fig:3d_1024_match_timeperiteration}
\end{figure}
Again the iteration time of the different methods is completely comparable with that of the reference method, the proportional growth with respect to the degree $k$ of the polynomial is consistent with the increase in the number of sparse matrix-vector products.

\subsubsection{VBM}\label{sec:sm_vbm_poisson3d}

This set of experiments extends those in Section~\ref{sec:poisson3D} of the paper. In Fig.~\ref{fig:3d_1024_vbm_tsolve} we report the solution times for the different smoothers using in each case the VBM aggregation procedure.
\begin{figure}[htbp]
	\centering
%
%
\definecolor{mycolor1}{rgb}{0.00000,0.44700,0.74100}%
\definecolor{mycolor2}{rgb}{0.85000,0.32500,0.09800}%
\definecolor{mycolor3}{rgb}{0.92900,0.69400,0.12500}%
\definecolor{mycolor4}{rgb}{0.49400,0.18400,0.55600}%
\begin{tikzpicture}

\begin{axis}[%
width=0.33\columnwidth,
height=0.168\columnwidth,
at={(0\columnwidth,0.538\columnwidth)},
scale only axis,
bar shift auto,
ybar=2*\pgflinewidth,
xmin=0.4,
xmax=5.6,
xtick={1,2,3,4,5},
xticklabels={{64},{128},{256},{512},{1024}},ticklabel style = {font=\scriptsize},
ymin=0,
ymax=13.6662,
ylabel style={font=\scriptsize},
ylabel={Solve time (s)},
axis background/.style={fill=white},
title style={font=\scriptsize,yshift=-0.8em},
title={$k=4$ degree/$\ell_1$-Jacobi sweeps},
xmajorgrids,
ymajorgrids,
legend columns=4,
legend style={at={(-0.3,1.4)}, anchor=north west, legend cell align=left, align=left, draw=none, fill=none, font=\scriptsize}
]
\addplot[ybar, bar width=0.145, fill=mycolor1, draw=black, area legend] table[row sep=crcr] {%
1	3.13754\\
2	4.03314\\
3	4.6204\\
4	5.19123\\
5	8.75351\\
};
\addplot[forget plot, color=white!15!black] table[row sep=crcr] {%
0.654545454545455	0\\
5.34545454545454	0\\
};
\addlegendentry{Chebyshev 4}

\addplot[ybar, bar width=0.145, fill=mycolor2, draw=black, area legend] table[row sep=crcr] {%
1	3.12623\\
2	4.17383\\
3	4.69222\\
4	5.42257\\
5	8.09073\\
};
\addplot[forget plot, color=white!15!black] table[row sep=crcr] {%
0.654545454545455	0\\
5.34545454545454	0\\
};
\addlegendentry{Opt. Chebyshev 4}

\addplot[ybar, bar width=0.145, fill=mycolor3, draw=black, area legend] table[row sep=crcr] {%
1	3.15258\\
2	4.10474\\
3	4.50265\\
4	5.23116\\
5	7.48149\\
};
\addplot[forget plot, color=white!15!black] table[row sep=crcr] {%
0.654545454545455	0\\
5.34545454545454	0\\
};
\addlegendentry{Opt. Chebyshev 1}

\addplot[ybar, bar width=0.145, fill=mycolor4, draw=black, area legend] table[row sep=crcr] {%
1	3.26063\\
2	4.38198\\
3	4.98551\\
4	5.20312\\
5	8.32194\\
};
\addplot[forget plot, color=white!15!black] table[row sep=crcr] {%
0.654545454545455	0\\
5.34545454545454	0\\
};
\addlegendentry{L1-Jacobi}

\end{axis}

\begin{axis}[%
width=0.33\columnwidth,
height=0.168\columnwidth,
at={(0.378\columnwidth,0.538\columnwidth)},
scale only axis,
bar shift auto,
ybar=2*\pgflinewidth,
xmin=0.4,
xmax=5.6,
xtick={1,2,3,4,5},
xticklabels={{64},{128},{256},{512},{1024}},ticklabel style = {font=\scriptsize},
ymin=0,
ymax=13.6662,
axis background/.style={fill=white},
title style={font=\scriptsize,yshift=-0.8em},
title={$k=6$ degree/$\ell_1$-Jacobi sweeps},
xmajorgrids,
ymajorgrids
]
\addplot[ybar, bar width=0.145, fill=mycolor1, draw=black, area legend] table[row sep=crcr] {%
1	3.69253\\
2	4.72677\\
3	5.69108\\
4	5.94\\
5	9.85938\\
};
\addplot[forget plot, color=white!15!black] table[row sep=crcr] {%
0.654545454545455	0\\
5.34545454545454	0\\
};
\addplot[ybar, bar width=0.145, fill=mycolor2, draw=black, area legend] table[row sep=crcr] {%
1	3.68954\\
2	4.82642\\
3	5.54549\\
4	6.55912\\
5	8.67707\\
};
\addplot[forget plot, color=white!15!black] table[row sep=crcr] {%
0.654545454545455	0\\
5.34545454545454	0\\
};
\addplot[ybar, bar width=0.145, fill=mycolor3, draw=black, area legend] table[row sep=crcr] {%
1	3.73362\\
2	4.67499\\
3	5.3865\\
4	5.92927\\
5	8.38325\\
};
\addplot[forget plot, color=white!15!black] table[row sep=crcr] {%
0.654545454545455	0\\
5.34545454545454	0\\
};
\addplot[ybar, bar width=0.145, fill=mycolor4, draw=black, area legend] table[row sep=crcr] {%
1	4.10668\\
2	5.38828\\
3	5.98553\\
4	6.21606\\
5	10.33\\
};
\addplot[forget plot, color=white!15!black] table[row sep=crcr] {%
0.654545454545455	0\\
5.34545454545454	0\\
};
\end{axis}

\begin{axis}[%
width=0.33\columnwidth,
height=0.168\columnwidth,
at={(0\columnwidth,0.269\columnwidth)},
scale only axis,
bar shift auto,
ybar=2*\pgflinewidth,
xmin=0.4,
xmax=5.6,
xtick={1,2,3,4,5},
xticklabels={{64},{128},{256},{512},{1024}},ticklabel style = {font=\scriptsize},
ymin=0,
ymax=13.6662,
ylabel style={font=\scriptsize},
ylabel={Solve time (s)},
axis background/.style={fill=white},
title style={font=\scriptsize,yshift=-0.8em},
title={$k=8$ degree/$\ell_1$-Jacobi sweeps},
xmajorgrids,
ymajorgrids
]
\addplot[ybar, bar width=0.145, fill=mycolor1, draw=black, area legend] table[row sep=crcr] {%
1	4.14593\\
2	5.30621\\
3	6.54808\\
4	7.03016\\
5	10.8176\\
};
\addplot[forget plot, color=white!15!black] table[row sep=crcr] {%
0.654545454545455	0\\
5.34545454545454	0\\
};
\addplot[ybar, bar width=0.145, fill=mycolor2, draw=black, area legend] table[row sep=crcr] {%
1	3.85771\\
2	5.41227\\
3	6.02423\\
4	7.32973\\
5	9.26067\\
};
\addplot[forget plot, color=white!15!black] table[row sep=crcr] {%
0.654545454545455	0\\
5.34545454545454	0\\
};
\addplot[ybar, bar width=0.145, fill=mycolor3, draw=black, area legend] table[row sep=crcr] {%
1	4.14938\\
2	5.26763\\
3	6.17834\\
4	6.83205\\
5	9.88275\\
};
\addplot[forget plot, color=white!15!black] table[row sep=crcr] {%
0.654545454545455	0\\
5.34545454545454	0\\
};
\addplot[ybar, bar width=0.145, fill=mycolor4, draw=black, area legend] table[row sep=crcr] {%
1	4.66769\\
2	6.43077\\
3	7.33031\\
4	7.4775\\
5	11.5877\\
};
\addplot[forget plot, color=white!15!black] table[row sep=crcr] {%
0.654545454545455	0\\
5.34545454545454	0\\
};
\end{axis}

\begin{axis}[%
width=0.33\columnwidth,
height=0.168\columnwidth,
at={(0.378\columnwidth,0.269\columnwidth)},
scale only axis,
bar shift auto,
ybar=2*\pgflinewidth,
xmin=0.4,
xmax=5.6,
xtick={1,2,3,4,5},
xticklabels={{64},{128},{256},{512},{1024}},ticklabel style = {font=\scriptsize},
xlabel style={font=\scriptsize},
xlabel={Number of GPUs},
ymin=0,
ymax=13.6662,
axis background/.style={fill=white},
title style={font=\scriptsize,yshift=-0.8em},
title={$k=10$ degree/$\ell_1$-Jacobi sweeps},
xmajorgrids,
ymajorgrids
]
\addplot[ybar, bar width=0.145, fill=mycolor1, draw=black, area legend] table[row sep=crcr] {%
1	4.49684\\
2	5.80622\\
3	6.81808\\
4	8.02738\\
5	12.0688\\
};
\addplot[forget plot, color=white!15!black] table[row sep=crcr] {%
0.654545454545455	0\\
5.34545454545454	0\\
};
\addplot[ybar, bar width=0.145, fill=mycolor2, draw=black, area legend] table[row sep=crcr] {%
1	4.15709\\
2	5.89775\\
3	6.43593\\
4	7.80655\\
5	10.7252\\
};
\addplot[forget plot, color=white!15!black] table[row sep=crcr] {%
0.654545454545455	0\\
5.34545454545454	0\\
};
\addplot[ybar, bar width=0.145, fill=mycolor3, draw=black, area legend] table[row sep=crcr] {%
1	4.49309\\
2	5.708\\
3	7.06671\\
4	7.56971\\
5	10.4504\\
};
\addplot[forget plot, color=white!15!black] table[row sep=crcr] {%
0.654545454545455	0\\
5.34545454545454	0\\
};
\addplot[ybar, bar width=0.145, fill=mycolor4, draw=black, area legend] table[row sep=crcr] {%
1	5.48079\\
2	7.01886\\
3	8.07357\\
4	8.8338\\
5	13.0593\\
};
\addplot[forget plot, color=white!15!black] table[row sep=crcr] {%
0.654545454545455	0\\
5.34545454545454	0\\
};
\end{axis}

\begin{axis}[%
width=0.33\columnwidth,
height=0.168\columnwidth,
at={(0\columnwidth,0\columnwidth)},
scale only axis,
bar shift auto,
ybar=2*\pgflinewidth,
xmin=0.4,
xmax=5.6,
xtick={1,2,3,4,5},
xticklabels={{64},{128},{256},{512},{1024}},ticklabel style = {font=\scriptsize},
xlabel style={font=\scriptsize},
xlabel={Number of GPUs},
ymin=0,
ymax=13.6662,
ylabel style={font=\scriptsize},
ylabel={Solve time (s)},
axis background/.style={fill=white},
title style={font=\scriptsize,yshift=-0.8em},
title={$k=12$ degree/$\ell_1$-Jacobi sweeps},
xmajorgrids,
ymajorgrids
]
\addplot[ybar, bar width=0.145, fill=mycolor1, draw=black, area legend] table[row sep=crcr] {%
1	4.72639\\
2	6.26982\\
3	7.05432\\
4	8.76728\\
5	12.7047\\
};
\addplot[forget plot, color=white!15!black] table[row sep=crcr] {%
0.654545454545455	0\\
5.34545454545454	0\\
};
\addplot[ybar, bar width=0.145, fill=mycolor2, draw=black, area legend] table[row sep=crcr] {%
1	4.78067\\
2	6.20936\\
3	6.92201\\
4	8.43873\\
5	11.2225\\
};
\addplot[forget plot, color=white!15!black] table[row sep=crcr] {%
0.654545454545455	0\\
5.34545454545454	0\\
};
\addplot[ybar, bar width=0.145, fill=mycolor3, draw=black, area legend] table[row sep=crcr] {%
1	4.72902\\
2	6.54575\\
3	7.55415\\
4	8.12151\\
5	11.4161\\
};
\addplot[forget plot, color=white!15!black] table[row sep=crcr] {%
0.654545454545455	0\\
5.34545454545454	0\\
};
\addplot[ybar, bar width=0.145, fill=mycolor4, draw=black, area legend] table[row sep=crcr] {%
1	6.28102\\
2	7.99286\\
3	9.00956\\
4	9.79645\\
5	13.6662\\
};
\addplot[forget plot, color=white!15!black] table[row sep=crcr] {%
0.654545454545455	0\\
5.34545454545454	0\\
};
\end{axis}
\end{tikzpicture}%
	
	\caption{Solve time (s). Time to solution (s) for the Poisson 3D test case employing the multigrid scheme based on the VBM aggregation in conjunction with the different polynomial accelerators and the baseline $\ell_1$-Jacobi for comparison.}
	\label{fig:3d_1024_vbm_tsolve}
\end{figure}
What we observe is that again as the number of iterations of the smoother increases, i.e., the degree of the polynomial, the performance in total solution time worsens and the best ones are obtained for the lowest degree. This is even in the face of the reduction in the total number of iterations observed in Fig.~\ref{fig:3d_1024_vbm_itercount}. 
\begin{figure}[htbp]
	\centering
%
%
\definecolor{mycolor1}{rgb}{0.00000,0.44700,0.74100}%
\definecolor{mycolor2}{rgb}{0.85000,0.32500,0.09800}%
\definecolor{mycolor3}{rgb}{0.92900,0.69400,0.12500}%
\definecolor{mycolor4}{rgb}{0.49400,0.18400,0.55600}%
\begin{tikzpicture}

\begin{axis}[%
width=0.33\columnwidth,
height=0.168\columnwidth,
at={(0\columnwidth,0.543\columnwidth)},
scale only axis,
bar shift auto,ybar=2*\pgflinewidth,
xmin=0.4,
xmax=5.6,
xtick={1,2,3,4,5},
xticklabels={{64},{128},{256},{512},{1024}},ticklabel style = {font=\scriptsize},ticklabel style = {font=\scriptsize},
ymin=0,
ymax=28,
ylabel style={font=\color{white!15!black}},
ylabel={Iterations},
axis background/.style={fill=white},
title style={font=\scriptsize,yshift=-0.8em},
title={$k=4$ degree/$\ell_1$-Jacobi sweeps},
xmajorgrids,
ymajorgrids,
legend columns=4,
legend style={at={(-0.3,1.4)}, anchor=north west, legend cell align=left, align=left, draw=none, fill=none, font=\scriptsize}
]
\addplot[ybar, bar width=0.145, fill=mycolor1, draw=black, area legend] table[row sep=crcr] {%
1	15\\
2	17\\
3	20\\
4	22\\
5	27\\
};
\addplot[forget plot, color=white!15!black] table[row sep=crcr] {%
0.654545454545455	0\\
5.34545454545454	0\\
};
\addlegendentry{Chebyshev 4}

\addplot[ybar, bar width=0.145, fill=mycolor2, draw=black, area legend] table[row sep=crcr] {%
1	15\\
2	17\\
3	20\\
4	22\\
5	27\\
};
\addplot[forget plot, color=white!15!black] table[row sep=crcr] {%
0.654545454545455	0\\
5.34545454545454	0\\
};
\addlegendentry{Opt. Chebyshev 4}

\addplot[ybar, bar width=0.145, fill=mycolor3, draw=black, area legend] table[row sep=crcr] {%
1	15\\
2	17\\
3	20\\
4	22\\
5	27\\
};
\addplot[forget plot, color=white!15!black] table[row sep=crcr] {%
0.654545454545455	0\\
5.34545454545454	0\\
};
\addlegendentry{Opt. Chebyshev 1}

\addplot[ybar, bar width=0.145, fill=mycolor4, draw=black, area legend] table[row sep=crcr] {%
1	16\\
2	19\\
3	22\\
4	23\\
5	28\\
};
\addplot[forget plot, color=white!15!black] table[row sep=crcr] {%
0.654545454545455	0\\
5.34545454545454	0\\
};
\addlegendentry{L1-Jacobi}

\end{axis}

\begin{axis}[%
width=0.33\columnwidth,
height=0.168\columnwidth,
at={(0.378\columnwidth,0.543\columnwidth)},
scale only axis,
bar shift auto,ybar=2*\pgflinewidth,
xmin=0.4,
xmax=5.6,
xtick={1,2,3,4,5},
xticklabels={{64},{128},{256},{512},{1024}},ticklabel style = {font=\scriptsize},
ymin=0,
ymax=28,
axis background/.style={fill=white},
title style={font=\scriptsize,yshift=-0.8em},
title={$k=6$ degree/$\ell_1$-Jacobi sweeps},
xmajorgrids,
ymajorgrids
]
\addplot[ybar, bar width=0.145, fill=mycolor1, draw=black, area legend] table[row sep=crcr] {%
1	14\\
2	16\\
3	19\\
4	20\\
5	25\\
};
\addplot[forget plot, color=white!15!black] table[row sep=crcr] {%
0.654545454545455	0\\
5.34545454545454	0\\
};
\addplot[ybar, bar width=0.145, fill=mycolor2, draw=black, area legend] table[row sep=crcr] {%
1	14\\
2	16\\
3	19\\
4	20\\
5	25\\
};
\addplot[forget plot, color=white!15!black] table[row sep=crcr] {%
0.654545454545455	0\\
5.34545454545454	0\\
};
\addplot[ybar, bar width=0.145, fill=mycolor3, draw=black, area legend] table[row sep=crcr] {%
1	14\\
2	16\\
3	19\\
4	20\\
5	25\\
};
\addplot[forget plot, color=white!15!black] table[row sep=crcr] {%
0.654545454545455	0\\
5.34545454545454	0\\
};
\addplot[ybar, bar width=0.145, fill=mycolor4, draw=black, area legend] table[row sep=crcr] {%
1	16\\
2	18\\
3	21\\
4	22\\
5	28\\
};
\addplot[forget plot, color=white!15!black] table[row sep=crcr] {%
0.654545454545455	0\\
5.34545454545454	0\\
};
\end{axis}

\begin{axis}[%
width=0.33\columnwidth,
height=0.168\columnwidth,
at={(0\columnwidth,0.271\columnwidth)},
scale only axis,
bar shift auto,ybar=2*\pgflinewidth,
xmin=0.4,
xmax=5.6,
xtick={1,2,3,4,5},
xticklabels={{64},{128},{256},{512},{1024}},ticklabel style = {font=\scriptsize},
ymin=0,
ymax=27,
ylabel style={font=\color{white!15!black}},
ylabel={Iterations},
axis background/.style={fill=white},
title style={font=\scriptsize,yshift=-0.8em},
title={$k=8$ degree/$\ell_1$-Jacobi sweeps},
xmajorgrids,
ymajorgrids
]
\addplot[ybar, bar width=0.145, fill=mycolor1, draw=black, area legend] table[row sep=crcr] {%
1	13\\
2	15\\
3	18\\
4	19\\
5	23\\
};
\addplot[forget plot, color=white!15!black] table[row sep=crcr] {%
0.654545454545455	0\\
5.34545454545454	0\\
};
\addplot[ybar, bar width=0.145, fill=mycolor2, draw=black, area legend] table[row sep=crcr] {%
1	12\\
2	15\\
3	17\\
4	19\\
5	23\\
};
\addplot[forget plot, color=white!15!black] table[row sep=crcr] {%
0.654545454545455	0\\
5.34545454545454	0\\
};
\addplot[ybar, bar width=0.145, fill=mycolor3, draw=black, area legend] table[row sep=crcr] {%
1	13\\
2	15\\
3	18\\
4	19\\
5	24\\
};
\addplot[forget plot, color=white!15!black] table[row sep=crcr] {%
0.654545454545455	0\\
5.34545454545454	0\\
};
\addplot[ybar, bar width=0.145, fill=mycolor4, draw=black, area legend] table[row sep=crcr] {%
1	15\\
2	18\\
3	21\\
4	22\\
5	27\\
};
\addplot[forget plot, color=white!15!black] table[row sep=crcr] {%
0.654545454545455	0\\
5.34545454545454	0\\
};
\end{axis}

\begin{axis}[%
width=0.33\columnwidth,
height=0.168\columnwidth,
at={(0.378\columnwidth,0.271\columnwidth)},
scale only axis,
bar shift auto,ybar=2*\pgflinewidth,
xmin=0.4,
xmax=5.6,
xtick={1,2,3,4,5},
xticklabels={{64},{128},{256},{512},{1024}},ticklabel style = {font=\scriptsize},
xlabel style={font=\scriptsize},
xlabel={Number of GPUs},
ymin=0,
ymax=27,
axis background/.style={fill=white},
title style={font=\scriptsize,yshift=-0.8em},
title={$k=10$ degree/$\ell_1$-Jacobi sweeps},
xmajorgrids,
ymajorgrids
]
\addplot[ybar, bar width=0.145, fill=mycolor1, draw=black, area legend] table[row sep=crcr] {%
1	12\\
2	14\\
3	16\\
4	18\\
5	22\\
};
\addplot[forget plot, color=white!15!black] table[row sep=crcr] {%
0.654545454545455	0\\
5.34545454545454	0\\
};
\addplot[ybar, bar width=0.145, fill=mycolor2, draw=black, area legend] table[row sep=crcr] {%
1	11\\
2	14\\
3	16\\
4	18\\
5	22\\
};
\addplot[forget plot, color=white!15!black] table[row sep=crcr] {%
0.654545454545455	0\\
5.34545454545454	0\\
};
\addplot[ybar, bar width=0.145, fill=mycolor3, draw=black, area legend] table[row sep=crcr] {%
1	12\\
2	14\\
3	17\\
4	18\\
5	22\\
};
\addplot[forget plot, color=white!15!black] table[row sep=crcr] {%
0.654545454545455	0\\
5.34545454545454	0\\
};
\addplot[ybar, bar width=0.145, fill=mycolor4, draw=black, area legend] table[row sep=crcr] {%
1	15\\
2	17\\
3	20\\
4	22\\
5	27\\
};
\addplot[forget plot, color=white!15!black] table[row sep=crcr] {%
0.654545454545455	0\\
5.34545454545454	0\\
};
\end{axis}

\begin{axis}[%
width=0.33\columnwidth,
height=0.168\columnwidth,
at={(0\columnwidth,0\columnwidth)},
scale only axis,
bar shift auto,ybar=2*\pgflinewidth,
xmin=0.4,
xmax=5.6,
xtick={1,2,3,4,5},
xticklabels={{64},{128},{256},{512},{1024}},ticklabel style = {font=\scriptsize},
xlabel style={font=\scriptsize},
xlabel={Number of GPUs},
ymin=0,
ymax=26,
ylabel style={font=\color{white!15!black}},
ylabel={Iterations},
axis background/.style={fill=white},
title style={font=\scriptsize,yshift=-0.8em},
title={$k=12$ degree/$\ell_1$-Jacobi sweeps},
xmajorgrids,
ymajorgrids
]
\addplot[ybar, bar width=0.145, fill=mycolor1, draw=black, area legend] table[row sep=crcr] {%
1	11\\
2	13\\
3	15\\
4	17\\
5	21\\
};
\addplot[forget plot, color=white!15!black] table[row sep=crcr] {%
0.654545454545455	0\\
5.34545454545454	0\\
};
\addplot[ybar, bar width=0.145, fill=mycolor2, draw=black, area legend] table[row sep=crcr] {%
1	11\\
2	13\\
3	15\\
4	17\\
5	20\\
};
\addplot[forget plot, color=white!15!black] table[row sep=crcr] {%
0.654545454545455	0\\
5.34545454545454	0\\
};
\addplot[ybar, bar width=0.145, fill=mycolor3, draw=black, area legend] table[row sep=crcr] {%
1	11\\
2	14\\
3	16\\
4	17\\
5	21\\
};
\addplot[forget plot, color=white!15!black] table[row sep=crcr] {%
0.654545454545455	0\\
5.34545454545454	0\\
};
\addplot[ybar, bar width=0.145, fill=mycolor4, draw=black, area legend] table[row sep=crcr] {%
1	15\\
2	17\\
3	20\\
4	21\\
5	26\\
};
\addplot[forget plot, color=white!15!black] table[row sep=crcr] {%
0.654545454545455	0\\
5.34545454545454	0\\
};
\end{axis}

\begin{axis}[%
width=0.859\columnwidth,
height=0.906\columnwidth,
at={(-0.112\columnwidth,-0.1\columnwidth)},
scale only axis,
xmin=0,
xmax=1,
ymin=0,
ymax=1,
axis line style={draw=none},
ticks=none,
axis x line*=bottom,
axis y line*=left
]
\end{axis}
\end{tikzpicture}%
	
	\caption{Iteration count. Total number of FCG iterations for the Poisson 3D test case employing the multigrid scheme based on the VBM aggregation in conjunction with the different polynomial accelerators and the baseline $\ell_1$-Jacobi for comparison.}
	\label{fig:3d_1024_vbm_itercount}
\end{figure}
The decrease in the number of iterations in the face of the increase in cost in terms of sparse matrix-vector products is not sufficient to lower the total solution time. On the other hand, the algorithmic scalability of the method remains consistent with what has been seen in other cases.
\begin{figure}[htbp]
	\centering
%
%
\definecolor{mycolor1}{rgb}{0.00000,0.44700,0.74100}%
\definecolor{mycolor2}{rgb}{0.85000,0.32500,0.09800}%
\definecolor{mycolor3}{rgb}{0.92900,0.69400,0.12500}%
\definecolor{mycolor4}{rgb}{0.49400,0.18400,0.55600}%
\begin{tikzpicture}

\begin{axis}[%
width=0.33\columnwidth,
height=0.168\columnwidth,
at={(0\columnwidth,0.541\columnwidth)},
scale only axis,
bar shift auto,ybar=2*\pgflinewidth,ybar=2*\pgflinewidth,
xmin=0.4,
xmax=5.6,
xtick={1,2,3,4,5},
xticklabels={{64},{128},{256},{512},{1024}},ticklabel style = {font=\scriptsize},
ymin=0,
ymax=0.604985,
ylabel style={font=\scriptsize},
ylabel={Time per iteration (s)},
axis background/.style={fill=white},
title style={font=\scriptsize,yshift=-0.8em},
title={$k=4$ degree/$\ell_1$-Jacobi sweeps},
xmajorgrids,
ymajorgrids,
legend columns=4,
legend style={at={(-0.3,1.4)}, anchor=north west, legend cell align=left, align=left, draw=none, fill=none, font=\scriptsize}
]
\addplot[ybar, bar width=0.145, fill=mycolor1, draw=black, area legend] table[row sep=crcr] {%
1	0.20917\\
2	0.237244\\
3	0.23102\\
4	0.235965\\
5	0.324204\\
};
\addplot[forget plot, color=white!15!black] table[row sep=crcr] {%
0.654545454545455	0\\
5.34545454545454	0\\
};
\addlegendentry{Chebyshev 4}

\addplot[ybar, bar width=0.145, fill=mycolor2, draw=black, area legend] table[row sep=crcr] {%
1	0.208415\\
2	0.245519\\
3	0.234611\\
4	0.24648\\
5	0.299657\\
};
\addplot[forget plot, color=white!15!black] table[row sep=crcr] {%
0.654545454545455	0\\
5.34545454545454	0\\
};
\addlegendentry{Opt. Chebyshev 4}

\addplot[ybar, bar width=0.145, fill=mycolor3, draw=black, area legend] table[row sep=crcr] {%
1	0.210172\\
2	0.241455\\
3	0.225133\\
4	0.23778\\
5	0.277092\\
};
\addplot[forget plot, color=white!15!black] table[row sep=crcr] {%
0.654545454545455	0\\
5.34545454545454	0\\
};
\addlegendentry{Opt. Chebyshev 1}

\addplot[ybar, bar width=0.145, fill=mycolor4, draw=black, area legend] table[row sep=crcr] {%
1	0.203789\\
2	0.230631\\
3	0.226614\\
4	0.226223\\
5	0.297212\\
};
\addplot[forget plot, color=white!15!black] table[row sep=crcr] {%
0.654545454545455	0\\
5.34545454545454	0\\
};
\addlegendentry{L1-Jacobi}

\end{axis}

\begin{axis}[%
width=0.33\columnwidth,
height=0.168\columnwidth,
at={(0.378\columnwidth,0.541\columnwidth)},
scale only axis,
bar shift auto,ybar=2*\pgflinewidth,
xmin=0.4,
xmax=5.6,
xtick={1,2,3,4,5},
xticklabels={{64},{128},{256},{512},{1024}},ticklabel style = {font=\scriptsize},
ymin=0,
ymax=0.604985,
axis background/.style={fill=white},
title style={font=\scriptsize,yshift=-0.8em},
title={$k=6$ degree/$\ell_1$-Jacobi sweeps},
xmajorgrids,
ymajorgrids
]
\addplot[ybar, bar width=0.145, fill=mycolor1, draw=black, area legend] table[row sep=crcr] {%
1	0.263752\\
2	0.295423\\
3	0.299531\\
4	0.297\\
5	0.394375\\
};
\addplot[forget plot, color=white!15!black] table[row sep=crcr] {%
0.654545454545455	0\\
5.34545454545454	0\\
};
\addplot[ybar, bar width=0.145, fill=mycolor2, draw=black, area legend] table[row sep=crcr] {%
1	0.263538\\
2	0.301651\\
3	0.291868\\
4	0.327956\\
5	0.347083\\
};
\addplot[forget plot, color=white!15!black] table[row sep=crcr] {%
0.654545454545455	0\\
5.34545454545454	0\\
};
\addplot[ybar, bar width=0.145, fill=mycolor3, draw=black, area legend] table[row sep=crcr] {%
1	0.266687\\
2	0.292187\\
3	0.2835\\
4	0.296464\\
5	0.33533\\
};
\addplot[forget plot, color=white!15!black] table[row sep=crcr] {%
0.654545454545455	0\\
5.34545454545454	0\\
};
\addplot[ybar, bar width=0.145, fill=mycolor4, draw=black, area legend] table[row sep=crcr] {%
1	0.256667\\
2	0.299349\\
3	0.285025\\
4	0.282548\\
5	0.36893\\
};
\addplot[forget plot, color=white!15!black] table[row sep=crcr] {%
0.654545454545455	0\\
5.34545454545454	0\\
};
\end{axis}

\begin{axis}[%
width=0.33\columnwidth,
height=0.168\columnwidth,
at={(0\columnwidth,0.271\columnwidth)},
scale only axis,
bar shift auto,ybar=2*\pgflinewidth,
xmin=0.4,
xmax=5.6,
xtick={1,2,3,4,5},
xticklabels={{64},{128},{256},{512},{1024}},ticklabel style = {font=\scriptsize},
ymin=0,
ymax=0.604985,
ylabel style={font=\scriptsize},
ylabel={Time per iteration (s)},
axis background/.style={fill=white},
title style={font=\scriptsize,yshift=-0.8em},
title={$k=8$ degree/$\ell_1$-Jacobi sweeps},
xmajorgrids,
ymajorgrids
]
\addplot[ybar, bar width=0.145, fill=mycolor1, draw=black, area legend] table[row sep=crcr] {%
1	0.318918\\
2	0.353748\\
3	0.363782\\
4	0.370008\\
5	0.470332\\
};
\addplot[forget plot, color=white!15!black] table[row sep=crcr] {%
0.654545454545455	0\\
5.34545454545454	0\\
};
\addplot[ybar, bar width=0.145, fill=mycolor2, draw=black, area legend] table[row sep=crcr] {%
1	0.321476\\
2	0.360818\\
3	0.354366\\
4	0.385775\\
5	0.402638\\
};
\addplot[forget plot, color=white!15!black] table[row sep=crcr] {%
0.654545454545455	0\\
5.34545454545454	0\\
};
\addplot[ybar, bar width=0.145, fill=mycolor3, draw=black, area legend] table[row sep=crcr] {%
1	0.319183\\
2	0.351176\\
3	0.343241\\
4	0.359582\\
5	0.411781\\
};
\addplot[forget plot, color=white!15!black] table[row sep=crcr] {%
0.654545454545455	0\\
5.34545454545454	0\\
};
\addplot[ybar, bar width=0.145, fill=mycolor4, draw=black, area legend] table[row sep=crcr] {%
1	0.311179\\
2	0.357265\\
3	0.349062\\
4	0.339886\\
5	0.429174\\
};
\addplot[forget plot, color=white!15!black] table[row sep=crcr] {%
0.654545454545455	0\\
5.34545454545454	0\\
};
\end{axis}

\begin{axis}[%
width=0.33\columnwidth,
height=0.168\columnwidth,
at={(0.378\columnwidth,0.271\columnwidth)},
scale only axis,
bar shift auto,ybar=2*\pgflinewidth,
xmin=0.4,
xmax=5.6,
xtick={1,2,3,4,5},
xticklabels={{64},{128},{256},{512},{1024}},ticklabel style = {font=\scriptsize},
xlabel style={font=\scriptsize},
xlabel={Number of GPUs},
ymin=0,
ymax=0.604985,
axis background/.style={fill=white},
title style={font=\scriptsize,yshift=-0.8em},
title={$k=10$ degree/$\ell_1$-Jacobi sweeps},
xmajorgrids,
ymajorgrids
]
\addplot[ybar, bar width=0.145, fill=mycolor1, draw=black, area legend] table[row sep=crcr] {%
1	0.374737\\
2	0.41473\\
3	0.42613\\
4	0.445966\\
5	0.548582\\
};
\addplot[forget plot, color=white!15!black] table[row sep=crcr] {%
0.654545454545455	0\\
5.34545454545454	0\\
};
\addplot[ybar, bar width=0.145, fill=mycolor2, draw=black, area legend] table[row sep=crcr] {%
1	0.377917\\
2	0.421268\\
3	0.402246\\
4	0.433697\\
5	0.487511\\
};
\addplot[forget plot, color=white!15!black] table[row sep=crcr] {%
0.654545454545455	0\\
5.34545454545454	0\\
};
\addplot[ybar, bar width=0.145, fill=mycolor3, draw=black, area legend] table[row sep=crcr] {%
1	0.374424\\
2	0.407714\\
3	0.415689\\
4	0.42054\\
5	0.475019\\
};
\addplot[forget plot, color=white!15!black] table[row sep=crcr] {%
0.654545454545455	0\\
5.34545454545454	0\\
};
\addplot[ybar, bar width=0.145, fill=mycolor4, draw=black, area legend] table[row sep=crcr] {%
1	0.365386\\
2	0.412874\\
3	0.403678\\
4	0.401536\\
5	0.483678\\
};
\addplot[forget plot, color=white!15!black] table[row sep=crcr] {%
0.654545454545455	0\\
5.34545454545454	0\\
};
\end{axis}

\begin{axis}[%
width=0.33\columnwidth,
height=0.168\columnwidth,
at={(0\columnwidth,0\columnwidth)},
scale only axis,
bar shift auto,ybar=2*\pgflinewidth,
xmin=0.4,
xmax=5.6,
xtick={1,2,3,4,5},
xticklabels={{64},{128},{256},{512},{1024}},ticklabel style = {font=\scriptsize},
xlabel style={font=\scriptsize},
xlabel={Number of GPUs},
ymin=0,
ymax=0.604985,
ylabel style={font=\scriptsize},
ylabel={Time per iteration (s)},
axis background/.style={fill=white},
title style={font=\scriptsize,yshift=-0.8em},
title={$k=12$ degree/$\ell_1$-Jacobi sweeps},
xmajorgrids,
ymajorgrids
]
\addplot[ybar, bar width=0.145, fill=mycolor1, draw=black, area legend] table[row sep=crcr] {%
1	0.429671\\
2	0.482294\\
3	0.470288\\
4	0.515722\\
5	0.604985\\
};
\addplot[forget plot, color=white!15!black] table[row sep=crcr] {%
0.654545454545455	0\\
5.34545454545454	0\\
};
\addplot[ybar, bar width=0.145, fill=mycolor2, draw=black, area legend] table[row sep=crcr] {%
1	0.434606\\
2	0.477643\\
3	0.461467\\
4	0.496396\\
5	0.561125\\
};
\addplot[forget plot, color=white!15!black] table[row sep=crcr] {%
0.654545454545455	0\\
5.34545454545454	0\\
};
\addplot[ybar, bar width=0.145, fill=mycolor3, draw=black, area legend] table[row sep=crcr] {%
1	0.429911\\
2	0.467554\\
3	0.472134\\
4	0.477736\\
5	0.543626\\
};
\addplot[forget plot, color=white!15!black] table[row sep=crcr] {%
0.654545454545455	0\\
5.34545454545454	0\\
};
\addplot[ybar, bar width=0.145, fill=mycolor4, draw=black, area legend] table[row sep=crcr] {%
1	0.418735\\
2	0.470168\\
3	0.450478\\
4	0.466498\\
5	0.525623\\
};
\addplot[forget plot, color=white!15!black] table[row sep=crcr] {%
0.654545454545455	0\\
5.34545454545454	0\\
};
\end{axis}

\end{tikzpicture}%
	
	\caption{Time per iteration. Time per iteration for the Poisson 3D test case employing the multigrid scheme based on the VBM aggregation in conjunction with the different polynomial accelerators and the baseline $\ell_1$-Jacobi for comparison.}
	\label{fig:3d_1024_vbm_timeperiteration}
\end{figure}
The picture in Fig.~\ref{fig:3d_1024_vbm_timeperiteration} shows how the cost per iteration remains comparable to that of the $\ell_1$-Jacobi method as the number of GPUs increases. While the cost increases proportionally to the degree of the polynomial, i.e. the number of sparse matrix-vector products used. We finally observe that the time per iteration for the VBM aggregation is smaller than those for the compatible matching coarsening, therefore, although the number of iteration for compatible matching is generally smaller, the solve time obtained with the VBM scheme are smaller due to the smaller operator complexity of the built matrix hierarchies.

\subsection{Anisotropy 2D}
The following two sections contain additional information regarding the 2D anisotropic problem with angle $\theta = \pi/6$ and $\epsilon = 100$. They extend the results contained in Section~\ref{sec:anisotropy2D} adding to the considerations the effect of the smoothers for a higher degree of the polynomial and the use of matching-based aggregation.

\subsubsection{VBM}\label{sec:sm_vbm_aniso} In Fig.~\ref{fig:anisotropy_1024_soc1_tsolve} we report the time needed to reach the solution of the linear system by keeping the multigrid preconditioner hierarchy fixed and changing the smoother and between the different types of polynomial acceleration and with respect to the degree--number of sparse matrix-vector products used per iteration; we extend the results in Fig.~\ref{fig:anisotropy_1024_soc1_tsolve_reduced} by considering polynomial of higher degree while plotting all the figures on the same scale.
\begin{figure}[p]
	\centering
%
%
\definecolor{mycolor1}{rgb}{0.00000,0.44700,0.74100}%
\definecolor{mycolor2}{rgb}{0.85000,0.32500,0.09800}%
\definecolor{mycolor3}{rgb}{0.92900,0.69400,0.12500}%
\definecolor{mycolor4}{rgb}{0.49400,0.18400,0.55600}%
\begin{tikzpicture}

\begin{axis}[%
width=0.33\columnwidth,
height=0.168\columnwidth,
at={(0\columnwidth,0.466\columnwidth)},
scale only axis,
bar shift auto,
xmin=0.4,
xmax=5.6,
xtick={1,2,3,4,5},
xticklabels={{64},{128},{256},{512},{1024}},ticklabel style = {font=\scriptsize},
ymin=0,
ymax=295.042,
ybar=2*\pgflinewidth,
ylabel style={font=\scriptsize},
ylabel={Solve time (s)},
axis background/.style={fill=white},
title style={font=\scriptsize,yshift=-0.8em},
title={$k = 4$ degree/$\ell_1$-Jacobi sweeps},
xmajorgrids,
ymajorgrids,
legend columns=4,
legend style={at={(-0.3,1.4)}, anchor=north west, legend cell align=left, align=left, draw=none, fill=none, font=\scriptsize}
]
\addplot[ybar, bar width=0.145, fill=mycolor1, draw=black, area legend] table[row sep=crcr] {%
1	30.1411\\
2	44.9588\\
3	58.2224\\
4	96.3833\\
5	156.201\\
};
\addplot[forget plot, color=white!15!black] table[row sep=crcr] {%
0.654545454545455	0\\
5.34545454545454	0\\
};
\addlegendentry{Chebyshev 4}

\addplot[ybar, bar width=0.145, fill=mycolor2, draw=black, area legend] table[row sep=crcr] {%
1	30.2535\\
2	42.6847\\
3	58.8012\\
4	96.4112\\
5	150.529\\
};
\addplot[forget plot, color=white!15!black] table[row sep=crcr] {%
0.654545454545455	0\\
5.34545454545454	0\\
};
\addlegendentry{Opt. Chebyshev 4}

\addplot[ybar, bar width=0.145, fill=mycolor3, draw=black, area legend] table[row sep=crcr] {%
1	29.5266\\
2	44.2639\\
3	59.34\\
4	96.0507\\
5	158.523\\
};
\addplot[forget plot, color=white!15!black] table[row sep=crcr] {%
0.654545454545455	0\\
5.34545454545454	0\\
};
\addlegendentry{Opt. Chebyshev 1}

\addplot[ybar, bar width=0.145, fill=mycolor4, draw=black, area legend] table[row sep=crcr] {%
1	34.5084\\
2	47.7808\\
3	59.4805\\
4	99.6036\\
5	164.831\\
};
\addplot[forget plot, color=white!15!black] table[row sep=crcr] {%
0.654545454545455	0\\
5.34545454545454	0\\
};
\addlegendentry{$\ell_1$-Jacobi}

\end{axis}

\begin{axis}[%
width=0.33\columnwidth,
height=0.168\columnwidth,
at={(0.378\columnwidth,0.466\columnwidth)},
scale only axis,
bar shift auto,
ybar=2*\pgflinewidth,
xmin=0.4,
xmax=5.6,
xtick={1,2,3,4,5},
xticklabels={{64},{128},{256},{512},{1024}},ticklabel style = {font=\scriptsize},
ymin=0,
ymax=295.042,
axis background/.style={fill=white},
title style={font=\scriptsize,yshift=-0.8em},
title={$k = 6$ degree/$\ell_1$-Jacobi sweeps},
xmajorgrids,
ymajorgrids
]
\addplot[ybar, bar width=0.145, fill=mycolor1, draw=black, area legend] table[row sep=crcr] {%
1	37.6931\\
2	53.9501\\
3	63.6195\\
4	117.123\\
5	175.691\\
};
\addplot[forget plot, color=white!15!black] table[row sep=crcr] {%
0.654545454545455	0\\
5.34545454545454	0\\
};
\addplot[ybar, bar width=0.145, fill=mycolor2, draw=black, area legend] table[row sep=crcr] {%
1	35.7975\\
2	50.807\\
3	65.0002\\
4	109.008\\
5	177.742\\
};
\addplot[forget plot, color=white!15!black] table[row sep=crcr] {%
0.654545454545455	0\\
5.34545454545454	0\\
};
\addplot[ybar, bar width=0.145, fill=mycolor3, draw=black, area legend] table[row sep=crcr] {%
1	36.6786\\
2	51.8935\\
3	69.8769\\
4	112.191\\
5	188.397\\
};
\addplot[forget plot, color=white!15!black] table[row sep=crcr] {%
0.654545454545455	0\\
5.34545454545454	0\\
};
\addplot[ybar, bar width=0.145, fill=mycolor4, draw=black, area legend] table[row sep=crcr] {%
1	39.027\\
2	58.439\\
3	74.4741\\
4	120.732\\
5	221.694\\
};
\addplot[forget plot, color=white!15!black] table[row sep=crcr] {%
0.654545454545455	0\\
5.34545454545454	0\\
};
\end{axis}

\begin{axis}[%
width=0.33\columnwidth,
height=0.168\columnwidth,
at={(0\columnwidth,0.233\columnwidth)},
scale only axis,
bar shift auto,
ybar=2*\pgflinewidth,
xmin=0.4,
xmax=5.6,
xtick={1,2,3,4,5},
xticklabels={{64},{128},{256},{512},{1024}},ticklabel style = {font=\scriptsize},
ymin=0,
ymax=295.042,
ylabel style={font=\scriptsize},
ylabel={Solve time (s)},
axis background/.style={fill=white},
title style={font=\scriptsize,yshift=-0.8em},
title={$k = 8$ degree/$\ell_1$-Jacobi sweeps},
xmajorgrids,
ymajorgrids
]
\addplot[ybar, bar width=0.145, fill=mycolor1, draw=black, area legend] table[row sep=crcr] {%
1	41.3937\\
2	56.9124\\
3	76.7473\\
4	118.292\\
5	190.132\\
};
\addplot[forget plot, color=white!15!black] table[row sep=crcr] {%
0.654545454545455	0\\
5.34545454545454	0\\
};
\addplot[ybar, bar width=0.145, fill=mycolor2, draw=black, area legend] table[row sep=crcr] {%
1	41.4211\\
2	57.8245\\
3	73.5949\\
4	119.128\\
5	202.274\\
};
\addplot[forget plot, color=white!15!black] table[row sep=crcr] {%
0.654545454545455	0\\
5.34545454545454	0\\
};
\addplot[ybar, bar width=0.145, fill=mycolor3, draw=black, area legend] table[row sep=crcr] {%
1	43.1834\\
2	62.3894\\
3	76.2948\\
4	123.076\\
5	189.555\\
};
\addplot[forget plot, color=white!15!black] table[row sep=crcr] {%
0.654545454545455	0\\
5.34545454545454	0\\
};
\addplot[ybar, bar width=0.145, fill=mycolor4, draw=black, area legend] table[row sep=crcr] {%
1	47.6126\\
2	65.7615\\
3	89.5874\\
4	152.689\\
5	237.245\\
};
\addplot[forget plot, color=white!15!black] table[row sep=crcr] {%
0.654545454545455	0\\
5.34545454545454	0\\
};
\end{axis}

\begin{axis}[%
width=0.33\columnwidth,
height=0.168\columnwidth,
at={(0.378\columnwidth,0.233\columnwidth)},
scale only axis,
bar shift auto,
ybar=2*\pgflinewidth,
xmin=0.4,
xmax=5.6,
xtick={1,2,3,4,5},
xticklabels={{64},{128},{256},{512},{1024}},ticklabel style = {font=\scriptsize},
xlabel style={font=\scriptsize},
xlabel={Number of GPUs},
ymin=0,
ymax=295.042,
axis background/.style={fill=white},
title style={font=\scriptsize,yshift=-0.8em},
title={$k = 10$ degree/$\ell_1$-Jacobi sweeps},
xmajorgrids,
ymajorgrids
]
\addplot[ybar, bar width=0.145, fill=mycolor1, draw=black, area legend] table[row sep=crcr] {%
1	48.8974\\
2	66.22\\
3	83.5298\\
4	134.763\\
5	191.961\\
};
\addplot[forget plot, color=white!15!black] table[row sep=crcr] {%
0.654545454545455	0\\
5.34545454545454	0\\
};
\addplot[ybar, bar width=0.145, fill=mycolor2, draw=black, area legend] table[row sep=crcr] {%
1	49.7635\\
2	63.0106\\
3	78.2923\\
4	131.657\\
5	187.803\\
};
\addplot[forget plot, color=white!15!black] table[row sep=crcr] {%
0.654545454545455	0\\
5.34545454545454	0\\
};
\addplot[ybar, bar width=0.145, fill=mycolor3, draw=black, area legend] table[row sep=crcr] {%
1	49.4599\\
2	63.4882\\
3	79.2382\\
4	131.929\\
5	190.865\\
};
\addplot[forget plot, color=white!15!black] table[row sep=crcr] {%
0.654545454545455	0\\
5.34545454545454	0\\
};
\addplot[ybar, bar width=0.145, fill=mycolor4, draw=black, area legend] table[row sep=crcr] {%
1	55.2658\\
2	79.2231\\
3	100.932\\
4	168.261\\
5	261.055\\
};
\addplot[forget plot, color=white!15!black] table[row sep=crcr] {%
0.654545454545455	0\\
5.34545454545454	0\\
};
\end{axis}

\begin{axis}[%
width=0.33\columnwidth,
height=0.168\columnwidth,
at={(0\columnwidth,0\columnwidth)},
scale only axis,
bar shift auto,
ybar=2*\pgflinewidth,
xmin=0.4,
xmax=5.6,
xtick={1,2,3,4,5},
xticklabels={{64},{128},{256},{512},{1024}},ticklabel style = {font=\scriptsize},
xlabel style={font=\scriptsize},
xlabel={Number of GPUs},
ymin=0,
ymax=295.042,
ylabel style={font=\scriptsize},
ylabel={Solve time (s)},
axis background/.style={fill=white},
title style={font=\scriptsize,yshift=-0.8em},
title={$k = 12$ degree/$\ell_1$-Jacobi sweeps},
xmajorgrids,
ymajorgrids
]
\addplot[ybar, bar width=0.145, fill=mycolor1, draw=black, area legend] table[row sep=crcr] {%
1	54.2829\\
2	69.3239\\
3	87.9076\\
4	143.894\\
5	205.643\\
};
\addplot[forget plot, color=white!15!black] table[row sep=crcr] {%
0.654545454545455	0\\
5.34545454545454	0\\
};
\addplot[ybar, bar width=0.145, fill=mycolor2, draw=black, area legend] table[row sep=crcr] {%
1	51.3312\\
2	70.645\\
3	86.3897\\
4	135.271\\
5	203.488\\
};
\addplot[forget plot, color=white!15!black] table[row sep=crcr] {%
0.654545454545455	0\\
5.34545454545454	0\\
};
\addplot[ybar, bar width=0.145, fill=mycolor3, draw=black, area legend] table[row sep=crcr] {%
1	51.8609\\
2	72.2986\\
3	86.9647\\
4	154.646\\
5	218.241\\
};
\addplot[forget plot, color=white!15!black] table[row sep=crcr] {%
0.654545454545455	0\\
5.34545454545454	0\\
};
\addplot[ybar, bar width=0.145, fill=mycolor4, draw=black, area legend] table[row sep=crcr] {%
1	62.899\\
2	88.0817\\
3	112.498\\
4	194.632\\
5	295.042\\
};
\addplot[forget plot, color=white!15!black] table[row sep=crcr] {%
0.654545454545455	0\\
5.34545454545454	0\\
};
\end{axis}
\end{tikzpicture}%
	\caption{Solve time (s). Time to solution (s) for the Anisotropy 2D test case $(\varepsilon=100,\theta=\pi/6)$ employing the multigrid scheme based on the VBM aggregation in conjunction with the different polynomial accelerators and the baseline $\ell_1$-Jacobi for comparison.}
	\label{fig:anisotropy_1024_soc1_tsolve}
\end{figure}
We observe that the best solution time is obtained for the smoother of lower degree, even if the number of iterations to convergence is decreased with increasing $k$, as depicted in Fig.~\ref{fig:anisotropy_1024_soc1_itercount}.
\begin{figure}[p]
	\centering
%
%
\definecolor{mycolor1}{rgb}{0.00000,0.44700,0.74100}%
\definecolor{mycolor2}{rgb}{0.85000,0.32500,0.09800}%
\definecolor{mycolor3}{rgb}{0.92900,0.69400,0.12500}%
\definecolor{mycolor4}{rgb}{0.49400,0.18400,0.55600}%
\begin{tikzpicture}

\begin{axis}[%
width=0.33\columnwidth,
height=0.168\columnwidth,
at={(0\columnwidth,0.463\columnwidth)},
scale only axis,
bar shift auto,
xmin=0.4,
xmax=5.6,
xtick={1,2,3,4,5},
xticklabels={{64},{128},{256},{512},{1024}},ticklabel style = {font=\scriptsize},
xlabel style={font=\scriptsize},
ybar=2*\pgflinewidth,
ymin=0,
ymax=611,
ylabel style={font=\scriptsize},
ylabel={Iterations},
axis background/.style={fill=white},
title style={font=\scriptsize,yshift=-0.8em},
title={$k=4$ degree/$\ell_1$-sweeps},
xmajorgrids,
ymajorgrids,
legend columns=4,
legend style={at={(-0.3,1.4)}, anchor=north west, legend cell align=left, align=left, draw=none, fill=none, font=\scriptsize}
]
\addplot[ybar, bar width=0.145, fill=mycolor1, draw=black, area legend] table[row sep=crcr] {%
1	160\\
2	209\\
3	256\\
4	391\\
5	569\\
};
\addplot[forget plot, color=white!15!black] table[row sep=crcr] {%
0.654545454545455	0\\
5.34545454545454	0\\
};
\addlegendentry{Chebyshev 4}

\addplot[ybar, bar width=0.145, fill=mycolor2, draw=black, area legend] table[row sep=crcr] {%
1	159\\
2	209\\
3	254\\
4	387\\
5	562\\
};
\addplot[forget plot, color=white!15!black] table[row sep=crcr] {%
0.654545454545455	0\\
5.34545454545454	0\\
};
\addlegendentry{Opt. Chebyshev 4}

\addplot[ybar, bar width=0.145, fill=mycolor3, draw=black, area legend] table[row sep=crcr] {%
1	159\\
2	209\\
3	254\\
4	391\\
5	567\\
};
\addplot[forget plot, color=white!15!black] table[row sep=crcr] {%
0.654545454545455	0\\
5.34545454545454	0\\
};
\addlegendentry{Opt. Chebyshev 1}

\addplot[ybar, bar width=0.145, fill=mycolor4, draw=black, area legend] table[row sep=crcr] {%
1	173\\
2	225\\
3	272\\
4	420\\
5	611\\
};
\addplot[forget plot, color=white!15!black] table[row sep=crcr] {%
0.654545454545455	0\\
5.34545454545454	0\\
};
\addlegendentry{L1-Jacobi}

\end{axis}

\begin{axis}[%
width=0.33\columnwidth,
height=0.168\columnwidth,
at={(0.378\columnwidth,0.463\columnwidth)},
scale only axis,
bar shift auto,
xmin=0.4,
xmax=5.6,
xtick={1,2,3,4,5},
xticklabels={{64},{128},{256},{512},{1024}},ticklabel style = {font=\scriptsize},
xlabel style={font=\scriptsize},
ybar=2*\pgflinewidth,
ymin=0,
ymax=611,
axis background/.style={fill=white},
title style={font=\scriptsize,yshift=-0.8em},
title={$k=6$ degree/$\ell_1$-sweeps},
xmajorgrids,
ymajorgrids
]
\addplot[ybar, bar width=0.145, fill=mycolor1, draw=black, area legend] table[row sep=crcr] {%
1	149\\
2	196\\
3	238\\
4	361\\
5	519\\
};
\addplot[forget plot, color=white!15!black] table[row sep=crcr] {%
0.654545454545455	0\\
5.34545454545454	0\\
};
\addplot[ybar, bar width=0.145, fill=mycolor2, draw=black, area legend] table[row sep=crcr] {%
1	148\\
2	194\\
3	234\\
4	354\\
5	507\\
};
\addplot[forget plot, color=white!15!black] table[row sep=crcr] {%
0.654545454545455	0\\
5.34545454545454	0\\
};
\addplot[ybar, bar width=0.145, fill=mycolor3, draw=black, area legend] table[row sep=crcr] {%
1	151\\
2	198\\
3	240\\
4	363\\
5	523\\
};
\addplot[forget plot, color=white!15!black] table[row sep=crcr] {%
0.654545454545455	0\\
5.34545454545454	0\\
};
\addplot[ybar, bar width=0.145, fill=mycolor4, draw=black, area legend] table[row sep=crcr] {%
1	167\\
2	218\\
3	264\\
4	410\\
5	599\\
};
\addplot[forget plot, color=white!15!black] table[row sep=crcr] {%
0.654545454545455	0\\
5.34545454545454	0\\
};
\end{axis}

\begin{axis}[%
width=0.33\columnwidth,
height=0.168\columnwidth,
at={(0\columnwidth,0.232\columnwidth)},
scale only axis,
bar shift auto,
xmin=0.4,
xmax=5.6,
xtick={1,2,3,4,5},
xticklabels={{64},{128},{256},{512},{1024}},ticklabel style = {font=\scriptsize},
xlabel style={font=\scriptsize},
ybar=2*\pgflinewidth,
ymin=0,
ymax=611,
ylabel style={font=\scriptsize},
ylabel={Iterations},
axis background/.style={fill=white},
title style={font=\scriptsize,yshift=-0.8em},
title={$k=8$ degree/$\ell_1$-sweeps},
xmajorgrids,
ymajorgrids
]
\addplot[ybar, bar width=0.145, fill=mycolor1, draw=black, area legend] table[row sep=crcr] {%
1	143\\
2	181\\
3	223\\
4	331\\
5	474\\
};
\addplot[forget plot, color=white!15!black] table[row sep=crcr] {%
0.654545454545455	0\\
5.34545454545454	0\\
};
\addplot[ybar, bar width=0.145, fill=mycolor2, draw=black, area legend] table[row sep=crcr] {%
1	141\\
2	180\\
3	217\\
4	322\\
5	457\\
};
\addplot[forget plot, color=white!15!black] table[row sep=crcr] {%
0.654545454545455	0\\
5.34545454545454	0\\
};
\addplot[ybar, bar width=0.145, fill=mycolor3, draw=black, area legend] table[row sep=crcr] {%
1	145\\
2	185\\
3	225\\
4	338\\
5	482\\
};
\addplot[forget plot, color=white!15!black] table[row sep=crcr] {%
0.654545454545455	0\\
5.34545454545454	0\\
};
\addplot[ybar, bar width=0.145, fill=mycolor4, draw=black, area legend] table[row sep=crcr] {%
1	163\\
2	213\\
3	260\\
4	402\\
5	587\\
};
\addplot[forget plot, color=white!15!black] table[row sep=crcr] {%
0.654545454545455	0\\
5.34545454545454	0\\
};
\end{axis}

\begin{axis}[%
width=0.33\columnwidth,
height=0.168\columnwidth,
at={(0.378\columnwidth,0.232\columnwidth)},
scale only axis,
bar shift auto,
xmin=0.4,
xmax=5.6,
xtick={1,2,3,4,5},
xticklabels={{64},{128},{256},{512},{1024}},ticklabel style = {font=\scriptsize},
xlabel style={font=\scriptsize},
xlabel={Number of GPUs},
ybar=2*\pgflinewidth,
ymin=0,
ymax=611,
axis background/.style={fill=white},
title style={font=\scriptsize,yshift=-0.8em},
title={$k=10$ degree/$\ell_1$-sweeps},
xmajorgrids,
ymajorgrids
]
\addplot[ybar, bar width=0.145, fill=mycolor1, draw=black, area legend] table[row sep=crcr] {%
1	138\\
2	173\\
3	207\\
4	310\\
5	429\\
};
\addplot[forget plot, color=white!15!black] table[row sep=crcr] {%
0.654545454545455	0\\
5.34545454545454	0\\
};
\addplot[ybar, bar width=0.145, fill=mycolor2, draw=black, area legend] table[row sep=crcr] {%
1	137\\
2	171\\
3	202\\
4	301\\
5	413\\
};
\addplot[forget plot, color=white!15!black] table[row sep=crcr] {%
0.654545454545455	0\\
5.34545454545454	0\\
};
\addplot[ybar, bar width=0.145, fill=mycolor3, draw=black, area legend] table[row sep=crcr] {%
1	140\\
2	177\\
3	209\\
4	318\\
5	443\\
};
\addplot[forget plot, color=white!15!black] table[row sep=crcr] {%
0.654545454545455	0\\
5.34545454545454	0\\
};
\addplot[ybar, bar width=0.145, fill=mycolor4, draw=black, area legend] table[row sep=crcr] {%
1	160\\
2	210\\
3	256\\
4	395\\
5	577\\
};
\addplot[forget plot, color=white!15!black] table[row sep=crcr] {%
0.654545454545455	0\\
5.34545454545454	0\\
};
\end{axis}

\begin{axis}[%
width=0.33\columnwidth,
height=0.168\columnwidth,
at={(0\columnwidth,0\columnwidth)},
scale only axis,
bar shift auto,
xmin=0.4,
xmax=5.6,
xtick={1,2,3,4,5},
xticklabels={{64},{128},{256},{512},{1024}},ticklabel style = {font=\scriptsize},
xlabel style={font=\scriptsize},
xlabel={Number of GPUs},
ybar=2*\pgflinewidth,
ymin=0,
ymax=611,
ylabel style={font=\scriptsize},
ylabel={Iterations},
axis background/.style={fill=white},
title style={font=\scriptsize,yshift=-0.8em},
title={$k=12$ degree/$\ell_1$-sweeps},
xmajorgrids,
ymajorgrids
]
\addplot[ybar, bar width=0.145, fill=mycolor1, draw=black, area legend] table[row sep=crcr] {%
1	133\\
2	165\\
3	194\\
4	291\\
5	398\\
};
\addplot[forget plot, color=white!15!black] table[row sep=crcr] {%
0.654545454545455	0\\
5.34545454545454	0\\
};
\addplot[ybar, bar width=0.145, fill=mycolor2, draw=black, area legend] table[row sep=crcr] {%
1	132\\
2	164\\
3	188\\
4	275\\
5	380\\
};
\addplot[forget plot, color=white!15!black] table[row sep=crcr] {%
0.654545454545455	0\\
5.34545454545454	0\\
};
\addplot[ybar, bar width=0.145, fill=mycolor3, draw=black, area legend] table[row sep=crcr] {%
1	135\\
2	170\\
3	202\\
4	301\\
5	414\\
};
\addplot[forget plot, color=white!15!black] table[row sep=crcr] {%
0.654545454545455	0\\
5.34545454545454	0\\
};
\addplot[ybar, bar width=0.145, fill=mycolor4, draw=black, area legend] table[row sep=crcr] {%
1	158\\
2	207\\
3	253\\
4	390\\
5	568\\
};
\addplot[forget plot, color=white!15!black] table[row sep=crcr] {%
0.654545454545455	0\\
5.34545454545454	0\\
};
\end{axis}
\end{tikzpicture}%
	\caption{Iteration count. Total number of FCG iterations for the Anisotropy 2D test case $(\varepsilon=100,\theta=\pi/6)$ employing the multigrid scheme based on the VBM aggregation in conjunction with the different polynomial accelerators and the baseline $\ell_1$-Jacobi for comparison.}
	\label{fig:anisotropy_1024_soc1_itercount}
\end{figure}
The decrease is not sufficient to compensate for the increase in application cost, i.e., the increase in the number of sparse matrix-vector products.

Finally we focus on the scalability of implementing different polynomial smoothers on the GPU, looking at the time per iteration in Fig.~\ref{fig:anisotropy_1024_soc1_timeperiteration}.
\begin{figure}[p]
	\centering
%
%
\definecolor{mycolor1}{rgb}{0.00000,0.44700,0.74100}%
\definecolor{mycolor2}{rgb}{0.85000,0.32500,0.09800}%
\definecolor{mycolor3}{rgb}{0.92900,0.69400,0.12500}%
\definecolor{mycolor4}{rgb}{0.49400,0.18400,0.55600}%
\begin{tikzpicture}

\begin{axis}[%
width=0.33\columnwidth,
height=0.168\columnwidth,
at={(0\columnwidth,0.467\columnwidth)},
scale only axis,
bar shift auto,
xmin=0.4,
xmax=5.6,
xtick={1,2,3,4,5},
xticklabels={{64},{128},{256},{512},{1024}},ticklabel style = {font=\scriptsize},
ymin=0,
ymax=0.535495,
ybar=2*\pgflinewidth,
ylabel style={font=\scriptsize},
ylabel={Time per iteration (s)},
axis background/.style={fill=white},
title style={font=\scriptsize,yshift=-0.8em},
title={$k = 4$ degree/$\ell_1$-Jacobi sweeps},
xmajorgrids,
ymajorgrids,
legend columns=4,
legend style={at={(-0.3,1.4)}, anchor=north west, legend cell align=left, align=left, draw=none, fill=none, font=\scriptsize}
]
\addplot[ybar, bar width=0.145, fill=mycolor1, draw=black, area legend] table[row sep=crcr] {%
1	0.188382\\
2	0.215114\\
3	0.227431\\
4	0.246505\\
5	0.274518\\
};
\addplot[forget plot, color=white!15!black] table[row sep=crcr] {%
0.654545454545455	0\\
5.34545454545454	0\\
};
\addlegendentry{Chebyshev 4}

\addplot[ybar, bar width=0.145, fill=mycolor2, draw=black, area legend] table[row sep=crcr] {%
1	0.190274\\
2	0.204233\\
3	0.231501\\
4	0.249125\\
5	0.267846\\
};
\addplot[forget plot, color=white!15!black] table[row sep=crcr] {%
0.654545454545455	0\\
5.34545454545454	0\\
};
\addlegendentry{Opt. Chebyshev 4}

\addplot[ybar, bar width=0.145, fill=mycolor3, draw=black, area legend] table[row sep=crcr] {%
1	0.185702\\
2	0.211789\\
3	0.233622\\
4	0.245654\\
5	0.279582\\
};
\addplot[forget plot, color=white!15!black] table[row sep=crcr] {%
0.654545454545455	0\\
5.34545454545454	0\\
};
\addlegendentry{Opt. Chebyshev 1}

\addplot[ybar, bar width=0.145, fill=mycolor4, draw=black, area legend] table[row sep=crcr] {%
1	0.199471\\
2	0.212359\\
3	0.218678\\
4	0.237151\\
5	0.269773\\
};
\addplot[forget plot, color=white!15!black] table[row sep=crcr] {%
0.654545454545455	0\\
5.34545454545454	0\\
};
\addlegendentry{L1-Jacobi}

\end{axis}

\begin{axis}[%
width=0.33\columnwidth,
height=0.168\columnwidth,
at={(0.378\columnwidth,0.467\columnwidth)},
scale only axis,
bar shift auto,
xmin=0.4,
xmax=5.6,
xtick={1,2,3,4,5},
xticklabels={{64},{128},{256},{512},{1024}},ticklabel style = {font=\scriptsize},
ymin=0,
ymax=0.535495,
ybar=2*\pgflinewidth,
ylabel style={font=\scriptsize},
axis background/.style={fill=white},
title style={font=\scriptsize,yshift=-0.8em},
title={$k = 6$ degree/$\ell_1$-Jacobi sweeps},
xmajorgrids,
ymajorgrids
]
\addplot[ybar, bar width=0.145, fill=mycolor1, draw=black, area legend] table[row sep=crcr] {%
1	0.252974\\
2	0.275256\\
3	0.267309\\
4	0.324439\\
5	0.338519\\
};
\addplot[forget plot, color=white!15!black] table[row sep=crcr] {%
0.654545454545455	0\\
5.34545454545454	0\\
};
\addplot[ybar, bar width=0.145, fill=mycolor2, draw=black, area legend] table[row sep=crcr] {%
1	0.241875\\
2	0.261892\\
3	0.277779\\
4	0.307932\\
5	0.350576\\
};
\addplot[forget plot, color=white!15!black] table[row sep=crcr] {%
0.654545454545455	0\\
5.34545454545454	0\\
};
\addplot[ybar, bar width=0.145, fill=mycolor3, draw=black, area legend] table[row sep=crcr] {%
1	0.242905\\
2	0.262088\\
3	0.291154\\
4	0.309067\\
5	0.360224\\
};
\addplot[forget plot, color=white!15!black] table[row sep=crcr] {%
0.654545454545455	0\\
5.34545454545454	0\\
};
\addplot[ybar, bar width=0.145, fill=mycolor4, draw=black, area legend] table[row sep=crcr] {%
1	0.233695\\
2	0.268069\\
3	0.282099\\
4	0.294468\\
5	0.370107\\
};
\addplot[forget plot, color=white!15!black] table[row sep=crcr] {%
0.654545454545455	0\\
5.34545454545454	0\\
};
\end{axis}

\begin{axis}[%
width=0.33\columnwidth,
height=0.168\columnwidth,
at={(0\columnwidth,0.233\columnwidth)},
scale only axis,
bar shift auto,
xmin=0.4,
xmax=5.6,
ybar=2*\pgflinewidth,
xtick={1,2,3,4,5},
xticklabels={{64},{128},{256},{512},{1024}},ticklabel style = {font=\scriptsize},
ymin=0,
ymax=0.535495,
ylabel style={font=\scriptsize},
ylabel={Time per iteration (s)},
axis background/.style={fill=white},
title style={font=\scriptsize,yshift=-0.8em},
title={$k = 8$ degree/$\ell_1$-Jacobi sweeps},
xmajorgrids,
ymajorgrids
]
\addplot[ybar, bar width=0.145, fill=mycolor1, draw=black, area legend] table[row sep=crcr] {%
1	0.289467\\
2	0.314433\\
3	0.344158\\
4	0.357377\\
5	0.401122\\
};
\addplot[forget plot, color=white!15!black] table[row sep=crcr] {%
0.654545454545455	0\\
5.34545454545454	0\\
};
\addplot[ybar, bar width=0.145, fill=mycolor2, draw=black, area legend] table[row sep=crcr] {%
1	0.293766\\
2	0.321247\\
3	0.339147\\
4	0.369962\\
5	0.442612\\
};
\addplot[forget plot, color=white!15!black] table[row sep=crcr] {%
0.654545454545455	0\\
5.34545454545454	0\\
};
\addplot[ybar, bar width=0.145, fill=mycolor3, draw=black, area legend] table[row sep=crcr] {%
1	0.297817\\
2	0.33724\\
3	0.339088\\
4	0.36413\\
5	0.393268\\
};
\addplot[forget plot, color=white!15!black] table[row sep=crcr] {%
0.654545454545455	0\\
5.34545454545454	0\\
};
\addplot[ybar, bar width=0.145, fill=mycolor4, draw=black, area legend] table[row sep=crcr] {%
1	0.292102\\
2	0.308739\\
3	0.344567\\
4	0.379823\\
5	0.404165\\
};
\addplot[forget plot, color=white!15!black] table[row sep=crcr] {%
0.654545454545455	0\\
5.34545454545454	0\\
};
\end{axis}

\begin{axis}[%
width=0.33\columnwidth,
height=0.168\columnwidth,
at={(0.378\columnwidth,0.233\columnwidth)},
scale only axis,
bar shift auto,
xmin=0.4,
xmax=5.6,
ybar=2*\pgflinewidth,
xtick={1,2,3,4,5},
xticklabels={{64},{128},{256},{512},{1024}},ticklabel style = {font=\scriptsize},
xlabel style={font=\color{white!15!black}},
xlabel={Number of GPUs},
ymin=0,
ymax=0.535495,
ylabel style={font=\scriptsize},
axis background/.style={fill=white},
title style={font=\scriptsize,yshift=-0.8em},
title={$k = 10$ degree/$\ell_1$-Jacobi sweeps},
xmajorgrids,
ymajorgrids
]
\addplot[ybar, bar width=0.145, fill=mycolor1, draw=black, area legend] table[row sep=crcr] {%
1	0.354329\\
2	0.382774\\
3	0.403526\\
4	0.434718\\
5	0.447462\\
};
\addplot[forget plot, color=white!15!black] table[row sep=crcr] {%
0.654545454545455	0\\
5.34545454545454	0\\
};
\addplot[ybar, bar width=0.145, fill=mycolor2, draw=black, area legend] table[row sep=crcr] {%
1	0.363237\\
2	0.368483\\
3	0.387586\\
4	0.437397\\
5	0.45473\\
};
\addplot[forget plot, color=white!15!black] table[row sep=crcr] {%
0.654545454545455	0\\
5.34545454545454	0\\
};
\addplot[ybar, bar width=0.145, fill=mycolor3, draw=black, area legend] table[row sep=crcr] {%
1	0.353285\\
2	0.358691\\
3	0.37913\\
4	0.41487\\
5	0.430846\\
};
\addplot[forget plot, color=white!15!black] table[row sep=crcr] {%
0.654545454545455	0\\
5.34545454545454	0\\
};
\addplot[ybar, bar width=0.145, fill=mycolor4, draw=black, area legend] table[row sep=crcr] {%
1	0.345411\\
2	0.377253\\
3	0.394266\\
4	0.425978\\
5	0.452435\\
};
\addplot[forget plot, color=white!15!black] table[row sep=crcr] {%
0.654545454545455	0\\
5.34545454545454	0\\
};
\end{axis}

\begin{axis}[%
width=0.33\columnwidth,
height=0.168\columnwidth,
at={(0\columnwidth,0\columnwidth)},
scale only axis,
bar shift auto,
xmin=0.4,
xmax=5.6,
ybar=2*\pgflinewidth,
xtick={1,2,3,4,5},
xticklabels={{64},{128},{256},{512},{1024}},ticklabel style = {font=\scriptsize},
xlabel style={font=\color{white!15!black}},
xlabel={Number of GPUs},
ymin=0,
ymax=0.535495,
ylabel style={font=\scriptsize},
ylabel={Time per iteration (s)},
axis background/.style={fill=white},
title style={font=\scriptsize,yshift=-0.8em},
title={$k = 12$ degree/$\ell_1$-Jacobi sweeps},
xmajorgrids,
ymajorgrids
]
\addplot[ybar, bar width=0.145, fill=mycolor1, draw=black, area legend] table[row sep=crcr] {%
1	0.408142\\
2	0.420145\\
3	0.453132\\
4	0.494481\\
5	0.516692\\
};
\addplot[forget plot, color=white!15!black] table[row sep=crcr] {%
0.654545454545455	0\\
5.34545454545454	0\\
};
\addplot[ybar, bar width=0.145, fill=mycolor2, draw=black, area legend] table[row sep=crcr] {%
1	0.388873\\
2	0.430762\\
3	0.45952\\
4	0.491894\\
5	0.535495\\
};
\addplot[forget plot, color=white!15!black] table[row sep=crcr] {%
0.654545454545455	0\\
5.34545454545454	0\\
};
\addplot[ybar, bar width=0.145, fill=mycolor3, draw=black, area legend] table[row sep=crcr] {%
1	0.384155\\
2	0.425286\\
3	0.430518\\
4	0.513773\\
5	0.527153\\
};
\addplot[forget plot, color=white!15!black] table[row sep=crcr] {%
0.654545454545455	0\\
5.34545454545454	0\\
};
\addplot[ybar, bar width=0.145, fill=mycolor4, draw=black, area legend] table[row sep=crcr] {%
1	0.398095\\
2	0.425516\\
3	0.444655\\
4	0.499055\\
5	0.519441\\
};
\addplot[forget plot, color=white!15!black] table[row sep=crcr] {%
0.654545454545455	0\\
5.34545454545454	0\\
};
\end{axis}
\end{tikzpicture}%
	\caption{Time per iteration. Time per iteration for the Anisotropy 2D test case $(\varepsilon=100,\theta=\pi/6)$ employing the multigrid scheme based on the VBM aggregation in conjunction with the different polynomial accelerators and the baseline $\ell_1$-Jacobi for comparison.}
	\label{fig:anisotropy_1024_soc1_timeperiteration}
\end{figure}
As the degree of the polynomial increases, the implementation scalability properties of the method are confirmed. The specific observation to be made is that the time per iteration always remains perfectly comparable with that of the basic smoother which is $\ell_1$-Jacobi and that this grows in the expected manner as the degree $k$ of the polynomial representing the smoother increases. In particular, this shows that our GPU implementation of the different methods behaves consistently with the degree of the polynomial.

\subsubsection{Compatible Matching}\label{sec:sm_match_aniso} 
In the paper, we only report the results for the VBM aggregation-based preconditioner, which achieves the best results on this problem. To complete the overview, we also consider here the solution employing the multigrid preconditioners with aggregation based on the compatible matching strategy while limiting ourselves, due to the limits in core-hours granted, to scaling results up to 512 GPUs. However, this is enough to describe the behavior of the algorithm. As for all the other cases, we start from Fig.~\ref{fig:anisotropy_1024_match_tsolve} where we report the solve time, and we observe that for low values of the degree we obtain both the best results and that the best result is the one with the quasi-optimal Chebyshev polynomials of 1\textsuperscript{st} degree. This confirms the larger impact of the optimized smoothers for this
type of coarsening scheme.
\begin{figure}[p]
	\centering
%
%
\definecolor{mycolor1}{rgb}{0.00000,0.44700,0.74100}%
\definecolor{mycolor2}{rgb}{0.85000,0.32500,0.09800}%
\definecolor{mycolor3}{rgb}{0.92900,0.69400,0.12500}%
\definecolor{mycolor4}{rgb}{0.49400,0.18400,0.55600}%
\begin{tikzpicture}

\begin{axis}[%
width=0.33\columnwidth,
height=0.168\columnwidth,
at={(0\columnwidth,0.513\columnwidth)},
scale only axis,
bar shift auto,
xmin=0.4,
xmax=4.6,
xtick={1,2,3,4},
xticklabels={{64},{128},{256},{512},{1024}},ticklabel style = {font=\scriptsize},
ybar=2*\pgflinewidth,
ymin=0,
ymax=321.91,
ylabel style={font=\scriptsize},
ylabel={Solve time (s)},
axis background/.style={fill=white},
title style={font=\scriptsize,yshift=-0.8em},
title={$k = 4$ degree/$\ell_1$-Jacobi sweeps},
xmajorgrids,
ymajorgrids,
legend columns=4,
legend style={at={(-0.3,1.4)}, anchor=north west, legend cell align=left, align=left, draw=none, fill=none, font=\scriptsize}
]
\addplot[ybar, bar width=0.145, fill=mycolor1, draw=black, area legend] table[row sep=crcr] {%
1	60.0782\\
2	87.3763\\
3	163.125\\
4	172.521\\
};
\addplot[forget plot, color=white!15!black] table[row sep=crcr] {%
0.654545454545455	0\\
4.34545454545454	0\\
};
\addlegendentry{Chebyshev 4}

\addplot[ybar, bar width=0.145, fill=mycolor2, draw=black, area legend] table[row sep=crcr] {%
1	61.0642\\
2	86.1187\\
3	170.53\\
4	171.81\\
};
\addplot[forget plot, color=white!15!black] table[row sep=crcr] {%
0.654545454545455	0\\
4.34545454545454	0\\
};
\addlegendentry{Opt. Chebyshev 4}

\addplot[ybar, bar width=0.145, fill=mycolor3, draw=black, area legend] table[row sep=crcr] {%
1	60.2353\\
2	86.4055\\
3	163.048\\
4	172.007\\
};
\addplot[forget plot, color=white!15!black] table[row sep=crcr] {%
0.654545454545455	0\\
4.34545454545454	0\\
};
\addlegendentry{Opt. Chebyshev 1}

\addplot[ybar, bar width=0.145, fill=mycolor4, draw=black, area legend] table[row sep=crcr] {%
1	61.5505\\
2	90.7259\\
3	190.288\\
4	184.531\\
};
\addplot[forget plot, color=white!15!black] table[row sep=crcr] {%
0.654545454545455	0\\
4.34545454545454	0\\
};
\addlegendentry{L1-Jacobi}

\end{axis}

\begin{axis}[%
width=0.33\columnwidth,
height=0.168\columnwidth,
at={(0.378\columnwidth,0.513\columnwidth)},
scale only axis,
bar shift auto,
xmin=0.4,
xmax=4.6,
xtick={1,2,3,4},
xticklabels={{64},{128},{256},{512},{1024}},ticklabel style = {font=\scriptsize},
ymin=0,
ybar=2*\pgflinewidth,
ymax=321.91,
axis background/.style={fill=white},
title style={font=\scriptsize,yshift=-0.8em},
title={$k = 6$ degree/$\ell_1$-Jacobi sweeps},
xmajorgrids,
ymajorgrids
]
\addplot[ybar, bar width=0.145, fill=mycolor1, draw=black, area legend] table[row sep=crcr] {%
1	70.6185\\
2	102.038\\
3	189.932\\
4	196.652\\
};
\addplot[forget plot, color=white!15!black] table[row sep=crcr] {%
0.654545454545455	0\\
4.34545454545454	0\\
};
\addplot[ybar, bar width=0.145, fill=mycolor2, draw=black, area legend] table[row sep=crcr] {%
1	70.1439\\
2	99.5092\\
3	185.548\\
4	192.249\\
};
\addplot[forget plot, color=white!15!black] table[row sep=crcr] {%
0.654545454545455	0\\
4.34545454545454	0\\
};
\addplot[ybar, bar width=0.145, fill=mycolor3, draw=black, area legend] table[row sep=crcr] {%
1	73.8137\\
2	102.154\\
3	187.392\\
4	196.937\\
};
\addplot[forget plot, color=white!15!black] table[row sep=crcr] {%
0.654545454545455	0\\
4.34545454545454	0\\
};
\addplot[ybar, bar width=0.145, fill=mycolor4, draw=black, area legend] table[row sep=crcr] {%
1	77.1042\\
2	112.249\\
3	213.77\\
4	220.45\\
};
\addplot[forget plot, color=white!15!black] table[row sep=crcr] {%
0.654545454545455	0\\
4.34545454545454	0\\
};
\end{axis}

\begin{axis}[%
width=0.33\columnwidth,
height=0.168\columnwidth,
at={(0\columnwidth,0.257\columnwidth)},
scale only axis,
bar shift auto,
xmin=0.4,
xmax=4.6,
xtick={1,2,3,4},
xticklabels={{64},{128},{256},{512},{1024}},ticklabel style = {font=\scriptsize},
ymin=0,
ybar=2*\pgflinewidth,
ymax=321.91,
ylabel style={font=\scriptsize},
ylabel={Solve time (s)},
axis background/.style={fill=white},
title style={font=\scriptsize,yshift=-0.8em},
title={$k = 8$ degree/$\ell_1$-Jacobi sweeps},
xmajorgrids,
ymajorgrids
]
\addplot[ybar, bar width=0.145, fill=mycolor1, draw=black, area legend] table[row sep=crcr] {%
1	77.5922\\
2	111.877\\
3	201.091\\
4	216.459\\
};
\addplot[forget plot, color=white!15!black] table[row sep=crcr] {%
0.654545454545455	0\\
4.34545454545454	0\\
};
\addplot[ybar, bar width=0.145, fill=mycolor2, draw=black, area legend] table[row sep=crcr] {%
1	74.739\\
2	108.749\\
3	182.919\\
4	211.954\\
};
\addplot[forget plot, color=white!15!black] table[row sep=crcr] {%
0.654545454545455	0\\
4.34545454545454	0\\
};
\addplot[ybar, bar width=0.145, fill=mycolor3, draw=black, area legend] table[row sep=crcr] {%
1	81.4718\\
2	115.073\\
3	198.202\\
4	220.297\\
};
\addplot[forget plot, color=white!15!black] table[row sep=crcr] {%
0.654545454545455	0\\
4.34545454545454	0\\
};
\addplot[ybar, bar width=0.145, fill=mycolor4, draw=black, area legend] table[row sep=crcr] {%
1	94.2121\\
2	132.64\\
3	252.312\\
4	257.065\\
};
\addplot[forget plot, color=white!15!black] table[row sep=crcr] {%
0.654545454545455	0\\
4.34545454545454	0\\
};
\end{axis}

\begin{axis}[%
width=0.33\columnwidth,
height=0.168\columnwidth,
at={(0.378\columnwidth,0.257\columnwidth)},
scale only axis,
bar shift auto,
xmin=0.4,
xmax=4.6,
xtick={1,2,3,4},
xticklabels={{64},{128},{256},{512},{1024}},
xlabel style={font=\scriptsize},ticklabel style = {font=\scriptsize},
xlabel={Number of GPUs},
ymin=0,
ybar=2*\pgflinewidth,
ymax=321.91,
axis background/.style={fill=white},
title style={font=\scriptsize,yshift=-0.8em},
title={$k = 10$ degree/$\ell_1$-Jacobi sweeps},
xmajorgrids,
ymajorgrids
]
\addplot[ybar, bar width=0.145, fill=mycolor1, draw=black, area legend] table[row sep=crcr] {%
1	84.4093\\
2	119.167\\
3	207.378\\
4	236.012\\
};
\addplot[forget plot, color=white!15!black] table[row sep=crcr] {%
0.654545454545455	0\\
4.34545454545454	0\\
};
\addplot[ybar, bar width=0.145, fill=mycolor2, draw=black, area legend] table[row sep=crcr] {%
1	80.1112\\
2	114.802\\
3	192.655\\
4	231.24\\
};
\addplot[forget plot, color=white!15!black] table[row sep=crcr] {%
0.654545454545455	0\\
4.34545454545454	0\\
};
\addplot[ybar, bar width=0.145, fill=mycolor3, draw=black, area legend] table[row sep=crcr] {%
1	87.9564\\
2	124.794\\
3	209.094\\
4	241.983\\
};
\addplot[forget plot, color=white!15!black] table[row sep=crcr] {%
0.654545454545455	0\\
4.34545454545454	0\\
};
\addplot[ybar, bar width=0.145, fill=mycolor4, draw=black, area legend] table[row sep=crcr] {%
1	107.632\\
2	153.231\\
3	285.305\\
4	294.201\\
};
\addplot[forget plot, color=white!15!black] table[row sep=crcr] {%
0.654545454545455	0\\
4.34545454545454	0\\
};
\end{axis}

\begin{axis}[%
width=0.33\columnwidth,
height=0.168\columnwidth,
at={(0\columnwidth,0\columnwidth)},
scale only axis,
bar shift auto,
xmin=0.4,
xmax=4.6,
xtick={1,2,3,4},
xticklabels={{64},{128},{256},{512},{1024}},
xlabel style={font=\scriptsize},ticklabel style = {font=\scriptsize},
xlabel={Number of GPUs},
ymin=0,
ybar=2*\pgflinewidth,
ymax=321.91,
ylabel style={font=\scriptsize},
ylabel={Solve time (s)},
axis background/.style={fill=white},
title style={font=\scriptsize,yshift=-0.8em},
title={$k = 12$ degree/$\ell_1$-Jacobi sweeps},
xmajorgrids,
ymajorgrids
]
\addplot[ybar, bar width=0.145, fill=mycolor1, draw=black, area legend] table[row sep=crcr] {%
1	91.0606\\
2	126.142\\
3	214.509\\
4	253.619\\
};
\addplot[forget plot, color=white!15!black] table[row sep=crcr] {%
0.654545454545455	0\\
4.34545454545454	0\\
};
\addplot[ybar, bar width=0.145, fill=mycolor2, draw=black, area legend] table[row sep=crcr] {%
1	85.9401\\
2	120.98\\
3	206.755\\
4	247.896\\
};
\addplot[forget plot, color=white!15!black] table[row sep=crcr] {%
0.654545454545455	0\\
4.34545454545454	0\\
};
\addplot[ybar, bar width=0.145, fill=mycolor3, draw=black, area legend] table[row sep=crcr] {%
1	94.7536\\
2	132.239\\
3	225.069\\
4	258.325\\
};
\addplot[forget plot, color=white!15!black] table[row sep=crcr] {%
0.654545454545455	0\\
4.34545454545454	0\\
};
\addplot[ybar, bar width=0.145, fill=mycolor4, draw=black, area legend] table[row sep=crcr] {%
1	119.841\\
2	172.529\\
3	321.208\\
4	321.91\\
};
\addplot[forget plot, color=white!15!black] table[row sep=crcr] {%
0.654545454545455	0\\
4.34545454545454	0\\
};
\end{axis}
\end{tikzpicture}%
	
	\caption{Solve time (s). Time to solution (s) for the Anisotropy 2D test case $(\varepsilon=100,\theta=\pi/6)$ employing the multigrid scheme based on the compatible matching aggregation in conjunction with the different polynomial accelerators and the baseline $\ell_1$-Jacobi for comparison.}
	\label{fig:anisotropy_1024_match_tsolve}
\end{figure}
If we turn now to the iteration count reported in Fig.~\ref{fig:anisotropy_1024_match_itercount} we observe again a behavior comparable to that analyzed in the case of the VBM aggregation, albeit with a higher overall number of iterations.
\begin{figure}[p]
	\centering
%
%
\definecolor{mycolor1}{rgb}{0.00000,0.44700,0.74100}%
\definecolor{mycolor2}{rgb}{0.85000,0.32500,0.09800}%
\definecolor{mycolor3}{rgb}{0.92900,0.69400,0.12500}%
\definecolor{mycolor4}{rgb}{0.49400,0.18400,0.55600}%
\begin{tikzpicture}

\begin{axis}[%
width=0.33\columnwidth,
height=0.168\columnwidth,
at={(0\columnwidth,0.514\columnwidth)},
scale only axis,
bar shift auto,
xmin=0.4,
xmax=4.6,
ybar=2*\pgflinewidth,
xtick={1,2,3,4},
xticklabels={{64},{128},{256},{512},{1024}},ticklabel style = {font=\scriptsize},
ymin=0,
ymax=913,
ylabel style={font=\scriptsize},
ylabel={Iterations},
axis background/.style={fill=white},
title style={font=\scriptsize,yshift=-0.8em},
title={$k = 4$ degree/$\ell_1$-Jacobi sweeps},
xmajorgrids,
ymajorgrids,
legend columns=4,
legend style={at={(-0.3,1.4)}, anchor=north west, legend cell align=left, align=left, draw=none, fill=none, font=\scriptsize}
]
\addplot[ybar, bar width=0.145, fill=mycolor1, draw=black, area legend] table[row sep=crcr] {%
1	352\\
2	498\\
4	835\\
3	606\\
};
\addplot[forget plot, color=white!15!black] table[row sep=crcr] {%
0.654545454545455	0\\
4.34545454545454	0\\
};
\addlegendentry{Chebyshev 4}

\addplot[ybar, bar width=0.145, fill=mycolor2, draw=black, area legend] table[row sep=crcr] {%
1	349\\
2	492\\
4	820\\
3	599\\
};
\addplot[forget plot, color=white!15!black] table[row sep=crcr] {%
0.654545454545455	0\\
4.34545454545454	0\\
};
\addlegendentry{Opt. Chebyshev 4}

\addplot[ybar, bar width=0.145, fill=mycolor3, draw=black, area legend] table[row sep=crcr] {%
1	352\\
2	497\\
4	833\\
3	604\\
};
\addplot[forget plot, color=white!15!black] table[row sep=crcr] {%
0.654545454545455	0\\
4.34545454545454	0\\
};
\addlegendentry{Opt. Chebyshev 1}

\addplot[ybar, bar width=0.145, fill=mycolor4, draw=black, area legend] table[row sep=crcr] {%
1	374\\
2	531\\
4	913\\
3	654\\
};
\addplot[forget plot, color=white!15!black] table[row sep=crcr] {%
0.654545454545455	0\\
4.34545454545454	0\\
};
\addlegendentry{L1-Jacobi}

\end{axis}

\begin{axis}[%
width=0.33\columnwidth,
height=0.168\columnwidth,
at={(0.378\columnwidth,0.514\columnwidth)},
scale only axis,
bar shift auto,
xmin=0.4,
xmax=4.6,
ybar=2*\pgflinewidth,
xtick={1,2,3,4},
xticklabels={{64},{128},{256},{512},{1024}},ticklabel style = {font=\scriptsize},
ymin=0,
ymax=913,
axis background/.style={fill=white},
title style={font=\scriptsize,yshift=-0.8em},
title={$k = 6$ degree/$\ell_1$-Jacobi sweeps},
xmajorgrids,
ymajorgrids
]
\addplot[ybar, bar width=0.145, fill=mycolor1, draw=black, area legend] table[row sep=crcr] {%
1	327\\
2	462\\
4	751\\
3	560\\
};
\addplot[forget plot, color=white!15!black] table[row sep=crcr] {%
0.654545454545455	0\\
4.34545454545454	0\\
};
\addplot[ybar, bar width=0.145, fill=mycolor2, draw=black, area legend] table[row sep=crcr] {%
1	321\\
2	451\\
4	727\\
3	549\\
};
\addplot[forget plot, color=white!15!black] table[row sep=crcr] {%
0.654545454545455	0\\
4.34545454545454	0\\
};
\addplot[ybar, bar width=0.145, fill=mycolor3, draw=black, area legend] table[row sep=crcr] {%
1	329\\
2	464\\
4	757\\
3	563\\
};
\addplot[forget plot, color=white!15!black] table[row sep=crcr] {%
0.654545454545455	0\\
4.34545454545454	0\\
};
\addplot[ybar, bar width=0.145, fill=mycolor4, draw=black, area legend] table[row sep=crcr] {%
1	366\\
2	520\\
4	890\\
3	637\\
};
\addplot[forget plot, color=white!15!black] table[row sep=crcr] {%
0.654545454545455	0\\
4.34545454545454	0\\
};
\end{axis}

\begin{axis}[%
width=0.33\columnwidth,
height=0.168\columnwidth,
at={(0\columnwidth,0.257\columnwidth)},
scale only axis,
bar shift auto,
xmin=0.4,
xmax=4.6,
ybar=2*\pgflinewidth,
xtick={1,2,3,4},
xticklabels={{64},{128},{256},{512},{1024}},ticklabel style = {font=\scriptsize},
ymin=0,
ymax=913,
ylabel style={font=\scriptsize},
ylabel={Iterations},
axis background/.style={fill=white},
title style={font=\scriptsize,yshift=-0.8em},
title={$k = 8$ degree/$\ell_1$-Jacobi sweeps},
xmajorgrids,
ymajorgrids
]
\addplot[ybar, bar width=0.145, fill=mycolor1, draw=black, area legend] table[row sep=crcr] {%
1	302\\
2	423\\
4	670\\
3	518\\
};
\addplot[forget plot, color=white!15!black] table[row sep=crcr] {%
0.654545454545455	0\\
4.34545454545454	0\\
};
\addplot[ybar, bar width=0.145, fill=mycolor2, draw=black, area legend] table[row sep=crcr] {%
1	292\\
2	410\\
4	641\\
3	506\\
};
\addplot[forget plot, color=white!15!black] table[row sep=crcr] {%
0.654545454545455	0\\
4.34545454545454	0\\
};
\addplot[ybar, bar width=0.145, fill=mycolor3, draw=black, area legend] table[row sep=crcr] {%
1	307\\
2	431\\
4	685\\
3	526\\
};
\addplot[forget plot, color=white!15!black] table[row sep=crcr] {%
0.654545454545455	0\\
4.34545454545454	0\\
};
\addplot[ybar, bar width=0.145, fill=mycolor4, draw=black, area legend] table[row sep=crcr] {%
1	361\\
2	511\\
4	871\\
3	624\\
};
\addplot[forget plot, color=white!15!black] table[row sep=crcr] {%
0.654545454545455	0\\
4.34545454545454	0\\
};
\end{axis}

\begin{axis}[%
width=0.33\columnwidth,
height=0.168\columnwidth,
at={(0.378\columnwidth,0.257\columnwidth)},
scale only axis,
bar shift auto,
xmin=0.4,
xmax=4.6,
ybar=2*\pgflinewidth,
xtick={1,2,3,4},
xticklabels={{64},{128},{256},{512},{1024}},ticklabel style = {font=\scriptsize},
xlabel style={font=\scriptsize},ticklabel style = {font=\scriptsize},
xlabel={Number of GPUs},
ymin=0,
ymax=913,
axis background/.style={fill=white},
title style={font=\scriptsize,yshift=-0.8em},
title={$k = 10$ degree/$\ell_1$-Jacobi sweeps},
xmajorgrids,
ymajorgrids
]
\addplot[ybar, bar width=0.145, fill=mycolor1, draw=black, area legend] table[row sep=crcr] {%
1	277\\
2	388\\
4	599\\
3	487\\
};
\addplot[forget plot, color=white!15!black] table[row sep=crcr] {%
0.654545454545455	0\\
4.34545454545454	0\\
};
\addplot[ybar, bar width=0.145, fill=mycolor2, draw=black, area legend] table[row sep=crcr] {%
1	266\\
2	372\\
4	569\\
3	474\\
};
\addplot[forget plot, color=white!15!black] table[row sep=crcr] {%
0.654545454545455	0\\
4.34545454545454	0\\
};
\addplot[ybar, bar width=0.145, fill=mycolor3, draw=black, area legend] table[row sep=crcr] {%
1	285\\
2	399\\
4	622\\
3	497\\
};
\addplot[forget plot, color=white!15!black] table[row sep=crcr] {%
0.654545454545455	0\\
4.34545454545454	0\\
};
\addplot[ybar, bar width=0.145, fill=mycolor4, draw=black, area legend] table[row sep=crcr] {%
1	355\\
2	503\\
4	854\\
3	613\\
};
\addplot[forget plot, color=white!15!black] table[row sep=crcr] {%
0.654545454545455	0\\
4.34545454545454	0\\
};
\end{axis}

\begin{axis}[%
width=0.33\columnwidth,
height=0.168\columnwidth,
at={(0\columnwidth,0\columnwidth)},
scale only axis,
bar shift auto,
xmin=0.4,
xmax=4.6,
ybar=2*\pgflinewidth,
xtick={1,2,3,4},
xticklabels={{64},{128},{256},{512},{1024}},ticklabel style = {font=\scriptsize},
xlabel style={font=\scriptsize},ticklabel style = {font=\scriptsize},
xlabel={Number of GPUs},
ymin=0,
ymax=913,
ylabel style={font=\scriptsize},
ylabel={Iterations},
axis background/.style={fill=white},
title style={font=\scriptsize,yshift=-0.8em},
title={$k = 12$ degree/$\ell_1$-Jacobi sweeps},
xmajorgrids,
ymajorgrids
]
\addplot[ybar, bar width=0.145, fill=mycolor1, draw=black, area legend] table[row sep=crcr] {%
1	256\\
2	356\\
4	543\\
3	461\\
};
\addplot[forget plot, color=white!15!black] table[row sep=crcr] {%
0.654545454545455	0\\
4.34545454545454	0\\
};
\addplot[ybar, bar width=0.145, fill=mycolor2, draw=black, area legend] table[row sep=crcr] {%
1	245\\
2	341\\
4	515\\
3	450\\
};
\addplot[forget plot, color=white!15!black] table[row sep=crcr] {%
0.654545454545455	0\\
4.34545454545454	0\\
};
\addplot[ybar, bar width=0.145, fill=mycolor3, draw=black, area legend] table[row sep=crcr] {%
1	266\\
2	372\\
4	569\\
3	474\\
};
\addplot[forget plot, color=white!15!black] table[row sep=crcr] {%
0.654545454545455	0\\
4.34545454545454	0\\
};
\addplot[ybar, bar width=0.145, fill=mycolor4, draw=black, area legend] table[row sep=crcr] {%
1	351\\
2	497\\
4	838\\
3	604\\
};
\addplot[forget plot, color=white!15!black] table[row sep=crcr] {%
0.654545454545455	0\\
4.34545454545454	0\\
};
\end{axis}
\end{tikzpicture}%
	
	\caption{Iteration count. Total number of FCG iterations for the Anisotropy 2D test case $(\varepsilon=100,\theta=\pi/6)$ employing the multigrid scheme based on the compatible matching aggregation in conjunction with the different polynomial accelerators and the baseline $\ell_1$-Jacobi for comparison.}
	\label{fig:anisotropy_1024_match_itercount}
\end{figure}
The behavior with respect to the implementation scalability of the smoothers remains unchanged with respect to the case with the VBM aggregation strategy as shown in~Fig.~\ref{fig:anisotropy_1024_match_timeperiteration}.
\begin{figure}[p]
	\centering
%
%
\definecolor{mycolor1}{rgb}{0.00000,0.44700,0.74100}%
\definecolor{mycolor2}{rgb}{0.85000,0.32500,0.09800}%
\definecolor{mycolor3}{rgb}{0.92900,0.69400,0.12500}%
\definecolor{mycolor4}{rgb}{0.49400,0.18400,0.55600}%
\begin{tikzpicture}

\begin{axis}[%
width=0.33\columnwidth,
height=0.168\columnwidth,
at={(0\columnwidth,0.513\columnwidth)},
scale only axis,
bar shift auto,
ybar=2*\pgflinewidth,
xmin=0.4,
xmax=4.6,
xtick={1,2,3,4},
xticklabels={{64},{128},{256},{512},{1024}},ticklabel style = {font=\scriptsize},
ymin=0,
ymax=0.54,
ylabel style={font=\scriptsize},
ylabel={Time per iteration (s)},
axis background/.style={fill=white},
title style={font=\scriptsize,yshift=-0.8em},
title={$k = 4$ degree/$\ell_1$-Jacobi sweeps},
xmajorgrids,
ymajorgrids,
legend columns=4,
legend style={at={(-0.3,1.4)}, anchor=north west, legend cell align=left, align=left, draw=none, fill=none, font=\scriptsize}
]
\addplot[ybar, bar width=0.145, fill=mycolor1, draw=black, area legend] table[row sep=crcr] {%
1	0.170677\\
2	0.175454\\
3	0.195359\\
4	0.284689\\
};
\addplot[forget plot, color=white!15!black] table[row sep=crcr] {%
0.654545454545455	0\\
4.34545454545454	0\\
};
\addlegendentry{Chebyshev 4}

\addplot[ybar, bar width=0.145, fill=mycolor2, draw=black, area legend] table[row sep=crcr] {%
1	0.174969\\
2	0.175038\\
3	0.207963\\
4	0.286827\\
};
\addplot[forget plot, color=white!15!black] table[row sep=crcr] {%
0.654545454545455	0\\
4.34545454545454	0\\
};
\addlegendentry{Opt. Chebyshev 4}

\addplot[ybar, bar width=0.145, fill=mycolor3, draw=black, area legend] table[row sep=crcr] {%
1	0.171123\\
2	0.173854\\
3	0.195735\\
4	0.284781\\
};
\addplot[forget plot, color=white!15!black] table[row sep=crcr] {%
0.654545454545455	0\\
4.34545454545454	0\\
};
\addlegendentry{Opt. Chebyshev 1}

\addplot[ybar, bar width=0.145, fill=mycolor4, draw=black, area legend] table[row sep=crcr] {%
1	0.164574\\
2	0.170859\\
3	0.208421\\
4	0.282158\\
};
\addplot[forget plot, color=white!15!black] table[row sep=crcr] {%
0.654545454545455	0\\
4.34545454545454	0\\
};
\addlegendentry{L1-Jacobi}

\end{axis}

\begin{axis}[%
width=0.33\columnwidth,
height=0.168\columnwidth,
at={(0.378\columnwidth,0.513\columnwidth)},
scale only axis,
bar shift auto,
ybar=2*\pgflinewidth,
xmin=0.4,
xmax=4.6,
xtick={1,2,3,4},
xticklabels={{64},{128},{256},{512},{1024}},ticklabel style = {font=\scriptsize},
ymin=0,
ymax=0.54,
axis background/.style={fill=white},
title style={font=\scriptsize,yshift=-0.8em},
title={$k = 6$ degree/$\ell_1$-Jacobi sweeps},
xmajorgrids,
ymajorgrids
]
\addplot[ybar, bar width=0.145, fill=mycolor1, draw=black, area legend] table[row sep=crcr] {%
1	0.215959\\
2	0.220862\\
3	0.252905\\
4	0.351165\\
};
\addplot[forget plot, color=white!15!black] table[row sep=crcr] {%
0.654545454545455	0\\
4.34545454545454	0\\
};
\addplot[ybar, bar width=0.145, fill=mycolor2, draw=black, area legend] table[row sep=crcr] {%
1	0.218517\\
2	0.220641\\
3	0.255224\\
4	0.35018\\
};
\addplot[forget plot, color=white!15!black] table[row sep=crcr] {%
0.654545454545455	0\\
4.34545454545454	0\\
};
\addplot[ybar, bar width=0.145, fill=mycolor3, draw=black, area legend] table[row sep=crcr] {%
1	0.224358\\
2	0.220159\\
3	0.247546\\
4	0.3498\\
};
\addplot[forget plot, color=white!15!black] table[row sep=crcr] {%
0.654545454545455	0\\
4.34545454545454	0\\
};
\addplot[ybar, bar width=0.145, fill=mycolor4, draw=black, area legend] table[row sep=crcr] {%
1	0.210667\\
2	0.215863\\
3	0.240191\\
4	0.346075\\
};
\addplot[forget plot, color=white!15!black] table[row sep=crcr] {%
0.654545454545455	0\\
4.34545454545454	0\\
};
\end{axis}

\begin{axis}[%
width=0.33\columnwidth,
height=0.168\columnwidth,
at={(0\columnwidth,0.256\columnwidth)},
scale only axis,
bar shift auto,
ybar=2*\pgflinewidth,
xmin=0.4,
xmax=4.6,
xtick={1,2,3,4},
xticklabels={{64},{128},{256},{512},{1024}},ticklabel style = {font=\scriptsize},
ymin=0,
ymax=0.54,
ylabel style={font=\scriptsize},
ylabel={Time per iteration (s)},
axis background/.style={fill=white},
title style={font=\scriptsize,yshift=-0.8em},
title={$k = 8$ degree/$\ell_1$-Jacobi sweeps},
xmajorgrids,
ymajorgrids
]
\addplot[ybar, bar width=0.145, fill=mycolor1, draw=black, area legend] table[row sep=crcr] {%
1	0.256928\\
2	0.264484\\
3	0.300135\\
4	0.417875\\
};
\addplot[forget plot, color=white!15!black] table[row sep=crcr] {%
0.654545454545455	0\\
4.34545454545454	0\\
};
\addplot[ybar, bar width=0.145, fill=mycolor2, draw=black, area legend] table[row sep=crcr] {%
1	0.255955\\
2	0.265242\\
3	0.285364\\
4	0.418882\\
};
\addplot[forget plot, color=white!15!black] table[row sep=crcr] {%
0.654545454545455	0\\
4.34545454545454	0\\
};
\addplot[ybar, bar width=0.145, fill=mycolor3, draw=black, area legend] table[row sep=crcr] {%
1	0.26538\\
2	0.26699\\
3	0.289346\\
4	0.418816\\
};
\addplot[forget plot, color=white!15!black] table[row sep=crcr] {%
0.654545454545455	0\\
4.34545454545454	0\\
};
\addplot[ybar, bar width=0.145, fill=mycolor4, draw=black, area legend] table[row sep=crcr] {%
1	0.260975\\
2	0.25957\\
3	0.289681\\
4	0.411964\\
};
\addplot[forget plot, color=white!15!black] table[row sep=crcr] {%
0.654545454545455	0\\
4.34545454545454	0\\
};
\end{axis}

\begin{axis}[%
width=0.33\columnwidth,
height=0.168\columnwidth,
at={(0.378\columnwidth,0.256\columnwidth)},
scale only axis,
bar shift auto,
ybar=2*\pgflinewidth,
xmin=0.4,
xmax=4.6,
xtick={1,2,3,4},
xticklabels={{64},{128},{256},{512},{1024}},ticklabel style = {font=\scriptsize},
xlabel style={font=\color{white!15!black}},
xlabel={Number of GPUs},
ymin=0,
ymax=0.54,
axis background/.style={fill=white},
title style={font=\scriptsize,yshift=-0.8em},
title={$k = 10$ degree/$\ell_1$-Jacobi sweeps},
xmajorgrids,
ymajorgrids
]
\addplot[ybar, bar width=0.145, fill=mycolor1, draw=black, area legend] table[row sep=crcr] {%
1	0.304727\\
2	0.30713\\
3	0.346206\\
4	0.484625\\
};
\addplot[forget plot, color=white!15!black] table[row sep=crcr] {%
0.654545454545455	0\\
4.34545454545454	0\\
};
\addplot[ybar, bar width=0.145, fill=mycolor2, draw=black, area legend] table[row sep=crcr] {%
1	0.30117\\
2	0.308608\\
3	0.338585\\
4	0.487847\\
};
\addplot[forget plot, color=white!15!black] table[row sep=crcr] {%
0.654545454545455	0\\
4.34545454545454	0\\
};
\addplot[ybar, bar width=0.145, fill=mycolor3, draw=black, area legend] table[row sep=crcr] {%
1	0.308619\\
2	0.312766\\
3	0.336164\\
4	0.486887\\
};
\addplot[forget plot, color=white!15!black] table[row sep=crcr] {%
0.654545454545455	0\\
4.34545454545454	0\\
};
\addplot[ybar, bar width=0.145, fill=mycolor4, draw=black, area legend] table[row sep=crcr] {%
1	0.303188\\
2	0.304634\\
3	0.33408\\
4	0.479936\\
};
\addplot[forget plot, color=white!15!black] table[row sep=crcr] {%
0.654545454545455	0\\
4.34545454545454	0\\
};
\end{axis}

\begin{axis}[%
width=0.33\columnwidth,
height=0.168\columnwidth,
at={(0\columnwidth,0\columnwidth)},
scale only axis,
bar shift auto,
ybar=2*\pgflinewidth,
xmin=0.4,
xmax=4.6,
xtick={1,2,3,4},
xticklabels={{64},{128},{256},{512},{1024}},ticklabel style = {font=\scriptsize},
ymin=0,
ymax=0.54,
ylabel style={font=\scriptsize},
ylabel={Time per iteration (s)},
axis background/.style={fill=white},
title style={font=\scriptsize,yshift=-0.8em},
title={$k = 12$ degree/$\ell_1$-Jacobi sweeps},
xmajorgrids,
ymajorgrids
]
\addplot[ybar, bar width=0.145, fill=mycolor1, draw=black, area legend] table[row sep=crcr] {%
1	0.355705\\
2	0.354331\\
3	0.395044\\
4	0.550149\\
};
\addplot[forget plot, color=white!15!black] table[row sep=crcr] {%
0.654545454545455	0\\
4.34545454545454	0\\
};
\addplot[ybar, bar width=0.145, fill=mycolor2, draw=black, area legend] table[row sep=crcr] {%
1	0.350776\\
2	0.354781\\
3	0.401467\\
4	0.55088\\
};
\addplot[forget plot, color=white!15!black] table[row sep=crcr] {%
0.654545454545455	0\\
4.34545454545454	0\\
};
\addplot[ybar, bar width=0.145, fill=mycolor3, draw=black, area legend] table[row sep=crcr] {%
1	0.356216\\
2	0.355481\\
3	0.395551\\
4	0.54499\\
};
\addplot[forget plot, color=white!15!black] table[row sep=crcr] {%
0.654545454545455	0\\
4.34545454545454	0\\
};
\addplot[ybar, bar width=0.145, fill=mycolor4, draw=black, area legend] table[row sep=crcr] {%
1	0.341427\\
2	0.347142\\
3	0.383303\\
4	0.532963\\
};
\addplot[forget plot, color=white!15!black] table[row sep=crcr] {%
0.654545454545455	0\\
4.34545454545454	0\\
};
\end{axis}
\end{tikzpicture}%
	
	\caption{Time per iteration. Time per iteration for the Anisotropy 2D test case $(\varepsilon=100,\theta=\pi/6)$ employing the multigrid scheme based on the compatible matching aggregation in conjunction with the different polynomial accelerators and the baseline $\ell_1$-Jacobi for comparison.}
	\label{fig:anisotropy_1024_match_timeperiteration}
\end{figure}

\section{Code snippets}\label{sm:code-snippet}
In this section we report the instructions needed to configure and setup the preconditioners used in the numerical experiments, together with the calls needed for the assembly phase on the GPU accelerator. To allow the possibility of a code in which the chosen matrix storage format is changed at runtime, either to accommodate different algorithmic needs or to be able to execute the same code on the CPU or GPU by changing only the allocation type, \texttt{PSCToolkit} exploits the Fortran \lstinline{MOLD} option and a dynamic data type management. The \lstinline{MOLD} option in Fortran is routinely used in the \lstinline{ALLOCATE} statements to create a new array or scalar variable that has the same type and type parameters as an existing array or scalar. We have adapted the same mechanism to work with the distributed matrix and vector data formats in \texttt{PSBLAS}. For the examples in Section~\ref{sec:results} we have used the GPU version of the Ellpack format, this can be achieved by instantiating the pointer variables:
\begin{lstlisting}
type(psb_d_cuda_hlg_sparse_mat), target :: ahlg
type(psb_d_vect_cuda), target           :: dvgpu
type(psb_i_vect_cuda), target           :: ivgpu
\end{lstlisting}
and pass them in the assembly phases of the matrix, descriptor and vector objects so that they are instantiated and built on the device,
\begin{lstlisting}
call psb_cdasb(desc,info,mold=ivgpu)
call psb_spasb(a,desc,info,mold=ahlg)
call psb_geasb(b,desc,info,mold=dvgpu)
\end{lstlisting}
where \lstinline{a}, \lstinline{b} and \lstinline{desc} are the sparse matrix $A$, the right-hand side $\mathbf{b}$ and the descriptor respectively. Setting and constructing the AMG preconditioners on the matrix $A$ is equally easy. This is an object of type
\begin{lstlisting}
type(amg_dprec_type)  :: prec
\end{lstlisting}
and after being declared, and the assembly of the matrix $A$ (\lstinline{call psb_spasb()}), can have options set with the following calls:
\begin{lstlisting}
call prec%init(ctxt,'ML',info)
call prec%set('ml_cycle','V-Cycle',info)
call prec%set('outer_sweeps',1,info)
call prec%set('aggr_prol','COUPLED',info)
call prec%set('par_aggr_alg','MATCHBOXP',   info)
call prec%set('aggr_type','SMOOTHED', info)
call prec%set('aggr_size',8, info)
call prec%set('smoother_type','POLY',info)
call prec%set('poly_degree',5,info)
call prec%set('poly_variant','CHEB_1_OPT',info)
call prec%set('coarse_solve','L1-JACOBI',info)
call prec%set('coarse_mat','DIST',info)
call prec%set('coarse_sweeps',30,info)
\end{lstlisting}
This set of calls sets up the V-cycle-based AMG preconditioner, which uses the matching-based aggregation scheme compatible with aggregates of size at most $8$. The smoother is the polynomial one based on Chebyshev polynomials of the 1\textsuperscript{st}-kind and degree $5$. The $\ell_1$-Jacobi method is used to solve the coarsest system ($30$ iterations).

Once we are satisfied with the algorithmic choices made we can ask for the assembly of the preconditioner, an operation that is done in two separate parts, one for the hierarchy and one for the smoothers and also ask to pre-allocate the work vectors on the device:
\begin{lstlisting}
call prec%hierarchy_build(a,desc,info)
call prec%smoothers_build(a,desc,info,amold=ahlg,vmold=dvgpu, &
& imold=ivgpu)
call prec%allocate_wrk(info,vmold=dvgpu)
\end{lstlisting}
Note that in general it is possible to reuse the same hierarchy with different smoothers, so if you change your mind---or want to try a different smoother with the same hierarchy---it is sufficient to make the appropriate set calls and call the function \lstinline!prec%smoothers_build( )!.

The solution employing the Flexible Conjugate Gradient can then be obtained by the call
\begin{lstlisting}
call psb_krylov('FCG',a,prec,b,xvec,1d-7,desc,info,itmax=100,itrace=1)
\end{lstlisting}

\end{appendices}


\bibliography{sn-bibliography}


\begin{thebibliography}{33}
\ifx \bisbn   \undefined \def \bisbn  #1{ISBN #1}\fi
\ifx \binits  \undefined \def \binits#1{#1}\fi
\ifx \bauthor  \undefined \def \bauthor#1{#1}\fi
\ifx \batitle  \undefined \def \batitle#1{#1}\fi
\ifx \bjtitle  \undefined \def \bjtitle#1{#1}\fi
\ifx \bvolume  \undefined \def \bvolume#1{\textbf{#1}}\fi
\ifx \byear  \undefined \def \byear#1{#1}\fi
\ifx \bissue  \undefined \def \bissue#1{#1}\fi
\ifx \bfpage  \undefined \def \bfpage#1{#1}\fi
\ifx \blpage  \undefined \def \blpage #1{#1}\fi
\ifx \burl  \undefined \def \burl#1{\textsf{#1}}\fi
\ifx \doiurl  \undefined \def \doiurl#1{\url{https://doi.org/#1}}\fi
\ifx \betal  \undefined \def \betal{\textit{et al.}}\fi
\ifx \binstitute  \undefined \def \binstitute#1{#1}\fi
\ifx \binstitutionaled  \undefined \def \binstitutionaled#1{#1}\fi
\ifx \bctitle  \undefined \def \bctitle#1{#1}\fi
\ifx \beditor  \undefined \def \beditor#1{#1}\fi
\ifx \bpublisher  \undefined \def \bpublisher#1{#1}\fi
\ifx \bbtitle  \undefined \def \bbtitle#1{#1}\fi
\ifx \bedition  \undefined \def \bedition#1{#1}\fi
\ifx \bseriesno  \undefined \def \bseriesno#1{#1}\fi
\ifx \blocation  \undefined \def \blocation#1{#1}\fi
\ifx \bsertitle  \undefined \def \bsertitle#1{#1}\fi
\ifx \bsnm \undefined \def \bsnm#1{#1}\fi
\ifx \bsuffix \undefined \def \bsuffix#1{#1}\fi
\ifx \bparticle \undefined \def \bparticle#1{#1}\fi
\ifx \barticle \undefined \def \barticle#1{#1}\fi
\bibcommenthead
\ifx \bconfdate \undefined \def \bconfdate #1{#1}\fi
\ifx \botherref \undefined \def \botherref #1{#1}\fi
\ifx \url \undefined \def \url#1{\textsf{#1}}\fi
\ifx \bchapter \undefined \def \bchapter#1{#1}\fi
\ifx \bbook \undefined \def \bbook#1{#1}\fi
\ifx \bcomment \undefined \def \bcomment#1{#1}\fi
\ifx \oauthor \undefined \def \oauthor#1{#1}\fi
\ifx \citeauthoryear \undefined \def \citeauthoryear#1{#1}\fi
\ifx \endbibitem  \undefined \def \endbibitem {}\fi
\ifx \bconflocation  \undefined \def \bconflocation#1{#1}\fi
\ifx \arxivurl  \undefined \def \arxivurl#1{\textsf{#1}}\fi
\csname PreBibitemsHook\endcsname

\bibitem[\protect\citeauthoryear{Lottes}{2023}]{LottesNLAA2023}
\begin{barticle}
\bauthor{\bsnm{Lottes}, \binits{J.}}:
\batitle{Optimal polynomial smoothers for multigrid {V}-cycles}.
\bjtitle{Numer. Lin. Alg. with Appl.}
\bvolume{30}(\bissue{6}),
\bfpage{2518}
(\byear{2023})
\doiurl{10.1002/nla.2518}
\end{barticle}
\endbibitem

\bibitem[\protect\citeauthoryear{Vassilevski}{2008}]{MR2427040}
\begin{bbook}
\bauthor{\bsnm{Vassilevski}, \binits{P.S.}}:
\bbtitle{Multilevel Block Factorization Preconditioners},
p. \bfpage{529}.
\bpublisher{Springer},
\blocation{New York}
(\byear{2008}).
\bcomment{Matrix-based analysis and algorithms for solving finite element
  equations}
\end{bbook}
\endbibitem

\bibitem[\protect\citeauthoryear{Xu and Zikatanov}{2017}]{MR3653855}
\begin{barticle}
\bauthor{\bsnm{Xu}, \binits{J.}},
\bauthor{\bsnm{Zikatanov}, \binits{L.}}:
\batitle{Algebraic multigrid methods}.
\bjtitle{Acta Numer.}
\bvolume{26},
\bfpage{591}--\blpage{721}
(\byear{2017})
\doiurl{10.1017/S0962492917000083}
\end{barticle}
\endbibitem

\bibitem[\protect\citeauthoryear{Anzt et~al.}{2020}]{MR4072454}
\begin{barticle}
\bauthor{\bsnm{Anzt}, \binits{H.}},
\bauthor{\bsnm{Boman}, \binits{E.}},
\bauthor{\bsnm{Falgout}, \binits{R.}},
\bauthor{\bsnm{al.}}:
\batitle{Preparing sparse solvers for exascale computing}.
\bjtitle{Philos. Trans. Roy. Soc. A}
\bvolume{378}(\bissue{2166}),
\bfpage{20190053}--\blpage{17}
(\byear{2020})
\doiurl{10.1098/rsta.2019.0053}
\end{barticle}
\endbibitem

\bibitem[\protect\citeauthoryear{Baker et~al.}{2011}]{MR2861652}
\begin{barticle}
\bauthor{\bsnm{Baker}, \binits{A.H.}},
\bauthor{\bsnm{Falgout}, \binits{R.D.}},
\bauthor{\bsnm{Kolev}, \binits{T.V.}},
\bauthor{\bsnm{Yang}, \binits{U.M.}}:
\batitle{Multigrid smoothers for ultraparallel computing}.
\bjtitle{SIAM J. Sci. Comput.}
\bvolume{33}(\bissue{5}),
\bfpage{2864}--\blpage{2887}
(\byear{2011})
\doiurl{10.1137/100798806}
\end{barticle}
\endbibitem

\bibitem[\protect\citeauthoryear{Adams et~al.}{2003}]{MR1985310}
\begin{barticle}
\bauthor{\bsnm{Adams}, \binits{M.}},
\bauthor{\bsnm{Brezina}, \binits{M.}},
\bauthor{\bsnm{Hu}, \binits{J.}},
\bauthor{\bsnm{Tuminaro}, \binits{R.}}:
\batitle{Parallel multigrid smoothing: polynomial versus {G}auss-{S}eidel}.
\bjtitle{J. Comput. Phys.}
\bvolume{188}(\bissue{2}),
\bfpage{593}--\blpage{610}
(\byear{2003})
\doiurl{10.1016/S0021-9991(03)00194-3}
\end{barticle}
\endbibitem

\bibitem[\protect\citeauthoryear{Br\"oker et~al.}{2001}]{MR1885607}
\begin{barticle}
\bauthor{\bsnm{Br\"oker}, \binits{O.}},
\bauthor{\bsnm{Grote}, \binits{M.J.}},
\bauthor{\bsnm{Mayer}, \binits{C.}},
\bauthor{\bsnm{Reusken}, \binits{A.}}:
\batitle{Robust parallel smoothing for multigrid via sparse approximate
  inverses}.
\bjtitle{SIAM J. Sci. Comput.}
\bvolume{23}(\bissue{4}),
\bfpage{1396}--\blpage{1417}
(\byear{2001})
\doiurl{10.1137/S1064827500380623}
\end{barticle}
\endbibitem

\bibitem[\protect\citeauthoryear{Huang et~al.}{2023}]{doi:10.1137/21M1430030}
\begin{barticle}
\bauthor{\bsnm{Huang}, \binits{R.}},
\bauthor{\bsnm{Li}, \binits{R.}},
\bauthor{\bsnm{Xi}, \binits{Y.}}:
\batitle{Learning optimal multigrid smoothers via neural networks}.
\bjtitle{SIAM J. Sci. Comput.}
\bvolume{45}(\bissue{3}),
\bfpage{199}--\blpage{225}
(\byear{2023})
\doiurl{10.1137/21M1430030}
\end{barticle}
\endbibitem

\bibitem[\protect\citeauthoryear{Flaig and Arbenz}{2011}]{FLAIG2011846}
\begin{barticle}
\bauthor{\bsnm{Flaig}, \binits{C.}},
\bauthor{\bsnm{Arbenz}, \binits{P.}}:
\batitle{A scalable memory efficient multigrid solver for micro-finite element
  analyses based on {CT} images}.
\bjtitle{Parallel Computing}
\bvolume{37}(\bissue{12}),
\bfpage{846}--\blpage{854}
(\byear{2011})
\doiurl{10.1016/j.parco.2011.08.001} .
\bcomment{6th International Workshop on Parallel Matrix Algorithms and
  Applications (PMAA'10)}
\end{barticle}
\endbibitem

\bibitem[\protect\citeauthoryear{Baker
  et~al.}{2012}]{10.1007/978-3-642-24025-6_18}
\begin{bchapter}
\bauthor{\bsnm{Baker}, \binits{A.H.}},
\bauthor{\bsnm{Falgout}, \binits{R.D.}},
\bauthor{\bsnm{Gamblin}, \binits{T.}},
\bauthor{\bsnm{Kolev}, \binits{T.V.}},
\bauthor{\bsnm{Schulz}, \binits{M.}},
\bauthor{\bsnm{Yang}, \binits{U.M.}}:
\bctitle{{Scaling Algebraic Multigrid Solvers: On the Road to Exascale}}.
In: \beditor{\bsnm{Bischof}, \binits{C.}},
\beditor{\bsnm{Hegering}, \binits{H.-G.}},
\beditor{\bsnm{Nagel}, \binits{W.E.}},
\beditor{\bsnm{Wittum}, \binits{G.}} (eds.)
\bbtitle{Competence in High Performance Computing 2010},
pp. \bfpage{215}--\blpage{226}.
\bpublisher{Springer},
\blocation{Berlin, Heidelberg}
(\byear{2012})
\end{bchapter}
\endbibitem

\bibitem[\protect\citeauthoryear{MacLachlan et~al.}{2006}]{MR2259545}
\begin{barticle}
\bauthor{\bsnm{MacLachlan}, \binits{S.}},
\bauthor{\bsnm{Manteuffel}, \binits{T.}},
\bauthor{\bsnm{McCormick}, \binits{S.}}:
\batitle{Adaptive reduction-based {AMG}}.
\bjtitle{Numer. Linear Algebra Appl.}
\bvolume{13}(\bissue{8}),
\bfpage{599}--\blpage{620}
(\byear{2006})
\doiurl{10.1002/nla.486}
\end{barticle}
\endbibitem

\bibitem[\protect\citeauthoryear{Gossler and Nabben}{2016}]{MR3498145}
\begin{barticle}
\bauthor{\bsnm{Gossler}, \binits{F.}},
\bauthor{\bsnm{Nabben}, \binits{R.}}:
\batitle{On {AMG} methods with {F}-smoothing based on {C}hebyshev polynomials
  and their relation to {AMG}r}.
\bjtitle{Electron. Trans. Numer. Anal.}
\bvolume{45},
\bfpage{146}--\blpage{159}
(\byear{2016})
\end{barticle}
\endbibitem

\bibitem[\protect\citeauthoryear{Siefert et~al.}{2022}]{9835299}
\begin{bchapter}
\bauthor{\bsnm{Siefert}, \binits{C.}},
\bauthor{\bsnm{Olivier}, \binits{S.L.}},
\bauthor{\bsnm{Voskuilen}, \binits{G.}},
\bauthor{\bsnm{Young}, \binits{J.}}:
\bctitle{Multi{G}rid on {FPGA} {U}sing {D}ata {P}arallel {C}++}.
In: \bbtitle{2022 IEEE International Parallel and Distributed Processing
  Symposium Workshops (IPDPSW)},
pp. \bfpage{907}--\blpage{910}
(\byear{2022}).
\doiurl{10.1109/IPDPSW55747.2022.00147}
\end{bchapter}
\endbibitem

\bibitem[\protect\citeauthoryear{Bassi et~al.}{2011}]{MR2787943}
\begin{barticle}
\bauthor{\bsnm{Bassi}, \binits{F.}},
\bauthor{\bsnm{Ghidoni}, \binits{A.}},
\bauthor{\bsnm{Rebay}, \binits{S.}}:
\batitle{Optimal {R}unge-{K}utta smoothers for the {$p$}-multigrid
  discontinuous {G}alerkin solution of the 1{D} {E}uler equations}.
\bjtitle{J. Comput. Phys.}
\bvolume{230}(\bissue{11}),
\bfpage{4153}--\blpage{4175}
(\byear{2011})
\doiurl{10.1016/j.jcp.2010.04.030}
\end{barticle}
\endbibitem

\bibitem[\protect\citeauthoryear{Bertaccini et~al.}{2017}]{MR3695922}
\begin{barticle}
\bauthor{\bsnm{Bertaccini}, \binits{D.}},
\bauthor{\bsnm{Donatelli}, \binits{M.}},
\bauthor{\bsnm{Durastante}, \binits{F.}},
\bauthor{\bsnm{Serra-Capizzano}, \binits{S.}}:
\batitle{Optimizing a multigrid {R}unge-{K}utta smoother for
  variable-coefficient convection-diffusion equations}.
\bjtitle{Linear Algebra Appl.}
\bvolume{533},
\bfpage{507}--\blpage{535}
(\byear{2017})
\doiurl{10.1016/j.laa.2017.07.036}
\end{barticle}
\endbibitem

\bibitem[\protect\citeauthoryear{Birken}{2012}]{MR2995785}
\begin{barticle}
\bauthor{\bsnm{Birken}, \binits{P.}}:
\batitle{Optimizing {R}unge-{K}utta smoothers for unsteady flow problems}.
\bjtitle{Electron. Trans. Numer. Anal.}
\bvolume{39},
\bfpage{298}--\blpage{312}
(\byear{2012})
\end{barticle}
\endbibitem

\bibitem[\protect\citeauthoryear{Langer}{2014}]{MR3254225}
\begin{barticle}
\bauthor{\bsnm{Langer}, \binits{S.}}:
\batitle{Agglomeration multigrid methods with implicit {R}unge-{K}utta
  smoothers applied to aerodynamic simulations on unstructured grids}.
\bjtitle{J. Comput. Phys.}
\bvolume{277},
\bfpage{72}--\blpage{100}
(\byear{2014})
\doiurl{10.1016/j.jcp.2014.07.050}
\end{barticle}
\endbibitem

\bibitem[\protect\citeauthoryear{Sundar et~al.}{2015}]{MR3367828}
\begin{barticle}
\bauthor{\bsnm{Sundar}, \binits{H.}},
\bauthor{\bsnm{Stadler}, \binits{G.}},
\bauthor{\bsnm{Biros}, \binits{G.}}:
\batitle{Comparison of multigrid algorithms for high-order continuous finite
  element discretizations}.
\bjtitle{Numer. Linear Algebra Appl.}
\bvolume{22}(\bissue{4}),
\bfpage{664}--\blpage{680}
(\byear{2015})
\doiurl{10.1002/nla.1979}
\end{barticle}
\endbibitem

\bibitem[\protect\citeauthoryear{Munch and
  Kronbichler}{2024}]{doi:10.1177/10943420231217221}
\begin{barticle}
\bauthor{\bsnm{Munch}, \binits{P.}},
\bauthor{\bsnm{Kronbichler}, \binits{M.}}:
\batitle{Cache-optimized and low-overhead implementations of additive {S}chwarz
  methods for high-order {FEM} multigrid computations}.
\bjtitle{The International Journal of High Performance Computing Applications}
\bvolume{38}(\bissue{3}),
\bfpage{192}--\blpage{209}
(\byear{2024})
\doiurl{10.1177/10943420231217221}
{\href{https://arxiv.org/abs/https://doi.org/10.1177/10943420231217221}{{https://doi.org/10.1177/10943420231217221}}}
\end{barticle}
\endbibitem

\bibitem[\protect\citeauthoryear{Ghysels et~al.}{2012}]{MR2896807}
\begin{barticle}
\bauthor{\bsnm{Ghysels}, \binits{P.}},
\bauthor{\bsnm{K\l{o}siewicz}, \binits{P.}},
\bauthor{\bsnm{Vanroose}, \binits{W.}}:
\batitle{Improving the arithmetic intensity of multigrid with the help of
  polynomial smoothers}.
\bjtitle{Numer. Linear Algebra Appl.}
\bvolume{19}(\bissue{2}),
\bfpage{253}--\blpage{267}
(\byear{2012})
\doiurl{10.1002/nla.1808}
\end{barticle}
\endbibitem

\bibitem[\protect\citeauthoryear{Kraus et~al.}{2012}]{KVZ2012}
\begin{barticle}
\bauthor{\bsnm{Kraus}, \binits{J.}},
\bauthor{\bsnm{Vassilevski}, \binits{P.}},
\bauthor{\bsnm{Zikatanov}, \binits{L.}}:
\batitle{Polynomial of best uniform approximation to and smoothing in two-level
  methods}.
\bjtitle{Computational Methods in Applied Mathematics}
\bvolume{12}(\bissue{4}),
\bfpage{448}--\blpage{468}
(\byear{2012})
\doiurl{10.2478/cmam-2012-0026}
\end{barticle}
\endbibitem

\bibitem[\protect\citeauthoryear{Van\v{e}k and Brezina}{2013}]{MR3083519}
\begin{barticle}
\bauthor{\bsnm{Van\v{e}k}, \binits{P.}},
\bauthor{\bsnm{Brezina}, \binits{M.}}:
\batitle{Nearly optimal convergence result for multigrid with aggressive
  coarsening and polynomial smoothing}.
\bjtitle{Appl. Math.}
\bvolume{58}(\bissue{4}),
\bfpage{369}--\blpage{388}
(\byear{2013})
\doiurl{10.1007/s10492-013-0018-2}
\end{barticle}
\endbibitem

\bibitem[\protect\citeauthoryear{Brannick et~al.}{2015}]{MR3395379}
\begin{barticle}
\bauthor{\bsnm{Brannick}, \binits{J.}},
\bauthor{\bsnm{Hu}, \binits{X.}},
\bauthor{\bsnm{Rodrigo}, \binits{C.}},
\bauthor{\bsnm{Zikatanov}, \binits{L.}}:
\batitle{Local {F}ourier analysis of multigrid methods with polynomial
  smoothers and aggressive coarsening}.
\bjtitle{Numer. Math. Theory Methods Appl.}
\bvolume{8}(\bissue{1}),
\bfpage{1}--\blpage{21}
(\byear{2015})
\doiurl{10.4208/nmtma.2015.w01si}
\end{barticle}
\endbibitem

\bibitem[\protect\citeauthoryear{D'Ambra et~al.}{2023}]{softimpact}
\begin{botherref}
\oauthor{\bsnm{D'Ambra}, \binits{P.}},
\oauthor{\bsnm{Durastante}, \binits{F.}},
\oauthor{\bsnm{Filippone}, \binits{S.}}:
Parallel sparse computation toolkit.
Software Impacts
\textbf{15}
(2023)
\doiurl{10.1016/j.simpa.2022.100463}
\end{botherref}
\endbibitem

\bibitem[\protect\citeauthoryear{D'Ambra et~al.}{2021}]{MR4331965}
\begin{barticle}
\bauthor{\bsnm{D'Ambra}, \binits{P.}},
\bauthor{\bsnm{Durastante}, \binits{F.}},
\bauthor{\bsnm{Filippone}, \binits{S.}}:
\batitle{A{MG} preconditioners for linear solvers towards extreme scale}.
\bjtitle{SIAM J. Sci. Comput.}
\bvolume{43}(\bissue{5}),
\bfpage{679}--\blpage{703}
(\byear{2021})
\doiurl{10.1137/20M134914X}
\end{barticle}
\endbibitem

\bibitem[\protect\citeauthoryear{McCormick}{1985}]{MR0795945}
\begin{barticle}
\bauthor{\bsnm{McCormick}, \binits{S.F.}}:
\batitle{Multigrid methods for variational problems: general theory for the
  {$V$}-cycle}.
\bjtitle{SIAM J. Numer. Anal.}
\bvolume{22}(\bissue{4}),
\bfpage{634}--\blpage{643}
(\byear{1985})
\doiurl{10.1137/0722039}
\end{barticle}
\endbibitem

\bibitem[\protect\citeauthoryear{Hackbusch}{1982}]{MR0685774}
\begin{bchapter}
\bauthor{\bsnm{Hackbusch}, \binits{W.}}:
\bctitle{Multigrid convergence theory}.
In: \bbtitle{Multigrid Methods ({C}ologne, 1981)}.
\bsertitle{Lecture Notes in Math.},
vol. \bseriesno{960},
pp. \bfpage{177}--\blpage{219}.
\bpublisher{Springer},
\blocation{Berlin-New York}
(\byear{1982})
\end{bchapter}
\endbibitem

\bibitem[\protect\citeauthoryear{Golub and Varga}{1961}]{MR0145679}
\begin{barticle}
\bauthor{\bsnm{Golub}, \binits{G.H.}},
\bauthor{\bsnm{Varga}, \binits{R.S.}}:
\batitle{Chebyshev semi-iterative methods, successive over-relaxation iterative
  methods, and second order {R}ichardson iterative methods. {II}}.
\bjtitle{Numer. Math.}
\bvolume{3},
\bfpage{157}--\blpage{168}
(\byear{1961})
\doiurl{10.1007/BF01386014}
\end{barticle}
\endbibitem

\bibitem[\protect\citeauthoryear{Mason and Handscomb}{2003}]{MR1937591}
\begin{bbook}
\bauthor{\bsnm{Mason}, \binits{J.C.}},
\bauthor{\bsnm{Handscomb}, \binits{D.C.}}:
\bbtitle{Chebyshev Polynomials},
p. \bfpage{341}.
\bpublisher{Chapman \& Hall/CRC},
\blocation{Boca Raton, FL}
(\byear{2003})
\end{bbook}
\endbibitem

\bibitem[\protect\citeauthoryear{Forsythe et~al.}{1977}]{MR0458783}
\begin{bbook}
\bauthor{\bsnm{Forsythe}, \binits{G.E.}},
\bauthor{\bsnm{Malcolm}, \binits{M.A.}},
\bauthor{\bsnm{Moler}, \binits{C.B.}}:
\bbtitle{Computer Methods for Mathematical Computations}.
\bsertitle{Prentice-Hall Series in Automatic Computation},
p. \bfpage{259}.
\bpublisher{Prentice-Hall, Inc.},
\blocation{Englewood Cliffs, NJ}
(\byear{1977})
\end{bbook}
\endbibitem

\bibitem[\protect\citeauthoryear{Van\v{e}k et~al.}{1996}]{MR1393006}
\begin{botherref}
\oauthor{\bsnm{Van\v{e}k}, \binits{P.}},
\oauthor{\bsnm{Mandel}, \binits{J.}},
\oauthor{\bsnm{Brezina}, \binits{M.}}:
Algebraic multigrid by smoothed aggregation for second and fourth order
  elliptic problems.
vol. 56,
pp. 179--196
(1996).
\doiurl{10.1007/BF02238511} .
International GAMM-Workshop on Multi-level Methods (Meisdorf, 1994)
\end{botherref}
\endbibitem

\bibitem[\protect\citeauthoryear{D'Ambra et~al.}{2018}]{MR3865827}
\begin{barticle}
\bauthor{\bsnm{D'Ambra}, \binits{P.}},
\bauthor{\bsnm{Filippone}, \binits{S.}},
\bauthor{\bsnm{Vassilevski}, \binits{P.S.}}:
\batitle{Boot{CM}atch: a software package for bootstrap {AMG} based on graph
  weighted matching}.
\bjtitle{ACM Trans. Math. Software}
\bvolume{44}(\bissue{4}),
\bfpage{39}--\blpage{25}
(\byear{2018})
\doiurl{10.1145/3190647}
\end{barticle}
\endbibitem

\bibitem[\protect\citeauthoryear{Notay and Napov}{2015}]{NOTAY2015237}
\begin{barticle}
\bauthor{\bsnm{Notay}, \binits{Y.}},
\bauthor{\bsnm{Napov}, \binits{A.}}:
\batitle{A massively parallel solver for discrete {P}oisson-like problems}.
\bjtitle{Journal of Computational Physics}
\bvolume{281},
\bfpage{237}--\blpage{250}
(\byear{2015})
\doiurl{10.1016/j.jcp.2014.10.043}
\end{barticle}
\endbibitem

\end{thebibliography}

\end{document}